\documentclass[12pt, a4paper]{amsart}

\usepackage{tikz-cd}\usetikzlibrary{babel} 
\usepackage{nicematrix,tikz}
\usepackage{amsmath}
\usepackage{amssymb}
\usepackage{bm}
\usepackage{amscd} 
\usepackage{enumitem}
\usepackage{mathtools} 
\usepackage{cite}
\usepackage{hyperref} 
\usepackage[initials,msc-links,backrefs]{amsrefs}

\usepackage{aliascnt}

\usepackage[OT2, T1]{fontenc}

\usepackage{cleveref}

\theoremstyle{plain}
\newtheorem{theorem}[equation]{Theorem}
\newtheorem{lemma}[equation]{Lemma}
\newtheorem{corollary}[equation]{Corollary}
\newtheorem{proposition}[equation]{Proposition}

\theoremstyle{definition}
\newtheorem{definition}[equation]{Definition}
\newtheorem{remark}[equation]{Remark}
\newtheorem{example}[equation]{Example}

\providecommand{\ignore}[1]{}

\providecommand{\R}{\mathbb{R}}
\providecommand{\Q}{\mathbb{Q}}
\providecommand{\Z}{\mathbb{Z}}

\providecommand{\A}{\mathbb{A}}
\providecommand{\C}{\mathbb{C}}

\DeclareMathOperator{\im}{im}

\DeclareMathOperator{\Aut}{Aut}

\DeclareMathOperator{\Ad}{Ad}
\DeclareMathOperator{\Sym}{Sym}

\DeclareMathOperator{\GL}{GL}
\DeclareMathOperator{\SL}{SL}
\DeclareMathOperator{\PSL}{PSL}
\DeclareMathOperator{\SU}{SU}
\DeclareMathOperator{\PSU}{PSU}
\DeclareMathOperator{\Spin}{Spin}
\DeclareMathOperator{\SO}{SO}
\DeclareMathOperator{\PSO}{PSO}
\DeclareMathOperator{\Sp}{Sp}
\DeclareMathOperator{\PSp}{PSp}
\DeclareMathOperator{\Gal}{Gal}
\DeclareMathOperator{\Br}{Br}

\usepackage[skip=6pt]{caption}
\usepackage{multirow}

\title{Congruence rigidity of algebraic groups}

\author[A. Baumann]{Adrian Baumann}
\author[H. Kammeyer]{Holger Kammeyer}
 
\address{Heinrich Heine University D{\"u}sseldorf, Faculty of Mathematics and Natural Sciences, Mathematical Institute, Germany}
\email{adrian.baumann@hhu.de}
\email{holger.kammeyer@hhu.de}
 
\subjclass[2010]{11E72, 20E18, 22E40}
\keywords{arithmetic groups, profinite rigidity, Galois cohomology}

\begin{document}

\begin{abstract}
We identify the simple algebraic groups over number fields that are, in a suitable sense, determined by their finite adele points.  Assuming CSP and Grothendieck rigidity, our results essentially characterize higher rank arithmetic groups that are profinitely solitary: the profinite commensurability class determines the commensurability class among finitely generated residually finite groups.  This generalizes previous work of the second author with R.\,Spitler from split groups to arbitrary groups.
\end{abstract}

\maketitle
\tableofcontents

\section{Introduction}

Let \(\mathbf{G}\) be a simply connected absolutely almost simple linear algebraic group defined over a number field \(k\) (a \emph{\(k\)-group} for short).  The purpose of this article is to describe explicitly under what conditions \(\mathbf{G}\) enjoys the following rigidity property.

\begin{definition} \label{def:congruence-rigid}
A \(k\)-group \(\mathbf{G}\) is called \emph{congruence rigid} if every \(l\)-group~\(\mathbf{H}\) that is locally isomorphic to \(\mathbf{G}\) is also globally isomorphic to \(\mathbf{G}\).
\end{definition}

Here \(\mathbf{G}\) is \emph{locally isomorphic} to \(\mathbf{H}\) if there exists an \mbox{\(\A^f_\Q\)-isomorphism} \(\Phi \colon \A^f_l \rightarrow \A^f_k\) such that \(\mathbf{G} \cong_\Phi \mathbf{H}\).  We say that \(\mathbf{G}\) is \emph{globally isomorphic} to \(\mathbf{H}\) if there is an isomorphism \(\phi \colon l \rightarrow k\) such that \(\mathbf{G} \cong_\phi \mathbf{H}\).  The notation \(\A^f_k\) denotes the \emph{finite adele ring}: the \(k\)-subalgebra of \(\prod_{v \nmid \infty} k_v\) of elements with almost all coordinates in the ring of integers \(\mathcal{O}_v \subset k_v\).

\subsection{Motivation} Our interest in this concept stems from its implications for two rigidity properties of arithmetic subgroups of \(\mathbf{G}\): \emph{congruence rigidity} and \emph{profinite solitude}.  To explain this, recall that we obtain an \emph{arithmetic subgroup} \(\mathbf{G}(\mathcal{O}_k)\) by choosing a model of the \(k\)-group scheme \(\mathbf{G}\) over the ring of integers \(\mathcal{O}_k \subset k\).  The \emph{congruence topology} on \(\mathbf{G}(\mathcal{O}_k)\) is the initial topology for all maps \(\mathbf{G}(\mathcal{O}_k) \rightarrow \mathbf{G}(\mathcal{O}_k / \mathfrak{a})\) with nonzero \(\mathfrak{a} \trianglelefteq \mathcal{O}_k\).  The corresponding uniform structure on \(\mathbf{G}(\mathcal{O}_k)\) yields the \emph{congruence completion} \(\overline{\mathbf{G}(\mathcal{O}_k)}\).  Using the notation ``\(\approx\)'' for abstractly commensurable groups, it is natural to define that \(\mathbf{G}(\mathcal{O}_k)\) is \emph{congruence rigid} if for every \(l\)-group \(\mathbf{H}\), we have that \(\overline{\mathbf{G}(\mathcal{O}_k)} \approx \overline{\mathbf{H}(\mathcal{O}_l)}\) implies \(\mathbf{G}(\mathcal{O}_k) \approx \mathbf{H}(\mathcal{O}_l)\).

\begin{theorem} \label{thm:cong-cong}
Among higher rank groups, the \(k\)-group \(\mathbf{G}\) is congruence rigid if and only if the arithmetic subgroup \(\mathbf{G}(\mathcal{O}_k)\) is congruence rigid.
\end{theorem}

This theorem explains the terminology in Definition~\ref{def:congruence-rigid}.  \emph{Higher rank} means \(\operatorname{rank}_\infty \mathbf{G} \coloneqq \sum_{v \mid \infty} \operatorname{rank}_{k_v} \mathbf{G} \ge 2\).  We recall in Section~\ref{section:cong-solitude} that Theorem~\ref{thm:cong-cong} is clear from \emph{Margulis superrigidity} and \emph{adelic superrigidity}.

\medskip
A conjecture of Serre says that higher rank groups have the \emph{congruence subgroup property (CSP)}, meaning \(\overline{\mathbf{G}(\mathcal{O}_k)} \approx \widehat{\mathbf{G}(\mathcal{O}_k)}\) holds for the \emph{profinite completion} \(\widehat{\mathbf{G}(\mathcal{O}_k)}\). We stress, however, that Theorem~\ref{thm:cong-cong} is not contingent on CSP.  On the other hand, if we do assume CSP, then congruence rigidity becomes the decisive condition for \emph{profinite solitude}, the second rigidity property of our interest, as we explain next.

\medskip
Recall that in general, a finitely generated residually finite group \(\Gamma\) is called \emph{(absolutely) profinitely solitary} if for every other such group \(\Delta\), we have that \(\widehat{\Gamma} \approx \widehat{\Delta}\) implies \(\Gamma \approx \Delta\).  Assuming CSP, Theorem~\ref{thm:cong-cong} thus reduces the question of profinite solitude for higher rank arithmetic groups among themselves to the question of congruence rigidity for the ambient algebraic group.  We want to leverage this result to absolute profinite solitude.  One obstacle is the longstanding problem whether arithmetic groups with CSP are \emph{Grothendieck rigid}~\cite{Platonov-Tavgen:grothendieck}*{Problem~2}.

\begin{definition}
  A group \(\Gamma\) is \emph{Grothendieck rigid} if no proper inclusion of a finitely generated subgroup \(H \lneq \Gamma\) induces an isomorphism \(\widehat{H} \cong \widehat{\Gamma}\).
\end{definition}

Though the problem was formulated decades ago, it seems that not for a single arithmetic group with CSP has it been determined whether it is Grothendieck rigid.  It is moreover not known if Grothendieck rigidity passes to finite index subgroups~\cite{Jaikin-Lubotzky:remarks}*{Question~4}.  Let us therefore say that a group is \emph{strongly Grothendieck rigid} if all finite index subgroups are Grothendieck rigid.  Note, however, that Grothendieck rigidity passes to finite index overgroups~\cite{Sun:3-manifold-groups}*{Lemma~2.1 and below} so that strong Grothendieck rigidity is a commensurability invariant.  On the bright side, we will explain that thanks to the work of R.\,Spitler~\cite{Spitler:profinite}, strong Grothendieck rigidity is the only impediment to promote congruence rigidity to absolute profinite solitude.

\begin{theorem} \label{thm:congruence-rigid-implies-solitary}
  Let \(\mathbf{G}\) be a congruence rigid \(k\)-group of higher rank with CSP and let \(\Gamma \le \mathbf{G}\) be a strongly Grothendieck rigid arithmetic \mbox{subgroup}.  Then \(\Gamma\) is profinitely solitary.
\end{theorem}

Next, we discuss that congruence rigidity of \(\mathbf{G}\) is essentially also necessary for the profinite solitude of \(\Gamma \le \mathbf{G}\).  Firstly, we clearly have the following.

\begin{theorem} \label{thm:not-solitary}
  Let \(\mathbf{G}\) be a \(k\)-group of higher rank with CSP that is not congruence rigid among higher rank groups with CSP.  Then no arithmetic subgroup \(\Gamma \le \mathbf{G}\) is profinitely solitary.
\end{theorem}

Assuming Serre's conjecture, it thus remains to consider the case that a given higher rank \(k\)-group \(\mathbf{G}\) is congruence rigid among higher rank groups, but there exists an \(l\)-group \(\mathbf{H}\) locally isomorphic to \(\mathbf{G}\) with \(\operatorname{rank}_\infty \mathbf{H} < 2\).  Clearly, if \(\operatorname{rank}_\infty \mathbf{H} = 0\), then \(\mathbf{H}(\mathcal{O}_l)\) is finite so that such an \(\mathbf{H}\) does not prevent \(\mathbf{G}(\mathcal{O}_k)\) from being profinitely solitary.  So the only \emph{remaining case} is that there exists such an \(\mathbf{H}\) with \(\operatorname{rank}_\infty \mathbf{H} = 1\).  This case will be discussed in Section~\ref{section:cong-solitude} below.  At this point, let us just say that the remaining case is rare: it cannot occur if \(\mathbf{G}\) has type \({}^1 A_n\) for \(n \ge 2\), \({}^1 D_{2n}\) for \(n \ge 2\), \({}^3 D_4\), \(E_6\), \(E_7\), \(E_8\), or \(G_2\) and in the other types, it can only occur if the absolute rank of \(\mathbf{G}\) is small.  Typically, a strongly Grothendieck rigid arithmetic group \(\mathbf{G}(\mathcal{O}_k)\) in the remaining case will still be profinitely solitary.  However, in type \(C_n\) and \(F_4\), the question depends on the well-known open problem whether lattices in \(\operatorname{Sp}(n,1)\) and \(F_{4(-20)}\) have CSP.

\subsection{Approach and methods} With Theorems~\ref{thm:congruence-rigid-implies-solitary} and~\ref{thm:not-solitary} as motivation in mind, let us now explain how to determine whether a given \(k\)-group \(\mathbf{G}\) is congruence rigid.  As a first observation, it already follows from the work of Linowitz--McReynolds--Miller~\cite{Linowitz-et-al:locally-equivalent}, that it is essentially necessary to assume that \(k\) is \emph{locally determined}, meaning every number field \(l\) with \(\A^f_k \cong \A^f_l\) is isomorphic to \(k\).  Note that all number fields of degree six or less over \(\Q\) are locally determined~\cite{Perlis:equation}*{Theorem~3}.  Assuming \(k\) is locally determined, congruence rigidity of \(\mathbf{G}\) asserts that every \(k\)-group \(\mathbf{H}\) with \(\mathbf{G} \cong_\Phi \mathbf{H}\) for some \(\Phi \in \Aut(\A^f_k)\) satisfies \(\mathbf{G} \cong_\phi \mathbf{H}\) for some \(\phi \in \Aut(k)\).  We can express this condition in Galois cohomology as follows.  Let \(\mathbf{G_0}\) be the unique \(\Q\)-split group of the same Cartan--Killing type as \(\mathbf{G}\).  Then \(\mathbf{G}\) corresponds to a unique class in \(H^1(k, \Aut \mathbf{G_0})\) and congruence rigidity of \(\mathbf{G}\) translates immediately to the statement that the fiber of the class of \(\mathbf{G}\) under the map
\begin{equation} \label{eq:goverline} \overline{g} \colon H^1(k, \Aut \mathbf{G_0})\ / \Aut(k) \longrightarrow \prod_{v \nmid \infty} H^1(k_v, \Aut \mathbf{G_0})\ \big/ \Aut(\A_k^f) \end{equation}
is a singleton.  Here \(\Aut(k)\) acts functorially on \(H^1(k, \Aut \mathbf{G_0})\).  Similarly, an automorphism of \(\A^f_k\) splits into a family of isomorphisms \(k_v \rightarrow k_{v'}\) for a bijection \(v \mapsto v'\) of the finite places of \(k\) so that \(\Aut(\A^f_k)\) acts functorially on \(\prod_{v \nmid \infty} H^1(k_v, \Aut \mathbf{G_0})\).  The diagonal map
\[ g \colon H^1(k, \Aut \mathbf{G_0}) \longrightarrow \prod_{v \nmid \infty} H^1(k_v, \Aut \mathbf{G_0}) \]
is equivariant with respect to the natural inclusion \(\Aut(k) \le \Aut(\A^f_k)\) which gives the orbit map~\(\overline{g}\).

As the functorial action on both sides is base point preserving, we see that \(\ker \overline{g}\) is trivial if and only if \(\ker g\) is trivial.  The complete description of when the latter happens, together with the corresponding conclusions on profinite solitude of arithmetic subgroups of \emph{split} groups was given in joint work of the second author with R.\,Spitler~\cite{Kammeyer-Spitler:chevalley}.  For context, let us briefly recall the strategy we employed there.

In this case, \(\mathbf{G} = \mathbf{G_0}\) and we have the split short exact sequence
\[ \begin{tikzcd} 1 \ar[r] & \Ad \mathbf{G_0} \ar[r, "i"] & \Aut \mathbf{G_0} \ar[r, "p"] & \Sym \Delta \ar[l, dashrightarrow, bend left=15, "s"] \ar[r] & 1 \end{tikzcd} \]
where \(\Ad \mathbf{G_0}\) is the adjoint form of inner automorphisms and \(\Sym \Delta\) is the constant group scheme of Dynkin diagram symmetries.  We obtain a corresponding split short exact sequence in first Galois cohomology.  Since a number theoretic consideration gives that the diagonal map
\[ h \colon H^1(k, \Sym \Delta) \longrightarrow \prod_{v \nmid \infty} H^1(k_v, \Sym \Delta) \]
has always trivial kernel, a five lemma type argument shows that \(\ker g\) is trivial if and only if
\[ f \colon H^1(k, \Ad \mathbf{G_0}) \longrightarrow \prod_{v \nmid \infty} H^1(k_v, \Ad \mathbf{G_0}) \]
has trivial kernel.  In turn, thanks to the short exact sequence
\[ 1 \longrightarrow Z(\mathbf{G_0}) \longrightarrow \mathbf{G_0} \longrightarrow \Ad \mathbf{G_0} \longrightarrow 1, \]
and the Hasse principle for simply connected groups, the question if \(\ker f\) is trivial can essentially be traced back to the question if the map
\[ d \colon H^2(k, Z(\mathbf{G_0})) \longrightarrow \bigoplus_{v \nmid \infty} H^2(k_v, Z(\mathbf{G_0})) \]
is injective.  This problem, finally, can be solved by means of the Albert--Brauer--Hasse--Noether theorem from global class field theory.

\medskip
For the general case of a possibly non-split \(k\)-group \(\mathbf{G}\), we have to find out whether other fibers of \(\overline{g}\) are singletons.  This is considerably more intricate as five more layers of difficulty arise in the above program:

\begin{enumerate}[label=(\roman*)]
\item \label{item:map-h} The actions of \(\Aut(k)\) and \(\Aut(\A^f_k)\) have to be taken into account throughout.  In particular, the map
  \[ \qquad \overline{h} \colon H^1(k, \Sym \Delta) / \Aut(k) \longrightarrow \prod_{v \nmid \infty} H^1(k_v, \Sym \Delta) / \Aut(\A^f_k) \]
  is in general, even outside type \(D_4\), no longer injective.

\item In order to make the passage from the automorphism group to the adjoint form, we have to replace \(\mathbf{G_0}\) with the unique \(k\)-quasi-split group \(\mathbf{G_1}\) of which \(\mathbf{G}\) is an inner \(k\)-twist.

\item While the map \(H^1(k, \Ad \mathbf{G_1}) \rightarrow H^1(k, \Aut \mathbf{G_1})\) still has trivial kernel, this does not mean it is injective.  In fact, the fibers of that map are the orbits of the \(\Sym \Delta (k)\)-action on \(H^1(k, \Ad \mathbf{G_1})\).  Correspondingly, we have to find the singleton fibers of the double quotient map
\begin{align*}
  \qquad \overline{f_1}\colon & \Sym \Delta(k)
    \setminus H^1(k, \Ad \mathbf{G_1}) / \Aut(k)_{\gamma_1} \\
  & \longrightarrow \prod_{v \nmid \infty} \left(\Sym \Delta (k_v)
    \setminus H^1(k_v, \Ad \mathbf{G_1})\right)\
    \big/ \Aut(\A_k^f)_{\gamma_1'}.
\end{align*}
Here, \(\gamma_1\) and \(\gamma_1'\) denote the classes determined by \(\mathbf{G}\) in the domain and codomain of \(h\) and we mod out their stabilizer subgroups from the right.

\item The center of a non-split quasi-split group \(\mathbf{G_1}\) is typically a norm one group of the form \(\mathbf{R}^{(1)}_{l/k}(\mu_n)\) for a quadratic field extension \(l/k\).  After passing from \(H^1(k, \Ad \mathbf{G_1})\) to \(H^2(k, Z(\mathbf{G_1}))\), we thus cannot apply the Albert--Hasse--Brauer--Noether theorem anymore.  Instead, we have to rely on the more general Poitou--Tate duality theorem for general finite abelian Galois modules to detect singleton fibers of the map
\begin{align*}
  \qquad \overline{d_1}\colon & \Sym \Delta(k)
    \setminus H^2(k, Z(\mathbf{G_1})) / \Aut(k)_{\gamma_1} \\
  & \longrightarrow \bigoplus_{v \nmid \infty} \left(\Sym \Delta (k_v)
    \setminus H^2(k_v, Z(\mathbf{G_1}))\right)\
    \big/ \Aut(\A_k^f)_{\gamma_1'}.
\end{align*}

  \item To identify the singleton fibers of \(\overline{f_1}\), we moreover need to understand the image of \(H^1(k_v, \mathbf{G_1}) \rightarrow H^1(k_v, \Ad \mathbf{G_1})\) at real places \(v\) of \(k\).  While this map, again, has trivial kernel, it is in general not injective.
\end{enumerate}

All these challenges will be addressed in this article, culminating in Theorems~\ref{thm:cong-rigid-no-symmetries}, \ref{thm:sym-rigidity-no-reals}, \ref{thm:cong-rigid-A-n}, \ref{thm:cong-rigid-D-n}, and \ref{thm:cong-rigid-E-6} below.  Outside type \(D_4\) and assuming \(k\) is locally determined, these theorems give the full description of when \(\mathbf{G}\) is congruence rigid.  It comes as no surprise that due to the above complications, the theorems become lengthy to state so that we refrain from reproducing them in this introduction.  Instead, we give a list of properties of \(k\) and \(\mathbf{G}\) that, depending on the Cartan--Killing type, impede congruence rigidity of \(\mathbf{G}\).  For further illustration, we then explain these properties in the special case of type \(A_n\).  Afterwards, we conclude with some interesting special cases of our main results.

\subsection{Criteria} It turns out that there are six different reasons that may prevent a \(k\)-group \(\mathbf{G}\) from being congruence rigid.  Four of them concern the field \(k\) only.
\begin{enumerate}[label=(\Roman*)]
\item \label{item:locally-determined} The field \(k\) is not locally determined.
\item \label{item:quad-ext-global} The field \(k\) has \(\Q\)-isomorphic quadratic extensions \(l_1/k\) and \(l_2/k\) but no isomorphism \(l_1 \cong l_2\) restricts to an automorphism of \(k\).
\item \label{item:quad-ext-local}The field \(k\) has quadratic extensions \(l_1/k\) and \(l_2/k\) with an isomorphism \(\A^f_{l_1} \cong \A^f_{l_2}\) that restricts to an automorphism of \(\A^f_k\).
\item \label{item:real-places} The field \(k\) has too many real places.
\item \label{item:uniform} The group \(\mathbf{G}\) reduces to non-isomorphic groups at too many pairs of finite places \(v, w\) with \(k_v \cong k_w\).
\item \label{item:inner-twins} The group \(\mathbf{G}\) has an \emph{inner twin} at too many finite places \(v\): a nontrivial inner \(k_v\)-twist that is \(k_v\)-isomorphic to \(\mathbf{G}\) (Definition~\ref{def:inner-twin}).
\end{enumerate}

The precise meaning of ``too many'' depends in each case on the Cartan--Killing type of \(\mathbf{G}\).  The essence of our main results is that in each Cartan--Killing type except \(D_4\), a subset of these properties, each with appropriate refinements, is necessary and sufficient for congruence rigidity.  We decided to omit the triality type \(D_4\) from our discussion, because already examining \(\overline{h}\) becomes an unwieldy task if \(\Sym \Delta \cong S_3\).

\subsection{Illustration in type \texorpdfstring{\(A_n\)}{A}} Let us now discuss the above properties in the special case that \(\mathbf{G}\) has type \(A_{n-1}\).  As was mentioned, the result of Linowitz--McReynolds--Miller~\cite{Linowitz-et-al:locally-equivalent}*{Theorem~6.2} shows that \(k\) has to be locally determined if \(\mathbf{G}\) has inner type.  The simplest instance of this is that for non-isomorphic locally isomorphic fields \(k\) and \(l\) (e.g.\,\(k = \Q(\sqrt[8]{7})\) and \(l = \Q(\sqrt[8]{112})\), as provided by Komatsu~\cite{Komatsu:adele-rings}), we have that \(\mathbf{SL_n}\), once considered as \(k\)-group and once as \(l\)-group, are locally isomorphic but not globally so.  This is what necessitates~\ref{item:locally-determined}.

On the other hand, let \(\mathbf{G}\) and \(\mathbf{H}\) be the quasi-split \(k\)-groups of outer type determined by two quadratic extensions \(l_1/k\) and \(l_2/k\) as in \ref{item:quad-ext-global} or \ref{item:quad-ext-local}, meaning \(\mathbf{G} = \mathbf{SU_n}(l_1, h_1)\) and \(\mathbf{H} = \mathbf{SU_n}(l_2, h_2)\) where \(h_i\) is a non-degenerate hermitian form of maximal index with respect to the nontrivial Galois transformation of \(l_i/k\).  Then \(\mathbf{G}\) and \(\mathbf{H}\) are locally but not globally isomorphic.  More generally, we obtain the same phenomenon starting with any outer form \(\mathbf{G}\) whose outer type is determined by \(l_1/k\) (Theorem~\ref{thm:inner-twists-of-nontrivial-h-bar}).  An example of fields as in \ref{item:quad-ext-global} are certain quadratic extensions \(l_1/k\) and \(l_2/k\), where \(k\) is the degree six field with polynomial
  \[ x^6 - x^5 + 14x^4 - 148x^3 + 1241x^2 - 3647x + 5023, \]
  as we explain in Section~\ref{section:quadratic}.  An example of fields as in \ref{item:quad-ext-local} are Komatsu's examples \(l_1 = \Q(\sqrt[8]{7})\) and \(l_2 = \Q(\sqrt[8]{112})\) viewed as quadratic extensions of \(k = \Q(\sqrt[4]{7})\).

  Property~\ref{item:real-places} is probably the most apparent obstacle to congruence rigidity, simply because the \(\A^f_k\)-type ``does not see'' the type of \(\mathbf{G}\) at real places.  For instance, if \(D\) denotes the unique quaternion algebra over \(\Q(\sqrt{2})\) that ramifies precisely at the two real places, then \(\mathbf{SL_1}(D)\) is locally isomorphic to the \(\Q(\sqrt{2})\)-group \(\mathbf{SL_2}\) but not globally so.  So in type \(A_1\), two real places is already ``too many''.  Note, however, that in type \({}^1 A_{2n}\), the field \(k\) may have arbitrary many real places whereas for higher rank \({}^1 A_{2n+1}\) and the outer types  \({}^2 A_n\), the exact condition, described in Theorem~\ref{thm:cong-rigid-A-n} is somewhat delicate.

  Condition~\ref{item:uniform} owes to the fact that given a \(k\)-group \(\mathbf{G}\), many \(\Aut(\A^f_k)\)-permutations of the infinite tuple of \(k_v\)-localizations of \(\mathbf{G}\) will again be globally realized.  But the group \(\Aut(k)\) is small, so \(\overline{g}\) will map many different \(\Aut(k)\)-orbits to the same \(\Aut(\A^f_k)\)-orbit.  As a simple example, note that the rational prime 31 is completely split in \(\Q(\sqrt[3]{2})\) and let \(D_i\) for \(i = 1,2,3\) be the three quaternion algebras over \(\Q(\sqrt[3]{2})\) which ramify precisely at two of the three primes over 31.  Then all three \(\Q(\sqrt[3]{2})\)-groups \(\mathbf{SL_1}(D_i)\) are locally isomorphic but no two of them are globally isomorphic because \(\Aut(\Q(\sqrt[3]{2}))\) is trivial.  To rule out such behavior, we will introduce the notion of \emph{weak uniformity} in Definition~\ref{def:weakly-uniform-no-sym} which is a generalization of the notion of \emph{local uniformity} introduced for quaternion algebras by Bridson--McReynolds--Reid--Spitler in~\cite{Bridson-et-al:absolute}*{Definition~4.10}.

  Finally, to illustrate why Condition~\ref{item:inner-twins} becomes necessary, consider the degree 3 division algebras \(D_1\) and \(D_2\) over \(\Q\) whose Brauer invariants in \(\Q/\Z\) at the primes 2, 3, 5, and 7 are given in Table~\ref{table:division-alg} 
\begin{table}[htb]
  \begin{tabular}{l|llll}
    & 2   & 3   & 5   & 7   \\ \hline
    \(D_1\) & 1/3 & 2/3 & 1/3 & 2/3 \\
    \(D_2\) & 1/3 & 2/3 & 2/3 & 1/3 \\
  \end{tabular}
  \caption{Brauer invariants of division algebras.}
  \label{table:division-alg}
\end{table}
and are trivial at all other primes including \(\infty\).  Then the two groups \(\mathbf{SL_1}(D_i)\) of type \(A_2\) are locally isomorphic, and in fact, even \(\A^f_k\)-isomorphic.  But they are not globally isomorphic because \(D_1\) and \(D_2\) are not inverses of one another in \(\operatorname{Br}(\Q)\).  To rule out such behavior, we will extend our notion of weak unformity in Definition~\ref{def:weakly-uniform-sym} to incorporate a condition on the action of the Dynkin diagram symmetries whenever such symmetries exist.

\subsection{Results} As mentioned above, outside type \(D_4\) and assuming \(k\) is locally determined, the full description of when \(\mathbf{G}\) is congruence rigid is given in Theorems~\ref{thm:cong-rigid-no-symmetries}, \ref{thm:sym-rigidity-no-reals}, \ref{thm:cong-rigid-A-n}, \ref{thm:cong-rigid-D-n}, and \ref{thm:cong-rigid-E-6} below.  In this section, we present some interesting special cases of our theorems.

\begin{theorem}[Corollary~\ref{cor:cong-rigid-split}] \label{thm:intro-quasi-split-galois}
    Let \(k\) be Galois over \(\Q\) and let \(\mathbf{G}\) be a \(k\)-group.  Assume furthermore that \(\mathbf{G}\) is quasi-split over \(k\) and not of type \(D_4\).  Then \(\mathbf{G}\) is congruence rigid if and only if  one of the following conditions is satisfied:
  \begin{enumerate}[label=(\roman*)]
    \item \(\mathbf{G}\) is of type \(A_{2n}\) for \(n \ge 1\) and splits at all infinite places of \(k\).
    \item \(\mathbf{G}\) is of type \(A_{2n+1}\) for \(n \ge 0\) or \(C_{n \ge 2}\), and either \(k = \Q\) and \(\mathbf{G}\) splits at the real place or \(k\) is totally imaginary.
    \item \(\mathbf{G}\) is of type \(B_{n \geq 3}\), \(D_{n \geq 5}\), \(E_6\), \(E_7\), \(E_8\), \(F_4\), or \(G_2\) and \(k\) is totally imaginary.
  \end{enumerate}
\end{theorem}

In fact, the assumption that \(k/\Q\) is Galois both has the consequence that \(k\) is locally determined~\cite{Klingen:similarities}*{Theorem~III.1.4.b)} and that the map
\[ \overline{h} \colon \ k^\times/(k^\times)^2 \,/\, \Aut(k) \longrightarrow \prod_{v \nmid \infty} k_v^\times / (k_v^\times)^2 \,/\, \Aut(\A_k^f)  \]
that arises for \(\Sym \Delta \cong \Z / 2\) is injective.  We prove the latter in Theorem~\ref{thm:h-bar-injective} by a representation theoretic argument for the occuring Galois groups in Proposition~\ref{prop:almost-conjugate}.  It involves the Mackey restriction formula whose relevance in this context was pointed out to us by OpenAI GPT 5.  Theorem~\ref{thm:intro-quasi-split-galois} extends the main result of the second author with R.\,Spitler~\cite{Kammeyer-Spitler:chevalley} from split groups to quasi-split groups over Galois extensions.

For the special case that \(k = \Q\), we have the following general result where the additional condition appearing in \ref{item:a4n3} will be explained in Section~\ref{section:ade}.

\begin{theorem}[Corollary~\ref{cor:cong-rigid-Q}]
  Let \(\mathbf{G}\) be a \(\Q\)-group not of type \(D_4\).  Then \(\mathbf{G}\) is congruence rigid if and only if  one of the following conditions is satisfied:
    \begin{enumerate}[label=(\roman*)]
  \item \(\mathbf{G}\) is of type \(A_1\).
  \item \(\mathbf{G}\) is of type \({}^1 A_{2n}\) or \({}^1 A_{4n+1}\) for \(n \ge 1\) and weakly uniform.
    \item \(\mathbf{G}\) is of type \({}^2 A_{2n}\) for \(n \ge 1\), is not quasi-split at at most one finite place and splits over \(\R\).
   \item \label{item:a4n3} \(\mathbf{G}\) is of type \({}^1 A_{4n+3}\) for \(n \ge 0\), weakly uniform and no subset of coordinates of \(\omega\) adds up to \(\pm (n+1)\) in \(\Z/(4n+4) \cong H^0(k,Z')^*\).
    \item \(\mathbf{G}\) is of type \({}^2 A_{2n+1}\) for \(n \ge 1\), has an inner twin at at most one finite place, and over \(\R\) the group \(\mathbf{G}\) either becomes inner type or is isomorphic to \(\SU(3,1)\) if \(n=1\).
    \item \(\mathbf{G}\) is of type \(C_n\) for \(n \geq 2\) and over \(\R\) the group \(\mathbf{G}\) becomes isomorphic to \(\Sp(2n, \R)\).
    \item \(\mathbf{G}\) is of type \({}^1 D_{2n}\) for \(n \geq 3\), has an inner twin at exactly one finite place, and over \(\R\) the group \(\mathbf{G}\) becomes isomorphic to \(\Spin^*(4n)\).
    \item \(\mathbf{G}\) is of type \({}^2 D_{2n}\) for \(n \geq 3\), has no inner twins at finite places, and over \(\R\) the group \(\mathbf{G}\) becomes isomorphic to \(\Spin^*(4n)\).
    \item \(\mathbf{G}\) is of type \({}^1 D_5\), has no inner twins at finite places, and over \(\R\) the group \(\mathbf{G}\) becomes isomorphic to \(\Spin(7,3)\).
    \item \(\mathbf{G}\) is of type \({}^2 D_{2n+1}\) for \(n \geq 2\), has an inner twin at at most one finite place, and over \(\R\) the group \(\mathbf{G}\) becomes isomorphic to \(\Spin^*(4n+2)\) or to \(\Spin(7,3)\) if \(n=2\).
  \end{enumerate}
\end{theorem}

In particular, the theorem says the following.
\begin{corollary}
  No \(\Q\)-group of type \(B_n\) for \(n \ge 3\), \({}^1 D_{2n+1}\) for \(n \ge 3\), \(E_6\), \(E_7\), \(E_8\), \(F_4\), or \(G_2\) is congruence rigid. 
\end{corollary}

One rough way to make Condition~\ref{item:real-places} from the above list precise is the following.

\begin{theorem} \label{thm:3-real-places-not-rigid}
  Suppose \(k\) has three or more real places and the \(k\)-group \(\mathbf{G}\) is neither of type \(A_{2n}\) nor \(D_4\).  Then \(\mathbf{G}\) is not congruence rigid.
\end{theorem}

As our final special case, the following theorems say that Condition~\ref{item:inner-twins} quickly spoils congruence rigidity as well.  Recall that Condition~\ref{item:inner-twins} is only relevant in types with Dynkin diagram symmetries so that these results exhaust all cases except type \(D_4\).

\begin{theorem}[Theorem~\ref{thm:many-inner-twins}]
  Let \(\mathbf{G}\) be a \(k\)-group of type \({}^1 A_{n \ge 2}\), \({}^1 D_{n \ge 5}\), or \({}^1 E_6\).  Let \(r\) be the number of finite places \(v\) where \(\mathbf{G}\) has an inner twin.  Let furthermore
  \begin{align*}m =
    \begin{cases}
      2n+1 & \text{if } \mathbf{G} \text{ has type } {}^1 A_{2n}, \\
      n+1 & \text{if } \mathbf{G} \text{ has type } {}^1 A_{2n+1}, \\
      2 & \text{if } \mathbf{G} \text{ has type } {}^1 D_{n \geq 5}, \\
      3 & \text{if } \mathbf{G} \text{ has type } {}^1 E_6.
    \end{cases}
  \end{align*}
  If \([k : \Q] < 2^{\left\lfloor\frac{r-1}{m}\right\rfloor}\), then \(\mathbf{G}\) is not congruence rigid.
\end{theorem}

\begin{theorem}\label{thm:outer-types-cong-rigid}
  Let \(\mathbf{G}\) be a \(k\)-group of type \({}^2 A_{n \geq 2}\), \({}^2 D_{n \geq 5}\), or \({}^2 E_6\).  Suppose there are at least two finite places of \(k\) where \(\mathbf{G}\) has an inner twin.  Then \(\mathbf{G}\) is not congruence rigid.
\end{theorem}

So all in all, one might want to say that congruence rigidity is a fairly rare property among \(k\)-groups that nevertheless occurs throughout all Cartan--Killing types under rigidity assumptions on \(k\) and uniformity assumptions on \(\mathbf{G}\).

\subsection{Relation to previous work}  So far, profinite solitude was investigated in two different set-ups: for lattices in higher rank simple Lie groups~\cites{Kammeyer-Kionke:adelic-superrigidity, Kammeyer:absolute, Kammeyer-Kionke:rigidity} and for Chevalley groups, meaning arithmetic subgroups of split \(k\)-groups~\cite{Kammeyer-Spitler:chevalley}.  By Margulis arithmeticity, the former are precisely the arithmetic subgroups of \(k\)-groups \(\mathbf{G}\) over number fields \(k\) with at most one complex place and such that \(\mathbf{G}\) is anisotropic at all but one infinite place.  Chevalley groups are of the opposite nature: as arithmetic subgroups of globally split groups, they are maximally isotropic at all places.  This work finally treats the general case and as such encompasses and extensively generalizes the results from~\cites{Kammeyer:absolute, Kammeyer-Spitler:chevalley} dealing with absolute profinite solitude.  The results in~\cites{Kammeyer-Kionke:adelic-superrigidity, Kammeyer-Kionke:rigidity} about profinite solitude of lattices in the same Lie group, however, remain of a different nature.

\subsection{Outline} The outline of the article is as follows.  In Section~\ref{section:cong-solitude}, we quickly prove the motivating Theorems~\ref{thm:cong-cong}, \ref{thm:congruence-rigid-implies-solitary}, and \ref{thm:not-solitary}.  We also give some more remarks on what can be inferred on profinite solitude from congruence rigidity if \(\mathbf{G}\) is congruence rigid among higher rank groups but not among groups of rank one.  Section~\ref{section:exact-sequences} gives the general Galois cohomological method to determine the singleton fibers of \(\overline{g}\).  In Section~\ref{section:adelic-automorphisms}, we examine the \(\Aut(\A^f_k)\)-action on the codomain of \(\overline{g}\) and we show that \(\overline{h}\) having trivial fiber at the quasi-split form of \(\mathbf{G}\) is a necessary condition for congruence rigidity.  In Section~\ref{section:quadratic}, we give conditions under which fibers of \(\overline{h}\) are trivial and nontrivial, respectively.  Section~\ref{section:inner-forms} strengthens the result of Section~\ref{section:adelic-automorphisms} by showing that \(\mathbf{G}\) has singleton fiber under \(\overline{g}\) if and only if the corresponding fibers of \(\overline{h}\) and \(\overline{f_1}\) are singletons.  Section~\ref{section:real-galois} shows, in turn, that singleton fibers occur for \(\overline{f_1}\) if and only if they occur for \(\overline{d_1}\) and if at every real place, \(\mathbf{G}\) becomes one of a certain list of real forms that we spell out explicitly.  Section~\ref{section:abceefg} concludes the characterization of congruence rigidity when \(\mathbf{G}\) has no Dynkin diagram symmetries.  For groups with Dynkin diagram symmetries, the notion of weak uniformity has to be adapted.  This will be done in Section~\ref{section:weak-uniformity} and we examine this generalized property in some detail.  With this preparation, the characterization of congruence rigidity for \(\mathbf{G}\) with Dynkin diagram symmetries will be given in Section~\ref{section:ade}.  Finally, Section~\ref{section:special-cases} deduces the special cases that we highlighted above.

\subsection{Acknowledgements} Both authors were supported by the RTG 2240 ``Algebro-geometric methods in algebra, arithmetic and topology'' (DFG 284078965).  We are indebted to Daniel Echtler, Benjamin Klopsch, Otto Overkamp, Fabian Rodatz, and Giada Serafini for many helpful discussions.

\section{Congruence rigidity and profinite solitude}
\label{section:cong-solitude}

In this section, we quickly provide proofs for the motivating theorems in the introduction relating congruence rigidity to profinite solitude.  We then comment on the case that a higher rank \(k\)-group \(\mathbf{G}\) is congruence rigid among higher rank groups but there exists a rank one \(l\)-group \(\mathbf{H}\) locally but not globally isomorphic to \(\mathbf{G}\).

\begin{proof}[Proof of Theorem~\ref{thm:cong-cong}.]
      Let \(\mathbf{G}\) be a higher rank \(k\)-group and let \(\mathbf{H}\) be a higher rank \(l\)-group.  It is an immediate consequence of Margulis superrigidity that \(\mathbf{G}(\mathcal{O}_k) \approx \mathbf{H}(\mathcal{O}_l)\) if and only if \(\mathbf{G}\) is globally isomorphic to \(\mathbf{H}\).  The precise deduction can be found in~\cite{Kammeyer:profinite-commensurability}*{Corollary~16}.  Similarly, it is immediate from adelic superrigidity~\cite{Kammeyer-Kionke:adelic-superrigidity}*{Theorem~3.2} that \(\overline{\mathbf{G}(\mathcal{O}_k)} \approx \overline{\mathbf{H}(\mathcal{O}_l)}\) if and only if \(\mathbf{G}\) is locally isomorphic to \(\mathbf{H}\).  From these two equivalences, the statement is clear.
    \end{proof}

\begin{proof}[Proof of Theorem~\ref{thm:congruence-rigid-implies-solitary}.]
  Let \(\Delta\) be finitely generated and residually finite with \(\widehat{\Gamma} \approx \widehat{\Delta}\).  Then there are finite index subgroups \(\Gamma_0 \le \Gamma\) and \(\Delta_0 \le \Delta\) such that \(\widehat{\Gamma_0} \cong \widehat{\Delta_0}\), hence~\cite{Spitler:profinite}*{Theorem~7.1} provides us with a number field \(l\), an \(l\)-group \(\mathbf{H}\), an arithmetic subgroup \(\Lambda \le \mathbf{H}\), an injection \(g \colon \Delta_0 \rightarrow \Lambda\), and a local isomorphism \(\mathbf{G} \cong_\Phi \mathbf{H}\) over an \(\A^f_\Q\)-isomorphism \(\Phi \colon \A^f_l \rightarrow \A^f_k\).  In fact, the cited theorem due to Spitler assumes the two additional conditions that \(\Gamma\) has no nontrivial homomorphisms to the center of \(\mathbf{G}\) and that every automorphism of \(\mathbf{G}\) is induced by conjugation in some \(\mathbf{GL_n}\).  But the same arguments as in \cite{Kammeyer-Spitler:chevalley}*{Section~4} apply so that we still obtain the above conclusion.

  Since \(\mathbf{G}\) is congruence rigid, we have an isomorphism \(\phi \colon l \rightarrow k\) such that \(\mathbf{G} \cong_\phi \mathbf{H}\).  It follows that \(\Gamma \approx \Lambda\) and that \(\mathbf{H}\) has CSP so that Spitler's theorem gives additionally that either \(g(\Delta_0) = \Lambda\) or \(g(\Delta_0)\) is a proper finitely generated subgroup of \(\Lambda\) whose inclusion induces an isomorphism of profinite completions.  Since \(\Gamma\) is strongly Grothendieck rigid, so is \(\Lambda\).  We thus have \(g(\Delta_0) = \Lambda\), whence \(\Gamma \approx \Delta\).
\end{proof}

\begin{proof}[Proof of Theorem~\ref{thm:not-solitary}.]
By assumption, \(\mathbf{G}\) is not congruence rigid among higher rank groups with CSP.  So there exists a higher rank \(l\)-group \(\mathbf{H}\) with CSP locally but not globally isomorphic to \(\mathbf{G}\).  As in the proof of Theorem~\ref{thm:cong-cong}, this gives that \(\overline{\mathbf{G}(\mathcal{O}_k)} \approx \overline{\mathbf{H}(\mathcal{O}_l)}\) and since both groups have CSP, we also have \(\widehat{\mathbf{G}(\mathcal{O}_k)} \approx \widehat{\mathbf{H}(\mathcal{O}_l)}\). But \(\mathbf{G}(\mathcal{O}_k) \not\approx \mathbf{H}(\mathcal{O}_l)\) because \(\mathbf{G}\) is not globally isomorphic to \(\mathbf{H}\).
\end{proof}

Let us now discuss what we termed the remaining case in the introduction.  The remaining case occurs if \(\mathbf{G}\) is a higher rank \(k\)-group such that there exists an \(l\)-group \(\mathbf{H}\) with \(\operatorname{rank}_\infty \mathbf{H} = 1\) and such that \(\mathbf{G}\) is locally isomorphic to \(\mathbf{H}\).  Let us first collect when this cannot happen.

\begin{proposition}
  Let \(\mathbf{G}\) be a higher rank \(k\)-group.  Then in the following cases there exists no \(l\)-group \(\mathbf{H}\) with \(\operatorname{rank}_\infty \mathbf{H} = 1\) such that \(\mathbf{G}\) is locally isomorphic to \(\mathbf{H}\):
  \begin{enumerate}[label=(\roman*)]
  \item \label{item:two-real-places} \(k\) has two or more complex places.
  \item \label{item:various-types} \(\mathbf{G}\) has type \({}^1 A_{n \ge 2}\), \({}^1 D_{2n}\) for \(n \ge 2\), \({}^3 D_4\), \(E_6\), \(E_7\), \(E_8\), or \(G_2\).    
  \end{enumerate}
  \end{proposition}

  \begin{proof}
    Let \(\mathbf{H}\) be an \(l\)-group locally isomorphic to \(\mathbf{G}\).  If \(k\) has two complex places, then so does \(l\) because number fields with isomorphic adele rings are arithmetically equivalent, hence have the same signature~\cite{Klingen:similarities}*{Theorem~III.1.4.h)}.  Thus \(\operatorname{rank}_\infty \mathbf{H} \ge 2\) which gives~\ref{item:two-real-places}.

    Recall that the complete list of simply-connected absolutely simple \(\R\)-groups is given by
    \[ \operatorname{SU}(n,1) \text{ for } n \ge 1, \ \operatorname{Spin}(n,1) \text{ for } n \ge 3, \ \operatorname{Sp}(n,1) \text{ for } n \ge 2, \ F_{4(-20)}.\]
    The first family is of type \(A_1\) for \(n=1\) and \({}^2 A_n\) for \(n \ge 2\).  The second family is of type \(B_{n/2}\) for even \(n\), of type \({}^1 D_{(n+1)/2}\) for \(n \equiv 1 \text{ mod } 4\), and of type \({}^2 D_{(n+1)/2}\) for \(n \equiv 3 \text{ mod } 4\).  The third family is of type \(C_{n+1}\).  Of course \(F_{4(-20)}\) is of type \(F_4\).  It follows from~\cite{Kammeyer-Serafini:sign}*{Proposition~16} that if \(\mathbf{G}\) is globally of inner type (or of type \({}^3 D_4\)), then \(\mathbf{H}\) must reduce to an inner form of the same Cartan--Killing type at any real place.  Hence if \(\mathbf{G}\) has one of the types listed in~\ref{item:various-types}, then \(\mathbf{H}\) cannot become a rank one group at any real place of \(k\).  This proves~\ref{item:various-types}.
  \end{proof}

  By the proposition, we yet have to discuss the remaining case when \(\mathbf{G}\) has type \(A_1\), \({}^2 A_{n \ge 2}\), \(B_n\), \(D_n\) (not \( {}^1 D_{2n}\)), \(C_n\), or \(F_4\).  For type \(B_n\) and \(D_n\), it turns out that we still obtain the conclusion of Theorem~\ref{thm:congruence-rigid-implies-solitary}.
  
  \begin{theorem}
    Let \(\mathbf{G}\) be a higher rank \(k\)-group of type \(B_{n \ge 2}\) or \(D_{n \ge 5}\) that is congruence rigid among higher rank groups.  Then every strongly Grothendieck rigid arithmetic subgroup \(\Gamma \le \mathbf{G}\) is profinitely solitary.
  \end{theorem}

  \begin{proof}
    Note that \(\mathbf{G}\) has CSP because Serre's conjecture is known in the given types~\cite{Platonov-Rapinchuk:algebraic-groups}*{Theorem~9.23, p.\,568}.  Let \(\Delta\) be a finitely generated residually finite group profinitely commensurable with \(\Gamma\).  As in the proof of Theorem~\ref{thm:congruence-rigid-implies-solitary}, Spitler's theorem gives an \(l\)-group \(\mathbf{H}\) locally isomorphic to \(\mathbf{G}\) such that \(\Delta\) embeds as subgroup of an arithmetic subgroup \(\Lambda \le \mathbf{H}\) and the embedding \(\Delta \le \mathbf{H}\) is Zariski dense by construction.  Since \(\mathbf{G}\) is congruence rigid among higher rank groups, we only have to exclude the possibility that \(\operatorname{rank}_\infty \mathbf{H} = 1\).  To do so, we argue similarly as in~\cite{Kammeyer-Spitler:chevalley}*{Section~4}.  Indeed, if \(\operatorname{rank}_\infty \mathbf{H} = 1\), then \(l\) is totally real and \(\mathbf{H}\) is anisotropic at all but one real place where \(\mathbf{H}\) is isogenous to some \(\operatorname{SO}(m,1)(\mathbb{R})\).  So a finite index subgroup \(\Delta_0\) embeds Zariski densely into \(\operatorname{SO}(m,1)(\mathbb{R})\).  Thus some element in \(\Delta_0\) acts loxodromically on \(\mathbb{H}^m\), hence it generates a geometrically finite infinite cyclic subgroup of \(\mathbf{H}(\mathcal{O}_l) \le \operatorname{SO}(m,1)(\mathbb{R})\) that is a virtual retract.  This shows that \(\Delta_0\) has a finite index subgroup with positive first Betti number.  By profinite invariance of the abelianization, so does a finite index subgroup of \(\mathbf{G}(\mathcal{O}_k)\) contradicting both CSP and Kazhdan's Property~(T).
  \end{proof}
    
We remark that if \(\mathbf{G}\) has type \({}^2 A_{n \ge 2}\) or \(C_{n \ge 3}\), the remaining case can only occur for small values of \(n\) and totally real number fields \(k\) of small degree.  Indeed, the group \(\mathbf{H}\) would have to be isomorphic to \(\operatorname{SU}(n,1)\) or \(\operatorname{Sp}(n,1)\) at one real place and anisotropic at all the others.  But if \(n\) is large and/or if there are many real places, then the description of real Galois cohomology in Table~\ref{table:real-cohomology} below gives that the real forms of \(\mathbf{G}\) occuring at real places of \(k\) can be replaced with other groups of type \(\operatorname{SU}(p,q)\) or \(\operatorname{Sp}(p,q)\) to produce several higher rank groups locally isomorphic to \(\mathbf{G}\).

So the question of profinite solitude in the remaining case stays unclear only in some special forms of type \(A_n\) and \(C_n\) for small \(n\), and in type \(D_4\) and \(F_4\).  Incidentally, these are also the cases in which the incomplete satus of CSP prevents us from saying anything definitive on profinite solitude anyway: CSP is open for \(\mathbf{SL_n}(D)\) for \(k\)-division algebras \(D\), triality forms of type \(D_4\), and \(\Q\)-groups that are \(\R\)-isomorphic to \(\operatorname{Sp}(n,1)\) or \(F_{4(-20)}\). 

\section{Exact sequences in Galois cohomology}
\label{section:exact-sequences}

In this section, we give the general outline of how to identify the singleton fibers of the map \(\overline{g}\) given in~\eqref{eq:goverline}.  Three different exact sequences in Galois cohomology will be of help for this purpose: the ones associated with the short exact sequences~\eqref{eq:ad-aut} and~\eqref{eq:center} below, as well as the Poitou--Tate sequence in~\eqref{eq:poitou-tate} associated with the center of \(\mathbf{G}\).  An important calculation for the case by case considerations in the subsequent chapter will be the explicit determination of the Poitou--Tate sequence for the various occurring centers in Table~\ref{tate-table}.

Let \(\mathbf{G_0}\) be a \(\Q\)-split simply connected absolutely almost simple linear algebraic \(\Q\)-group and let~\(k\) be a locally determined number field.  Our goal is to find the singleton fibers of the map
\[ \overline{g} \colon H^1(k, \Aut \mathbf{G_0})\ / \Aut(k)
  \longrightarrow \prod_{v \nmid \infty} H^1(k_v, \Aut \mathbf{G_0})\
  \big/ \Aut(\A_k^f). \]
  
If \(\mathbf{G_0}\) is of type \(A_n\) for \(n \geq 2\), \(D_n\) for \(n \geq 4\) or \(E_6\), \(\Aut \mathbf{G_0}\) has elements that do not arise as an inner automorphism as a result of exisiting nontrivial Dynkin diagram symmetries.  As announced in the introduction, our reults will assume \(\Sym \Delta\) to be either trivial or isomorphic to \(\Z/2\), which holds in all cases except \(D_4\).  However, until explicitly stated all subsequent statements hold for algebraic groups of arbitrary type.  We have a split short exact sequence~\cite{Platonov-Rapinchuk:algebraic-groups}*{(2.13), p.\,77}
\begin{equation} \label{eq:ad-aut} \begin{tikzcd} 1 \ar[r] & \Ad \mathbf{G_0} \ar[r, "i"] & \Aut \mathbf{G_0} \ar[r, "p"] & \Sym \Delta \ar[l, dashrightarrow, bend left=15, "s"] \ar[r] & 1 \end{tikzcd} \end{equation}
and the splitting is \(\Gal(k)\)-equivariant.  This induces a short exact sequence of cohomology sets
\[1 \rightarrow H^1(K, \Ad \mathbf{G_0}) \xrightarrow{i} H^1(K, \Aut \mathbf{G_0}) \xrightarrow{p} H^1(K, \Sym \Delta) \rightarrow 1\]
where \(K\) can either be \(k\) or one of its local completions at some place \(v\) and \(i\) and \(p\) denote the maps induced by the pushforward along the original \(i\) and \(p\), respectively.  Note that even though the kernel of \(i\) is trivial, the map \(i\) is in general not injective anymore.

Let \(\beta \in H^1(K, \Aut \mathbf{G_0})\) be the class of a cocycle \(b\) corresponding to some arbitrary \(K\)-form of \(\mathbf{G_0}\).  The class of the cocycle \(b_1 = s \circ  p \circ b\) in \(H^1(K, \Aut \mathbf{G_0})\) then classifies the unique \(K\)-quasi-split group \(\mathbf{G_1}\) that is an inner \(K\)-twist of \(\mathbf{G}\).  A quick calculation shows that the twisted Galois module \({}_{b_1}(\Aut \mathbf{G_0})\) is \(\Gal(K)\)-isomorphic to \(\Aut \mathbf{G_1}\).  So right multiplication with \(b_1\) defines a bijection
\[ \tau_{b_1} \colon H^1(K, \Aut \mathbf{G_1}) \xrightarrow{\ \cong \ } H^1(K, \Aut \mathbf{G_0}). \]
Let for the moment \(\widehat{\beta} \in H^1(K, \Aut \mathbf{G_1})\) be the preimage of \(\beta\) under the bijection \(\tau_{b_1}\).  Since twisting is functorial, we have a commutative square
\[ \begin{tikzcd}
    H^1(K, \Aut \mathbf{G_0}) \ar[r, "p"] & H^1(K, \Sym \Delta) \\
    H^1(K, \Aut \mathbf{G_1}) \ar[u, "\tau_{b_1}", "\cong"'] \ar[r, "p_1"] & H^1(K, {}_{p \circ b_1}(\Sym \Delta)) \ar[u, "\tau_{p \circ b_1}", "\cong"'].
  \end{tikzcd}
\]
Applying this twisting process both locally and globally, we obtain the map \(g_1\), which is a twist of \(g\), in the following diagram.
\begin{equation} \begin{tikzcd} \label{diagram:g-1-h-1-square}
    \prod\limits_{v \nmid \infty} H^1(k_v, \Aut \mathbf{G_1}) \ar[r, "P_1"] & \prod\limits_{v \nmid \infty} H^1(k_v, {}_{p \circ b_1}(\Sym \Delta)) \\
    H^1(k, \Aut \mathbf{G_1}) \ar[u, "g_1"] \ar[r, "p_1"] & H^1(k, {}_{p \circ b_1}(\Sym \Delta)) \ar[u, "h_1"]
  \end{tikzcd} \end{equation}
As \(p \circ b = p \circ b_1\), we have \(p_1(\widehat{\beta}) = 1\) which gives
\[\widehat{\beta} \in \im\left(i_1 \colon H^1(K, \Ad \mathbf{G_1}) \rightarrow H^1(K, \Aut \mathbf{G_1})\right) \]
by exactness.  Thus we can extend the diagram in ~\eqref{diagram:g-1-h-1-square} on the left by
\begin{equation} \label{diagram:f-1-g-1-square} \begin{tikzcd}
  \prod\limits_{v \nmid \infty} H^1(k_v, \Ad \mathbf{G_1}) \ar[r, "I_1"]
    & \prod\limits_{v \nmid \infty} H^1(k_v, \Aut \mathbf{G_1}) \\
  H^1(k, \Ad \mathbf{G_1}) \ar[r, "i_1"] \ar[u, "f_1"]
    & H^1(k, \Aut \mathbf{G_1}). \ar[u, "g_1"]
\end{tikzcd} \end{equation}
  From now on, we will write \(\beta\) for both \(\beta\) and \(\widehat{\beta}\) to simplify notation.  The fibers of \(i_1\) are the orbits of the action of \(H^0(K, {}_{p \circ b_1}(\Sym \Delta))\) on \(H^1(k, \Ad \mathbf{G_1})\) by \cite{Serre:galois-cohomology}*{I.5.5~Cor.\,2}.  If \(\Sym \Delta \cong \Z/2\), then it does not admit a nontrivial group action by automorphisms, so
\[H^0(K, {}_{p \circ b_1}(\Sym \Delta)) \cong \Z/2
  \cong H^0(K, \Sym \Delta).\]
We denote \(H^0(K, {}_{p \circ b_1}(\Sym \Delta))\) by \(S\) if \(K = k\) and \(S_v\) if \(K = k_v\).  By the previous remark, \(S\) and \(S_v\) do not depend on the choice of \(\mathbf{G_1}\) except in type \(D_4\).

As an approach for understanding
\[f_1\colon H^1(k, \Ad \mathbf{G_1})\to
  \prod_{v \nmid \infty} H^1(k_v, \Ad \mathbf{G_1}),\]
we now consider the short exact sequence
\begin{equation} \label{eq:center} 1 \longrightarrow Z(\mathbf{G_1}) \longrightarrow \mathbf{G_1} \longrightarrow \Ad \mathbf{G_1} \longrightarrow 1. \end{equation}
The map \(f_1\) thus also appears in the commutative square
\begin{equation} \label{diagram:boundary-map} \begin{tikzcd}
    \prod\limits_{v \nmid \infty} H^1(k_v, \Ad \mathbf{G_1}) \ar[r, "\prod (\delta_1)_v"]
      & \prod\limits_{v \nmid \infty} H^2(k_v, Z(\mathbf{G_1})) \\
    H^1(k, \Ad \mathbf{G_1}) \ar[r, "\delta_1"] \ar[u, "f_1"]
      & H^2(k, Z(\mathbf{G_1})) .\ar[u, "d_1"]
  \end{tikzcd} \end{equation}
We observe that \(S = H^0(K, \Sym \Delta)\) acts on \(Z(\mathbf{G_1})\), hence functorially on \(H^2(K, Z(\mathbf{G_1}))\) for both local and number fields \(K\) and the boundary map \(\delta_1\) is equivariant with respect to this action, so it induces orbit maps
\[\widetilde{\delta_1}\colon H^1(K, \Ad \mathbf{G_1})/H^0(K, \Sym \Delta)\to
  H^2(K, Z(\mathbf{G_1}))/H^0(K, \Sym \Delta).\]
As we were unable to find a proof for this fact, we provide one here.

\begin{proposition} \label{prop:s-equivariant}
  The differential \(\delta_1 \colon H^1(K, \Ad \mathbf{G_1}) \rightarrow H^2(K, Z(\mathbf{G_1}))\) is \(H^0(K, \Sym \Delta)\)-equivariant.
\end{proposition}

\begin{proof}
  We let \(H^0(K, \Sym \Delta) \eqqcolon S\) regardless of whether \(K\) is a number field or a local field.  The short exact sequence~\eqref{eq:ad-aut} still splits after twisting with the cocycle \(b_1\) because \(b_1\) is a lift of a cocycle in the quotient group \(\Sym \Delta\).  Therefore, we have
  \[S = (\Aut \mathbf{G_1})(K) \,/\, (\Ad \mathbf{G_1})(K).\]
  So take an element \(\chi \cdot (\Ad \mathbf{G_1})(K) \in S\) with \(\chi \in (\Aut \mathbf{G_1})(K)\). Let moreover \([a_\sigma Z(\mathbf{G_1})] \in H^1(K, \Ad \mathbf{G_1})\) where \(a_\sigma\) is a (set theoretical) lift to \(\mathbf{G_1}\) of a cocycle in \(\Ad \mathbf{G_1}\).  In the interpretation of \(\Ad \mathbf{G_1}\) as a subgroup of \(\Aut \mathbf{G_1}\), the element \(a_\sigma Z(\mathbf{G_1})\) corresponds to the well defined inner automorphism \(c_{a_\sigma} = a_\sigma (\,\cdot\,) a_\sigma^{-1}\).  Then by definition
  \[ \chi \cdot (\Ad \mathbf{G_1})(K).[a_\sigma Z(\mathbf{G_1})] = [\chi^\sigma \circ c_{a_\sigma} \circ \chi^{-1}] = [c_{\chi(a_\sigma)}] \]
  because \(\chi \in (\Aut\mathbf{G_1})(K)\) hence \(\chi^\sigma = \chi\).  Applying \(\delta_1\), we obtain
  \begin{align*}
    \delta_1([c_{\chi(a_\sigma)}])
      &= \delta_1([\chi(a_\sigma)Z(\mathbf{G_1})])
        = [\chi(a_\sigma)\, {}^\sigma\! \chi(a_\tau) \, (\chi(a_{\sigma \tau}))^{-1}] \\
    & = [\chi(a_\sigma \, {}^\sigma a_\tau \, a_{\sigma \tau}^{-1})]
      = \chi \cdot (\Ad \mathbf{G_1})(K). \delta_1([a_\sigma Z(\mathbf{G_1})]). \qedhere
  \end{align*}
\end{proof}

The upper horizontal map \(\prod_{v \nmid \infty} (\delta_1)_v\) in diagram~\eqref{diagram:boundary-map} is a factorwise isomorphism and \(\delta_1\) is surjective by theorems of Kneser~\cite{Platonov-Rapinchuk:algebraic-groups}*{Theorem~6.20 and Corollary, p.\,326}.  Both of these properties remain intact after factoring by \(S\) and \(\prod_{v \nmid \infty} S_v\), respectively.  We will analyze the diagonal map \(d_1\) by means of Poitou--Tate duality~\cite{Harari:galois}*{Theorem~17.13.(c), p.\,265}.  A first outcome is that \(d_1\), hence also the diagonal map \(f_1\), takes in fact values in the direct sum so that after factoring by \(S\) and \(\prod_{v \nmid \infty} S_v\), the diagram~\ref{diagram:boundary-map} has the form
\begin{equation} \label{eq:ad-center-square}
  \begin{tikzcd}
    \bigoplus\limits_{v \nmid \infty} H^1(k_v, \Ad \mathbf{G_1}) / S_v \ar[r, "\cong"]
      & \bigoplus\limits_{v \nmid \infty} H^2(k_v, Z(\mathbf{G_1})) / S_v\\
    H^1(k, \Ad \mathbf{G_1}) / S \ar[r, ->>, "\widetilde{\delta_1}"] \ar[u, "\widetilde{f_1}"]
      & H^2(k, Z(\mathbf{G_1})) / S. \ar[u, "\widetilde{d_1}"]
  \end{tikzcd} \end{equation}
  
  Since \(Z \coloneqq Z(\mathbf{G_1})\) is a finite abelian \(\operatorname{Gal}(k)\)-module, global Poitou--Tate duality~\cite{Harari:galois}*{Theorem~17.13.(c), p.\,265} provides a short exact sequence
\begin{equation} \label{eq:poitou-tate} 0 \longrightarrow H^2(k, Z) \xrightarrow{\ b \ } \bigoplus_v H^2(k_v, Z) \xrightarrow{\ c \ } H^ 0(k, Z')^* \longrightarrow 0. \end{equation}
Here, \(A^* = \operatorname{Hom}_c(A, \Q/\Z)\) denotes the \emph{dual} of a topological group~\(A\), consisting of continuous homomorphisms.  It is the usual Pontryagin dual if \(A\) is discrete torsion, which is true in our case.  The Galois module \(Z' = \operatorname{Hom}(Z, \mathbf{GL_1})\) is the \emph{Cartier dual} of \(Z\) with the \(\operatorname{Gal}(k)\)-action given by \((\sigma. f)(z) = \sigma(f(\sigma^{-1}(z)))\) for \(\sigma \in \operatorname{Gal}(k)\), \(f \in Z'\) and \(z \in Z\).  The homomorphism \(b\) is the diagonal to all (finite and infinite) localizations.  It turns out that the previously introduced map \(d\) is in fact the restriction of \(b\) to the finite places of \(k\).  It is proven in~\cite{Prasad-Rapinchuk:prescribed}*{p.\,658} that \(b\) is indeed injective for all our occurring Galois modules \(Z = Z(\mathbf{G_1})\).  The homomorphism \(c\) can be described as follows:  Let \(a\) be the diagonal map
  \[ a \colon H^0(k, Z') \longrightarrow \prod_v \widetilde{H}^0(k_v, Z') \]
  where we use the convention that \(\widetilde{H}^0(k_v, Z')\) denotes regular group cohomology \(H^0(k_v, Z')\) whenever \(v\) is finite and the Tate cohomology \(\widehat{H}^0(k_v, Z')\) whenever \(v\) is real or complex.  Then \(c\) is equal to \(a^*\) precomposed with the isomorphism \(H^2(k_v,Z) \cong \widetilde{H}^0(k_v, Z')^*\) by local Tate duality.

\begin{theorem} \label{thm:tate-map}
The map \(c\) in \eqref{eq:poitou-tate} is given by the maps in Table~\ref{tate-table} for finite places \(v\) and real places, respectively.
\end{theorem}

{\small
\begin{table}[tb]
    \renewcommand{\arraystretch}{2.5}
  \[\begin{NiceArray}{|c|c|ccc|}
    \hline
    \text{Type} & Z(\mathbf{G_1}) & H^2(k_v, Z(\mathbf{G_1}))
      & H^0(k, Z(\mathbf{G_1})')^* & H^2(\R, Z(\mathbf{G_1})) \\
    \hline\hline
    {}^1 A_{2n} & \mu_{2n+1}
      & \Z/(2n+1) & \Z/(2n+1) & 0  \\
    {}^2 A_{2n} & \mathbf{R}^{(1)}_{l/k}(\mu_{2n+1})
      & 0 & 0 & 0  \\
    \hline
    {}^1 A_{2n+1} & \mu_{2n+2}
      & \Z/(2n+2) & \Z/(2n+2) & \Z/2  \\
    {}^2 A_{2n+1} & \mathbf{R}^{(1)}_{l/k}(\mu_{2n+2})
      & \Z/2 & \Z/2 & \Z/2  \\
    \hline
    B_n & \mu_2 & \Z/2 & \Z/2 & \Z/2  \\
    \hline
    C_n & \mu_2 & \Z/2 & \Z/2 & \Z/2  \\
    \hline
    {}^1 D_{2n} & \mu_2 \times \mu_2
      & \Z/2 \times \Z/2 & \Z/2 \times \Z/2 & \Z/2 \times \Z/2  \\
    {}^2 D_{2n} & R_{l/k}(\mu_2)
      & \Z/2 & \Z/2 & 0  \\
    \hline
    {}^1 D_{2n+1} & \mu_4
      & \Z/4 & \Z/4 & \Z/2  \\
    {}^2 D_{2n+1} & \mathbf{R}^{(1)}_{l/k}(\mu_4)
      & \Z/2 & \Z/2 & \Z/2  \\
    \hline
    {}^1 E_6 & \mu_3
      & \Z/3 & \Z/3 & 0  \\
    {}^2 E_6 & \mathbf{R}^{(1)}_{l/k}(\mu_3)
      & 0 & 0 & 0  \\
    \hline
    E_7 & \mu_2 & \Z/2 & \Z/2 & \Z/2  \\
    \hline
    
    \CodeAfter
    \begin{tikzpicture}
    \begin{scope}[shorten <=2pt, every node/.style={above, font=\scriptsize}]
    \draw[->] (2-3) -- node{id} (2-4);
    \draw[->] (2-5) -- (2-4);
    \draw[->, shorten >=10pt] (2-3) -- (3-4);
    \draw[->, shorten >=10pt] (2-5) -- (3-4);
    \draw[->, shorten >=10pt] (3-3) -- (3-4);
    \draw[->, shorten >=10pt] (3-5) -- (3-4);
    \draw[->] (4-3) -- node{id} (4-4);
    \draw[->] (4-5) -- node{$\cdot(n+1)$} (4-4);
    \draw[->, shorten >=10pt] (4-3) -- node[sloped, anchor=center, above, font=\scriptsize]{pr} (5-4);
    \draw[->, shorten >=10pt] (4-5) -- node[sloped, anchor=center, above, font=\scriptsize]{id} (5-4);
    \draw[->, shorten >=10pt] (5-3) -- node{id} (5-4);
    \draw[->, shorten >=10pt] (5-5) -- node{id} (5-4);
    \draw[->, shorten >=10pt] (6-3) -- node{id} (6-4);
    \draw[->, shorten >=10pt] (6-5) -- node{id} (6-4);
    \draw[->, shorten >=10pt] (7-3) -- node{id} (7-4);
    \draw[->, shorten >=10pt] (7-5) -- node{id} (7-4);
    \draw[->] (8-3) -- node{id} (8-4);
    \draw[->] (8-5) -- node{id} (8-4);
    \draw[->, shorten >=10pt] (8-3) -- node[anchor=center, above, font=\scriptsize]{$+$} (9-4);
    \draw[->, shorten >=10pt] (8-5) -- node[anchor=center, above, font=\scriptsize]{$+$} (9-4);
    \draw[->, shorten >=10pt] (9-3) -- node{id} (9-4);
    \draw[->, shorten >=10pt] (9-5) -- (9-4);
    \draw[->] (10-3) -- node{id} (10-4);
    \draw[->] (10-5) -- node{$\cdot 2$} (10-4);
    \draw[->, shorten >=10pt] (10-3) -- node[sloped, anchor=center, above, font=\scriptsize]{pr} (11-4);
    \draw[->, shorten >=10pt] (10-5) -- node[sloped, anchor=center, above, font=\scriptsize]{id} (11-4);
    \draw[->, shorten >=10pt] (11-3) -- node{id} (11-4);
    \draw[->, shorten >=10pt] (11-5) -- node{id} (11-4);
    \draw[->] (12-3) -- node{id} (12-4);
    \draw[->] (12-5) --  (12-4);
    \draw[->, shorten >=10pt] (12-3) -- (13-4);
    \draw[->, shorten >=10pt] (12-5) -- (13-4);
    \draw[->, shorten >=10pt] (13-3) -- (13-4);
    \draw[->, shorten >=10pt] (13-5) -- (13-4);
    \draw[->] (14-3) -- node{id} (14-4);
    \draw[->] (14-5) -- node{id} (14-4);
    \end{scope}
    \end{tikzpicture}
  \end{NiceArray}\]
  \caption{Local Poitou--Tate homomorphisms.}
  \label{tate-table}
\end{table}
}

\begin{proof}
  We use local Tate duality
  \[H^2(k_v, Z) \cong \widetilde{H}^0(k_v, Z')^*\]
  to understand the involved local cohomology groups.  An explicit description of \(\mathbf{R}^{(1)}_{l/k}(\mu_r)\) as Galois module can be found in \cite{Kammeyer-Serafini:sign}*{Proposition~10} (where in the context there, it is stated for even \(r\) but it holds true for general \(r\) with the same proof).   It is given by the abelian group \(\mu_r(\overline{k})\) where \(\sigma \in \Gal(k)\) acts as usual and in addition by inversion whenever \(\sigma \notin \Gal(l)\).
  
  Recall now that the map \(c_v\) is
  \[H^2(k_v, Z) \xrightarrow{\sim} \widetilde{H}^0(k_v, Z')^* \xrightarrow{a_v^*} H^0(k, Z')^*\]
  where the first map is again the isomorphism given by local Tate duality.  We thus see that \(c_v\) is injective if \(a_v\) is surjective and as \(\Q/\Z\) is an injective \(\Z\)-module, \(c_v\) is surjective if \(a_v\) is injective.  The group \(\widetilde{H}^0(k, Z')\) consists of \(\Gal(k)\)-equivariant homomorphisms \(f \colon Z \rightarrow \mathbf{GL_1}\) and \(\widetilde{H}^0(k_v, Z')\) is given by \(\Gal(k_v)\)-equivariant homomorphisms \(f \colon Z \rightarrow \mathbf{GL_1}\).  If \(v\) is a real place, \(\widetilde{H}\) becomes Tate cohomology and we furthermore have to form the quotient modulo norms, meaning maps of the form \(f^\sigma \!\cdot\! f\) where \(\sigma \in \Gal(k_v)\) is the nontrivial element.  For finite places, \(\Gal(k_v) \subseteq \Gal(k)\) implies that \(H^0(k, Z') \subseteq H^0(k_v, Z') = \widetilde{H}^0(k_v, Z')\), so \(a_v\) is injective and \(c_v\) thus surjective regardless of whether or not \(l \subseteq k_v\), i.e.\,regardless of whether or not \(\mathbf{G_1}\) splits at \(v\).

  If \(Z = \mu_r\), every homomorphism \(f \colon Z \rightarrow \mathbf{GL_1}\) is of the form \(f(\zeta) = \zeta^m\), so Galois equivariance is an empty condition.  Thus \(H^0(k, Z') = H^0(k_v, Z')\) is a cyclic group of order \(r\) and \(a_v\) is bijective in a canonical manner.  Therefore, we can fix an identification \(H^0(k, Z') \cong \Z/r\) and let the isomorphisms \(c_v\) induce identifications \(H^2(k_v, Z) \cong \Z/r\) to make it the identity as claimed.  In the case of real places on the other hand, the group \(\widehat{H}^0(k_v, Z')\) has order one or two depending on whether \(r\) is odd or even, so \(a_v\) is merely surjective here.  Thus \(c_v\) is injective and the \(H^2(k_v, Z)\) are then embedded in the only possible way.

  If \(Z = \mathbf{R}^{(1)}_{l/k} (\mu_r)\) for a quadratic extension \(l/k\), the above description shows that \(H^0(k, Z')\) has order one or two depending on whether \(r\) is odd or even.  As \(c_v\) was surjective for finite places, one can check based on the involved groups that \(c_v\) has to be either the identity or the unique projection, depending on whether \(\mathbf{G_1}\) remains of outer type or splits at \(v\).  Let now \(v\) be a real place.  If \(v\) becomes complex in \(l\), then the nontrivial element of \(\Gal(k_v)\) restricts to an element in \(\Gal(k)\) which does not fix \(l\) pointwise, so \(H^0(k, Z') = H^0(k_v, Z') = \widehat{H}^0(k_v, Z')\).  If on the other hand \(v\) splits in \(l\), meaning \(l \subset k_v\), then \(Z'\) is a trivial \(\Gal(k_v)\)-module, hence \(H^0(k_v, Z')\) has order~\(r\), so \(\widehat{H}^0(k_v, Z')\) has order one or two depending on whether \(r\) is odd or even.  Moreover, the nontrivial element of \(H^0(k, Z')\) for even \(r\), mapping a primitive \(r\)-th root of unity to \(-1\), maps to the nontrivial element of \(\widehat{H}^0(k_v, Z')\).  So \(a_v\) is again surjective.
  
  If \(Z = \mu_2 \times \mu_2\), we see that \(H^0(k, Z') = H^0(k_v, Z') = \widehat{H}^0(k_v, Z) \cong \{ \pm 1 \} \times \{ \pm 1\}\).  Thus \(c_v\) is bijective for both finite and real places in this case.
  
  Let finally \(Z = \mathbf{R}_{l/k}(\mu_2)\) for a quadratic extension \(l/k\). Then
  \begin{align*}
    H^0(k, \mathbf{R}_{l/k}(\mu_2)') &= \operatorname{Hom}_{\operatorname{Gal}(k)}(\mu_2(l \otimes_k \overline{k}), \overline{k}^*) \\
    &= \operatorname{Hom}_{\operatorname{Gal}(k)}(\mu_2(\overline{k}) \times \mu_2(\overline{k}), \mu_2(\overline{k})) \cong \{ \pm 1 \}
  \end{align*}
  where the last equality holds because the action of \(\operatorname{Gal}(k)\) on \(\mu_2(\overline{k}) \times \mu_2(\overline{k})\) factors through \(\operatorname{Gal}(l/k)\) which acts by swapping the two factors.  So the only nontrivial equivariant homomorphism sends both \((-1,1)\) and \((1,-1)\) to \(-1\).  Thus also \(H^0(k_v, \mathbf{R}_{l/k}(\mu_2)') \cong \{ \pm 1 \}\) if \(v\) is a real place that becomes complex in \(l\).  But the above homomorphism is a norm so that \(\widehat{H}^0(k_v, \mathbf{R}_{l/k}(\mu_2)')\) is trivial.  If on the other hand a real place \(v\) splits into two real places in \(l\), then \(\mathbf{R}_{l/k}(\mu_2)'(\overline{k_v})\) is a trivial \(\Gal(k_v)\)-module, so that
  \[ H^0(k_v, \mathbf{R}_{l/k}(\mu_2)') \cong \{ \pm 1 \} \times \{ \pm 1 \} \]
  and also \(\widehat{H}^0(k_v, \mathbf{R}_{l/k}(\mu_2)') \cong \{ \pm 1 \} \times \{ \pm 1 \}\).  The map \(a_v\) is in this case the diagonal embedding, so \(c_v\) is given by summing both coordinates.  By the same argument, this also holds for finite places at which \(\mathbf{G_1}\) splits.
\end{proof}

\begin{remark} \label{rmk:dynkin-actions}
  In the types \(A_{n \geq 2}\), \(D_{2n+1}\) and \(E_6\), the group \(S_v \cong \Z/2\) acts by inversion on \(H^2(k_v, Z(\mathbf{G_1}))\).  In type \(D_{2n}\) where \(n > 2\), \(S_v \cong \Z/2\) acts by swapping the two factors under the canonical isomorphism \(H^2(k_v, Z(\mathbf{G_1})) \cong \Z/2 \times \Z/2\) as discussed in \cite{Kneser:galois}*{Section~5}.  In all other included types, \(\Sym \Delta\) is trivial.
\end{remark}

\section{Adelic automorphisms acting on Galois cohomology}
\label{section:adelic-automorphisms}

In this section, we give an explicit description of the action of the automorphism group \(\Aut(\A^f_k)\) on \(\prod_{v \nmid \infty} H^1(k_v, \Aut \mathbf{G_0})\) and we show that a \(k\)-group \(\mathbf{G}\) of outer type (but not of type \(D_4\)) is never congruence rigid if its quasi-split form has nontrivial fiber under \(\overline{h}\).  This separates the question of congruence rigidity into the number theoretic problem of determining the singleton fibers of \(\overline{h}\) and the Galois cohomological problem of determining the singleton fibers of \(\overline{f_1}\), both to be addressed in the next sections.

Let us first describe the structure of the automorphism group \(\Aut(\A^f_k)\) of a number field \(k\) in general.  As we mentioned in the introduction, restricting any \(\Phi \in \Aut(\A^f_k)\) to each factor of the restricted product \(\A^f_k = \prod'_{v \nmid \infty} k_v\) decomposes \(\Phi\) into a family of isomorphisms \(\Phi_v \colon k_v \rightarrow k_{v'}\) for a permutation \(v \mapsto v'\) of the set of finite places of \(k\)~\cite{Klingen:similarities}*{Theorem VI.2.3}.  Clearly, \(k_v\) and \(k_{v'}\) can only be isomorphic if \(v\) and \(v'\) lie over the same rational prime \(p\).  Thus \(\Phi\) restricts to an automorphism on each {\'e}tale algebra \(\prod_{v \mid p} k_v \cong k \otimes_\Q \Q_p\).  Therefore we have the decomposition
\[ \operatorname{Aut} (\A^f_k/\A^f_\Q) = \prod_{p \nmid \infty}\,  \prod_{[k_v] \colon v \mid p}\, \operatorname{Aut}(k_v/\Q_p) \wr \operatorname{Sym}_{\# [k_v]} \]
as a product of wreath products where the second product runs over all isomorphism classes of completions \(k_v/\Q_p\).

Let from now on \(\mathbf{G_1}\) and \(\mathbf{G_2}\) be two \(k\)-quasi-split forms of \(\mathbf{G_0}\) and assume that they are not of type \(D_4\).  Assume furthermore that there is a \(\Phi \in \Aut(\A_k^f)\) such that for all finite places \(v\), the isomorphism \(\Phi_v\colon k_v \to k_{v'}\) induces a bijection
\[ (\Phi_v)_*\colon H^1(k_v, \Aut \mathbf{G_0})\to H^1(k_{v'}, \Aut \mathbf{G_0}) \] of pointed sets that maps the class of \(\mathbf{G_1}\) to the class of \(\mathbf{G_2}\).

\begin{proposition} \label{prop:both-split}
 The \(k\)-group \(\mathbf{G_1}\) is \(k\)-split if and only if \(\mathbf{G_2}\) is \(k\)-split.
\end{proposition}

\begin{proof}
  If \(\mathbf{G_1}\) splits globally, then it splits at all local (and in particular all finite) places.  This implies that all \((\Phi_v)_*(\mathbf{G_1})\) also split, as  \((\Phi_v)_*\) preserves the base point.  But if \(\mathbf{G_2}\) splits at all finite places, then its class \([\mathbf{G_2}] \in H^1(k, \Aut \mathbf{G_0})\) is mapped to the base point under \(h \circ p\) in the following diagram
  \[ \begin{tikzcd}
    \prod\limits_{v \nmid \infty} H^1(k_v, \Aut \mathbf{G_0}) \ar[r, "P"]
      & \prod\limits_{v \nmid \infty} H^1(k_v, \Sym \Delta) \\
    H^1(k, \Aut \mathbf{G_0}) \ar[u, "g"] \ar[r, "p"]
      & H^1(k, \Sym \Delta). \ar[u, "h"]
  \end{tikzcd} \]
  Since \(h\) has trivial kernel~\cite{Kammeyer-Spitler:chevalley}*{Proposition~4} (even in type \(D_4\)), we see that \(p\) already maps \([\mathbf{G_2}]\) to the base point, so \(\mathbf{G_2}\) is an inner twist of \(\mathbf{G_0}\).  But since \(\mathbf{G_2}\) was also assumed to be \(k\)-quasi-split, it already has to be split.  By symmetry, the converse holds true as well.
\end{proof}

Next, we will identify the precise nature of the action of \(\Aut(\A_k)\) on \(\prod H^1(k_v, \Aut \mathbf{G_0})\) with infinite places included.  The essential step is understanding the individual action of \(\Aut(k_v)\) on a single factor \(H^1(k_v, \Aut \mathbf{G_0})\).

\begin{remark} \label{rmk:H^1-decomposition}
  As \(\Aut (\R)\) and \(H^1(\C, \Aut \mathbf{G_0})\) are both trivial, we do not need to consider the infinite places of \(k\) seperately with respect to their \(\Aut (k_v)\)-action.
\end{remark}

\begin{theorem} \label{thm:H^1-decomposition}
  Let \(v\) be a finite place of \(k\) and assume that \(\mathbf{G_0}\) is not of type \(D_4\).  Then \(H^1(k_v, \Aut \mathbf{G_0})\) is in bijection with the disjoint union
  \[H^2(k_v, Z(\mathbf{G_0}))/\Sym \Delta  \ \ \sqcup \ \ M \times \left(k_v^\times / (k_v^\times)^2 \setminus \{1\}\right)\]
  where \(\Sym \Delta\) acts as discussed in \Cref{rmk:dynkin-actions} and
  \begin{align*}M =
    \begin{cases}
      \Z/2 & \text{if } \mathbf{G_0} \text{ has type }
        A_{2n+1} \text{ for } n \ge 1 \text{ or } D_n \text{ for } n \ge 5, \\
      \{0\} & \text{if } \mathbf{G_0} \text{ has type }
        A_{2n} \text{ or } E_6, \\
      \varnothing & \text{otherwise.}
    \end{cases}
  \end{align*}
  The bijection is equivariant with respect to the natural action of \(\Aut (k_v)\) on \(H^1(k_v, \Aut \mathbf{G_0})\), the trivial action both on \(H^2(k_v, Z(\mathbf{G_0}))/\Sym \Delta\) and the factor \(M\), and the natural action on \(k_v^\times / (k_v^\times)^2 \setminus \{1\}\).
\end{theorem}

\begin{proof}[Proof of Theorem~\ref{thm:H^1-decomposition}]
  Due to the sequence in~\eqref{eq:ad-aut} and \cite{Serre:galois-cohomology}*{I.5.5~Cor.\,2}, we obtain a  short exact sequence of cohomology sets
  \begin{align*}
  1 \rightarrow H^1(k_v, \Ad \mathbf{G_0}) / H^0(k_v, \Sym \Delta)
    & \xrightarrow{i} H^1(k_v, \Aut \mathbf{G_0}) \\
    & \xrightarrow{p} H^1(k_v, \Sym \Delta) \rightarrow 1
  \end{align*}
  where the first map is even injective.  By functoriality, the map \(p\) is equivariant with respect to the \(\Aut (k_v)\)-action on \(H^1(k_v, \Sym \Delta)\).  In the case \(\Sym \Delta \cong \Z/2\),  we have
  \[ H^1(k_v, \Sym \Delta) \cong k_v^\times / (k_v^\times)^2 \]
  and under this isomorphism, the action of \(\Aut (k_v)\) on \(k_v^\times / (k_v^\times)^2\) is the natural one.  Since \(i\) is injective, the sequence shows that \(\ker p\) is in bijection with \(H^1(k_v, \Ad \mathbf{G_0}) / H^0(k_v, \Sym \Delta)\) which, in turn, can be identified with \(H^2(k_v, Z(\mathbf{G_0})) / \Sym \Delta\) as we saw in \eqref{eq:ad-center-square}.  Furthermore, the bijection is also \(\Aut (k_v)\)-equivariant by \Cref{prop:s-equivariant}. If \(\mathbf{G_0}\) is not of type \(D_{2n}\), then \(Z(\mathbf{G_0}) = \mu_m\) for some \(m\).  In that case, we have a natural isomorphism \(H^2(k_v, Z(\mathbf{G_0})) \cong \Br(k_v)[m]\) to the subgroup of the Brauer group of \(k_v\) consisting of those elements whose order divides \(m\).   So the action is trivial because for every \(\mathfrak{p}\)-adic field \(k_v\), the \(\Aut (k_v)\)-action on \(\Br(k_v)\) is trivial by~\cite{Janusz:automorphisms}*{Lemma, p.\,385}.  If \(\mathbf{G_0}\) has type \(D_{2n}\) and hence \(Z(\mathbf{G_0}) \cong \mu_2 \times \mu_2\), we have a natural isomorphism \(H^2(k_v, Z(\mathbf{G_0})) \cong \operatorname{Br}(k_v)[2] \times \operatorname{Br}(k_v)[2]\) so the action is trivial, too.  If \(\mathbf{G_0}\) is not of type \(A_{n \geq 2}\), \(D_{n \geq 5}\) or \(E_6\), the map \(i\) is actually a bijection and the statement follows.

  If \(\mathbf{G_0}\) does have type \(A_{n \geq 2}\), \(D_{n \geq 5}\) or \(E_6\), then each nontrivial \(\gamma \in H^1(k_v, \Sym \Delta) \cong k_v^\times / (k_v^\times)^2\) determines a \(k_v\)-quasi-split form \(\mathbf{G_\gamma}\) corresponding to \(s(\gamma) \in H^1(k_v, \Aut \mathbf{G_0})\).  By twisting the entire Sequence~\eqref{eq:ad-aut} as discussed previously, we obtain the sequence
  \[
  1 \rightarrow H^1(k_v, \Ad \mathbf{G_\gamma}) / S_v
    \rightarrow H^1(k_v, \Aut \mathbf{G_\gamma})
    \rightarrow H^1(k_v, \Sym \Delta) \rightarrow 1.
  \]
  We observe that \(S_v\) must act trivially on \(H^1(k_v, \Ad \mathbf{G_1}) \le \Z/2\) by Table~\ref{tate-table} so that the quotient may be ommitted.  The kernel of this sequence is in bijection with the fiber \(p^{-1}(\gamma)\).  Therefore, we obtain a disjoint union decomposition of the complement of \(\ker p\) in \(H^1(k_v, \Aut \mathbf{G_0})\) by means of the bijection
  \[H^1(k_v, \Aut \mathbf{G_0}) \setminus \ker p \ \longrightarrow \!\!\!\! \bigsqcup_{
  \gamma \in (k_v^\times / (k_v^\times)^2) \setminus \{1\}} \!\! H^1(k_v, \Ad \mathbf{G_{\gamma}}).\]
 By functoriality, we have an \(\Aut (k_v)\)-action on the disjoint union such that \(\Phi_v \in \Aut (k_v)\) maps \(H^1(k_v, \Ad \mathbf{G_{\gamma}})\) bijectively and base point preserving to \(H^1(k_v, \Ad \mathbf{G}_{\Phi_v(\gamma)})\).  Since \(p\) is \(\Aut (k_v)\)-equivariant, so is the above bijection.  Furthermore, the isomorphism in the first row of Diagram~\ref{diagram:boundary-map} is similarly \(\Aut (k_v)\)-equivariant in the following sense.  The automorphism \(\Phi_v\) induces a commutative square
  \begin{equation} \label{diagram:phi-h1-h2-compability} \begin{tikzcd}
    H^1(k_v, \Ad \mathbf{G_{\gamma}}) \ar[r, "\cong"] \ar[d, "(\Phi_v)_*"]
      & H^2(k_v, Z(\mathbf{G_{\gamma}})) \ar[d, "(\Phi_v)_*"] \\
    H^1(k_v, \Ad \mathbf{G}_{\Phi_v(\gamma)}) \ar[r, "\cong"]
      & H^2(k_v, Z(\mathbf{G}_{\Phi_v(\gamma)}))
  \end{tikzcd} \end{equation}
by naturality.  But we have in fact that both the groups \(H^2(k_v, Z(\mathbf{G_{\gamma}}))\) and \(H^2(k_v, Z(\mathbf{G}_{\Phi_v(\gamma)}))\) are either trivial if \(\mathbf{G_{\gamma}}\) is of type \(A_{2n}\) or \(E_6\) or isomorphic to \(\Z/2\) if \(\mathbf{G_{\gamma}}\) is of type \(A_{2n+1}\) or \(D_{n \geq 5}\) by \Cref{tate-table}.  Hence, \((\Phi_v)_*\) is always the same unique isomorphism.  Setting \(M = H^2(k_v, Z(\mathbf{G_\gamma}))\) for any fixed nontrivial \(\gamma \in k_v^\times / (k_v^\times)^2\), we thus obtain the \(\Aut(k_v)\)-equivariant bijection
  \[\bigsqcup_{\gamma \in (k_v^\times / (k_v^\times)^2) \setminus \{1\}} \!\! H^2(k_v, Z(\mathbf{G_{\gamma}})) \ \ \longrightarrow \ \ M \times \left(k_v^\times / (k_v^\times)^2 \setminus \{1\}\right)\]
for the types \(A_{n \geq 2}\), \(D_{n \geq 5}\) or \(E_6\).
\end{proof}

Recall that for every finite place \(v\) of \(k\) and any \(k_v\)-quasi-split group \(\mathbf{G_1}\), we have the regular local Tate duality pairing
  \[ \langle \ , \, \rangle_v \colon H^0(k_v, Z(\mathbf{G_1})') \otimes H^2(k_v, Z(\mathbf{G_1})) \longrightarrow \Q/\Z \]
and a similar \(\Z/2\)-valued pairing for Tate cohomology if \(v\) is real.  We will now work towards understanding how this duality interacts with \(\Aut(\A_k^f)\).  As before, we assume that none of the involved \(k\)-groups are of type \(D_4\).

\begin{proposition} \label{prop:aut-kv-equivariance}
  Let \(\Phi_v \colon k_v \xrightarrow{\cong} k_w\) be an isomorphism where \(v\) and \(w\) are places of \(k\).  Then for every \(\eta \in H^0(k_v, Z(\mathbf{G_1})')\) and every \(\theta \in H^2(k_v, Z(\mathbf{G_1}))\), we have
  \[ \langle (\Phi_v)_* (\eta), (\Phi_v)_* (\theta) \rangle_w = \langle \eta, \theta \rangle_v. \]
\end{proposition}

\begin{proof}
  The statement is trivial for complex places, so let \(v\) be finite or real.  The pairing \(\langle \ , \, \rangle_v\) is given by cup product, followed by the map induced from \(\mathrm{ev} \colon Z(\mathbf{G_1})' \otimes Z(\mathbf{G_1}) \rightarrow \mathbf{GL_1}\) and the natural isomorphism \(H^2(k_v, \mathbf{GL_1}) \cong \Br{k_v} \cong \Q/\Z \text{ or } \Z / 2\).  The induced map \((\Phi_v)_*\) on the second cohomologies as seen in Diagram~\ref{diagram:phi-h1-h2-compability} has the 2-cocycle description \(a_{\sigma, \tau} \mapsto {}^\varphi a_{\varphi \sigma \varphi^{-1}, \varphi \tau \varphi^{-1}}\) where \(\varphi\) is an extension of \(\Phi_v\) to an isomorphism of the algebraic closures.  For \(f \in Z(\mathbf{G_1})'\) and \(z \in Z(\mathbf{G_1})\), we have
    \[\mathrm{ev}({}^\varphi f \otimes {}^\varphi z) = \varphi \circ f \circ \varphi^{-1} (\varphi (z)) = \varphi (f(z)) = {}^\varphi(\mathrm{ev}(f \otimes z)). \]
    As the induced map \(\mathrm{ev}_*\) composes a 2-cocycle with \(\mathrm{ev}\), we see that \((\Phi_v)_* \otimes (\Phi_v)_*\) and \(\mathrm{ev}_*\) commute.  Naturality of the cup product thus gives
  \[ \mathrm{ev}_*((\Phi_v)_* (\eta) \smile (\Phi_v)_* (\theta)) = (\Phi_v)_* (\mathrm{ev}_*(\eta \smile \theta)) \]
  for every \(\eta \in H^0(k_v, Z(\mathbf{G_1})')\) and \(\theta \in H^2(k_v, Z(\mathbf{G_1}))\).  Finally, it follows again from \cite{Janusz:automorphisms}*{Lemma, p.\,385}, that the triangle
  \[ \begin{tikzcd}
      H^2(k_v, \mathbf{GL_1}) \arrow[rr, "(\Phi_v)_*", "\cong"'] \arrow[dr, "\cong"'] & & H^2(k_w, \mathbf{GL_1}) \arrow[dl, "\cong"] \\
      & \Q / \Z &
    \end{tikzcd}
  \]
  commutes if \(v\) is finite.  If \(v\) is real, it is trivial that the same triangle commutes with \(\Z / 2\) in place of \(\Q / \Z\).  This completes the proof.
\end{proof}

\begin{corollary} \label{cor:poitou-tate-sum-equivariance}
  For every \(\Phi \in \Aut (\A_k)\), all \((\eta_v) \in \prod_v H^0(k_v, Z(\mathbf{G_1})')\) and all \((\theta_v) \in \bigoplus_v H^2(k_v, Z(\mathbf{G_1}))\), we have
  \[ \langle \Phi_*(\eta_v) , \Phi_*(\theta_v) \rangle = \langle (\eta_v), (\theta_v) \rangle. \]
\end{corollary}

\begin{proof}
  Since the global Poitou--Tate duality pairing is given by summing up the local Tate duality parings, the assertion is immediate from the previous proposition.
\end{proof}

To summarize the previous results, \Cref{thm:H^1-decomposition} and \Cref{rmk:H^1-decomposition} tell us the following: Fix \(\mathbf{G_1}\) and \(\mathbf{G_2}\) as before and take an arbitrary \(\Phi_v\colon k_v\to k_w\) (where \(v\) and \(w\) may now also be infinite places) such that \(\Phi_v\) maps the class of \([\mathbf{G_1}] \in H^1(k_v, \Aut \mathbf{G_0})\) to the class of \([\mathbf{G_2}] \in H^1(k_w, \Aut \mathbf{G_0})\).  Then the induced isomorphism
\[  (\Phi_v)_* \colon H^2(k_v, Z(\mathbf{G_1})) \xrightarrow{\cong} H^2(k_w, Z(\mathbf{G_2}))\]
does not depend on the exact choice of \(\Phi_v\).  Furthermore, \Cref{prop:aut-kv-equivariance} tells us that the same holds (in a way that is compatible with the Poitou--Tate duality pairing) for the induced isomorphism
\[  (\Phi_v)_* \colon H^0(k_v, Z(\mathbf{G_1})') \xrightarrow{\cong} H^0(k_w, Z(\mathbf{G_2})').\]
Let us denote by
\begin{align*}
  j_v \colon & H^0(k, Z(\mathbf{G_1})') \rightarrow H^0(k_v, Z(\mathbf{G_1})') \text{ and} \\
  j_w \colon & H^0(k, Z(\mathbf{G_2})') \rightarrow H^0(k_w, Z(\mathbf{G_2})')
\end{align*}
the maps induced by the completion maps \(k \rightarrow k_v\) and \(k \rightarrow k_w\).  If \(v\) and \(w\) are finite places, then \(j_v\) and \(j_w\) are inclusions.  

\begin{proposition} \label{prop:local-global-compability}
  There is an isomorphism
  \[\psi\colon H^0(k, Z(\mathbf{G_1})')\to H^0(k, Z(\mathbf{G_2})')\]
  such that for all \(\Phi_v\colon k_v\to k_w\) for which \((\Phi_v)_*\) maps the class of \([\mathbf{G_1}] \in H^1(k_v, \Aut \mathbf{G_0})\) to the class of \([\mathbf{G_2}] \in H^1(k_w, \Aut \mathbf{G_0})\),
  \begin{equation} \begin{tikzcd}
    H^0(k, Z(\mathbf{G_1})') \ar[r, "j_v"] \ar[d, "\psi"]
      & H^0(k_v, Z(\mathbf{G_1})') \ar[d, "(\Phi_v)_*"] \\
    H^0(k, Z(\mathbf{G_2})') \ar[r, "j_w"] 
      & H^0(k_w, Z(\mathbf{G_2})')
  \end{tikzcd} \end{equation}
  commutes.
\end{proposition}

\begin{proof}
  If \(\mathbf{G_1}\) is \(k\)-split, then so is \(\mathbf{G_2}\) by \Cref{prop:both-split} and in particular even \(\mathbf{G_1} \cong \mathbf{G_2}\) over \(k\).  Hence \(Z(\mathbf{G_1})' \cong Z(\mathbf{G_2})'\) is a constant group scheme, so the assertion is clear by choosing \(\psi\) to be the identity and noting that \((\Phi_v)_*\colon H^0(k_v, Z(\mathbf{G_1})')\to H^0(k_w, Z(\mathbf{G_2})')\) and likewise \(j_v\) and \(j_w\) induce the identity.  Note furthermore that there is nothing to prove if \(v\) is complex.

  Now let \(\mathbf{G_1}\) (and hence \(\mathbf{G_2}\)) be non-split globally.  If \(v\) is finite and \(\mathbf{G_1}\) does not split at \(v\), then neither does \(\mathbf{G_2}\) by  \((\Phi_v)_*(\mathbf{G_1}) = \mathbf{G_2}\).  Local Tate duality gives \(H^0(k_v, Z(\mathbf{G_1})') \cong H^2(k_v, Z(\mathbf{G_1}))^*\) and \(H^0(k_v, Z(\mathbf{G_2})') \cong H^2(k_v, Z(\mathbf{G_2}))^*\).  So from what we observed in the proof of \Cref{thm:H^1-decomposition}, we get that \(H^0(k_v,Z(\mathbf{G_1})') \cong H^0(k_w,Z(\mathbf{G_2})')\) are either trivial or \(\Z/2\).  Additionally, \(H^0(k,Z(\mathbf{G_1})')\) and \(H^0(k,Z(\mathbf{G_2})')\) are both also isomorphic to \(H^0(k_v,Z(\mathbf{G_1})')\).  As we obtain four times either the trivial group or four times \(\Z/2\), we choose the only possible isomorphism \(\psi\).  As \(j_w\) and \(j_v\) are both injective, it is automatic that \(j_w \circ \psi = (\Phi_v)_* \circ j_v\).

  If on the other hand \(\mathbf{G_1}\) splits at a finite place \(v\) and is not of type \({}^2 D_{2n}\) over \(k\), then \(\mathbf{G_2}\) splits at \(w\).  Thus the groups \(H^0(k_v, Z(\mathbf{G_1})') \cong H^0(k_w, Z(\mathbf{G_2})')\) are cyclic and \(j_v\) and \(j_w\) must be the unique injective homomorphism, respectively.  If \(\mathbf{G_1}\) and \(\mathbf{G_2}\) have type \({}^2 D_{2n}\), then \(Z(\mathbf{G_1}) = \mathbf{R}_{l_1/k}(\mu_2)\) and \(Z(\mathbf{G_2}) = \mathbf{R}_{l_2/k}(\mu_2)\) are restriction of scalars of \(\mu_2\) for quadratic extensions \(l_1/k\) and \(l_2/k\).  We obtain \(H^0(k,Z(\mathbf{G_1})') \cong \Z/2 \cong H^0(k,Z(\mathbf{G_2})')\) and furthermore \(H^0(k_v, Z(\mathbf{G_1})') \cong \Z/2 \times \Z/2\) as well as \(H^0(k_w, Z(\mathbf{G_2})') \cong \Z/2 \times \Z/2\).  Choosing the latter isomorphisms to be compatible with the one induced by \(\Phi_v\), one easily checks that under a correct identification (which is up to switching the factors of \(\Z/2 \times \Z/2\)), the maps \(H^0(k, Z(\mathbf{G_1})') \rightarrow H^0(k_v, Z(\mathbf{G_1})')\) and \(H^0(k, Z(\mathbf{G_2})') \rightarrow H^0(k_w, Z(\mathbf{G_2})')\) are the diagonal map \(\Z/2 \rightarrow \Z/2 \times \Z/2\).  Again we choose as \(\psi\) the only possible isomorphism \(H^0(k,Z(\mathbf{G_1})') \cong \Z/2 \cong H^0(k,Z(\mathbf{G_2})')\).  So the identity \(j_w \circ \psi = (\Phi_v)_*\circ j_v\) follows.

  If \(v\) is real, it is again straightforward to see from local Tate duality that the Tate cohomology group \(H^0(k_v, Z(\mathbf{G_1})')\) is trivial or \(\Z/2\) unless \(\mathbf{G_1}\) has type \({}^2 D_{2n}\).  Since the localization map \(H^0(k,Z(\mathbf{G_1})') \rightarrow H^0(k_v, Z(\mathbf{G_1})')\) is surjective and \(H^0(k, Z(\mathbf{G_1})')\) is cyclic and all of this also holds for \(\mathbf{G_2}\), these maps are then unique and \(j_w \circ \psi = (\Phi_v)_*\circ j_v\) follows.  Finally, let \(v\) be real and \(\mathbf{G_1}\) be of type \({}^2 D_{2n}\).  If \(\mathbf{G_1}\) remains non-split at \(v\) (and hence \(\mathbf{G_2}\) likewise at \(w\)), then \(H^0(k_v, Z(\mathbf{G_1})') \cong \Z / 2\), so \(j_w \circ \psi = (\Phi_v)_*\circ j_v\) is clear.  If \(\mathbf{G_1}\) splits at \(v\), then \(Z(\mathbf{G_1})\) becomes \(\mu_2 \times \mu_2\) at \(v\) and similarly for \(\mathbf{G_2}\), from which one easily sees that the norm map is trivial.  Thus zeroth Tate cohomology coincides with ordinary cohomology and the arguments from the case for finite places apply here as well.  Hence also in this case, \(j_w \circ \psi = (\Phi_v)_*\circ j_v\).
\end{proof}

The previous calculations lead to the observation that the kernel of \(c\) is invariant under the action of \(\Aut(\A_k)\) in the following sense.  

\begin{corollary} \label{cor:poitou-compatability}
  Let \(\Phi \in \Aut(\A_k)\) and let \(V \subseteq V(k)\) be a subset of places of \(k\) for which \(\Phi_v\colon k_v\to k_w\) maps the class of \([\mathbf{G_1}] \in H^1(k_v, \Aut \mathbf{G_0})\) to the class of \([\mathbf{G_2}] \in H^1(k_w, \Aut \mathbf{G_0})\).  Then the diagram
  \begin{equation} \begin{tikzcd}
    \bigoplus\limits_{v \in V} H^2(k_v, Z(\mathbf{G_1})) \ar[r, "c'"] \ar[d, "\Phi_*"]
      & H^0(k, Z(\mathbf{G_1})') \ar[d, "\psi^*"] \\
    \bigoplus\limits_{v \in V} H^2(k_w, Z(\mathbf{G_2})) \ar[r, "c'"] 
      & H^0(k, Z(\mathbf{G_2})')
  \end{tikzcd} \end{equation}
  commutes, where \(c'\) is the restriction of the map \(c\) from \Cref{eq:poitou-tate} to the direct sum over the places in \(V\).
\end{corollary}

\begin{proof}
  This follows from \Cref{cor:poitou-tate-sum-equivariance} and \Cref{prop:local-global-compability}, as \(c_v\) is equal to \(j_v^*\) composed with the duality map from the Poitou--Tate pairing. 
\end{proof}

This corollary now allows us to show that whenever \(\overline{h}\) is not injective, then neither is \(\overline{g}\):

\begin{theorem} \label{thm:inner-twists-of-nontrivial-h-bar}
Let \(\beta_1 \in H^1(k, \Aut \mathbf{G_0})\) be an arbitrary \(k\)-form of a split \(k\)-group \(\mathbf{G_0}\) not of type \(D_4\).  Let \(\gamma_1 \in H^1(k, \Sym \Delta)\) be its image under the map \(p\colon H^1(k, \Aut \mathbf{G_0}) \rightarrow H^1(k, \Sym \Delta)\).  If the class \([\gamma_1] \in H^1(k, \Sym \Delta)/\Aut(k)\) has nontrivial fiber under
\[\overline{h}\colon H^1(k, \Sym \Delta) / \Aut(k) \longrightarrow \prod_{v \nmid \infty} H^1(k_v, \Sym \Delta)\ \big/ \Aut(\A_k^f),\]
then \([\beta_1] \in H^1(k, \Aut \mathbf{G_0})/\Aut(k)\) has nontrivial fiber under
\[\overline{g}\colon H^1(k, \Aut \mathbf{G_0}) / \Aut(k) \longrightarrow \prod_{v \nmid \infty} H^1(k_v, \Aut \mathbf{G_0})\ \big/ \Aut(\A_k^f).\]
\end{theorem}

Note that the theorem is only of interest in case \(\Sym \Delta \cong \Z / 2\) where we have \(H^1(k, \Sym \Delta) \cong k^\times / (k^\times)^2\) so that \(\gamma_1\) corresponds to a quadratic extension of \(k\).

\begin{proof}
  Let \(\gamma_1'\) be the image of \(\gamma_1\) under \(h\).  If \(\gamma_1 = 1\), then \(\gamma_1' = 1\) at all finite places, so it is fixed by \(\Aut (\A^f_k)\) as the action is base point preserving.  Since \(h\) is injective, the fiber of \([\gamma_1]\) under \(\overline{h}\) is trivial, so the condition is not met.  Hence we can assume \(\gamma_1 \not\equiv 1 \mod (k^\times)^2\).  Injectivity of \(h\) and a nontrivial fiber of \(\overline{h}\) now imply the existence of some \(\gamma_2 \in H^1(k, \Sym \Delta)\) and some \(\Phi \in \Aut(\A^f_k)\) such that \(\Phi(\gamma_1') = \gamma_2'\) where \(\gamma_2' = h(\gamma_2)\) and such that \(\gamma_2\) does not lie in the \(\Aut(k)\)-orbit of \(\gamma_1\).  Let now \(\beta_1'\) be the image of \(\beta_1\) under \(g\) and let \(\beta_2' \coloneqq \Phi(\beta_1')\).
  
  All the various coordinates of \(\beta_2' \in \prod_{v \nmid \infty} H^1(k_v, \Aut \mathbf{G_0})\) correspond to inner twists of the \(k_v\)-quasi-split forms given by the images of the local coordinates of \(\gamma_2'\) under the respective sections of the surjective maps \(H^1(k_v, \Aut \mathbf{G_0}) \rightarrow H^1(k_v, \Sym \Delta)\).  As \(\gamma_2'\) is the image of \(\gamma_2\), all those local \(k_v\)-quasi-split forms can be realized as the same global \(k\)-quasi-split form which we denote by \(\mathbf{G_2}\).  By Proposition~\ref{prop:both-split}, the group \(\mathbf{G_2}\) does not split over \(k\).  Let as before \(\mathbf{G_1}\) be the \(k\)-quasi-split form of \(\beta_1\).  By twisting the diagram
  \[ \begin{tikzcd}
    \prod\limits_{v \nmid \infty} H^1(k_v, \Ad \mathbf{G_0}) \ar[r, "I"]
      & \prod\limits_{v \nmid \infty} H^1(k_v, \Aut \mathbf{G_0})\\
    H^1(k, \Ad \mathbf{G_0}) \ar[u, "f"] \ar[r, "i"]
      & H^1(k, \Aut \mathbf{G_0}) \ar[u, "g"]
  \end{tikzcd} \]
  to \(\mathbf{G_1}\) as in \eqref{diagram:f-1-g-1-square}, one obtains a (not necessarily unique) preimage \(\alpha_1 \in H^1(k, \Ad \mathbf{G_1})\) of \(\beta_1\) under \(i_1\) and \(\alpha_1' \coloneqq f_1(\alpha_1)\).  As in the proof of \Cref{thm:H^1-decomposition}, the automorphism \(\Phi\) induces an isomorphism of pointed sets \(\Phi_v\colon H^1(k_v, \Ad \mathbf{G_1}) \rightarrow H^1(k_w, \Ad \mathbf{G_2})\) for each finite place \(v\) of \(k\) where \(k_w\) is the codomain of \(\Phi_v\).  Let \(\alpha_2' \coloneqq \Phi(\alpha_1')\).
  
  As \(\Phi(\gamma_1') = \gamma_2'\), we know that the corresponding quadratic extensions \(k(\sqrt{\smash[b]{\gamma_1}})\) and \(k(\sqrt{\smash[b]{\gamma_2}})\) have isomorphic finite adele rings, hence are arithmetically equivalent.  By \cite{Klingen:similarities}*{Theorem~III.1.4.g)} those number fields have the same number of real places.  This means that the number of real places \(w\) at which the image of \(\gamma_i\) under
  \[H^1(k, \Sym \Delta)\to H^1(k_w, \Sym \Delta) \cong \pm 1\]
  is 1 is the same for \(\gamma_1\) and \(\gamma_2\).  In turn, this means that there is a bijection \(\psi\colon V_{\infty}(k)\to V_{\infty}(k)\) of the infinite places of \(k\) such that \(\mathbf{G_1}\) splits at \(w \in V_{\infty}(k)\) if and only if \(\mathbf{G_2}\) splits at \(\psi(w)\).  This induces isomorphisms compatible with the map \(c\) from the Poitou--Tate sequence~\eqref{eq:poitou-tate}
  \[\psi_w\colon H^2(k_w, Z(\mathbf{G_1}))\to
    H^2(k_{\psi(w)}, Z(\mathbf{G_2}))\]
  according to \Cref{cor:poitou-compatability}.
  
  We now construct \(\theta \in \bigoplus_{v \in V(k)} H^2(k_v, Z(\mathbf{G_2}))\) by
  \begin{align*}
    \theta_v
      & = \delta_v(\Phi(\alpha_1')_v) = \delta_v((\alpha_2')_v) \ \text{ if } v \nmid \infty \\
    \theta_w
      & = \psi_w(\delta(\alpha_1)_w) \ \text{ if } w \mid \infty.
  \end{align*}
  By construction and \Cref{cor:poitou-compatability}, we obtain \(c(\theta) = 0\), so \(\theta\) has a preimage under \(b\) and therefore \(\theta|_{V_f(k)} = (\prod_{v\nmid \infty} \delta_v)(\alpha_2')\) has a preimage under \(d\).  But as \(\delta\) is surjective and \((\prod_{v\nmid \infty} \delta_v)\) is an isomorphism, this implies that \(\alpha_2'\) has a preimage under
  \[ f_2 \colon H^1(k, \Ad \mathbf{G_2}) \longrightarrow \prod_{v \nmid \infty} H^1(k_v, \Ad \mathbf{G_2}). \]
  
  Denote this preimage by \(\alpha_2\) and let \(\beta_2\coloneqq i_2(\alpha_2) \in H^1(k, \Aut \mathbf{G_2})\).  Furthermore, by functoriality, the diagram
  \[ \begin{tikzcd}
    \prod\limits_{v \nmid \infty} H^1(k_v, \Ad \mathbf{G_1}) \ar[d, "\Phi"] \ar[r, "I_1"]
      & \prod\limits_{v \nmid \infty} H^1(k_v, \Aut \mathbf{G_1}) \ar[d, "\Phi"]\\
    \prod\limits_{w \nmid \infty} H^1(k_w, \Ad \mathbf{G_2}) \ar[r, "I_2"]
      & \prod\limits_{w \nmid \infty} H^1(k_w, \Aut \mathbf{G_2})
  \end{tikzcd} \]
  commutes, so \(I_2(\alpha_2') = \beta_2'\).  Hence
  \[g_2(\beta_2) = g_2(i_2(\alpha_2)) = f_2(I_2(\alpha_2)) = f_2(\alpha_2') = \beta_2',\]
  so after twisting back to \(\mathbf{G_0}\), one obtains a preimage of the element \(\beta_2' \in \prod_{v \nmid \infty} H^1(k_v, \Aut \mathbf{G_0})\) under \(g\).  This preimage \(\beta_2 \in H^1(k, \Aut \mathbf{G_0})\) is mapped to \(\gamma_2 \in H^1(k, \Sym \Delta)\) under the \(\Aut(k)\)-equivariant map \(p\) and \(\gamma_2\) was in an \(\Aut(k)\)-orbit different from \(\gamma_1\).  Thus, \(\beta_2\) and \(\beta_1\) lie in different \(\Aut(k)\)-orbits and the fiber of \(\overline{g}\) at \([\beta_1]\) is nontrivial.
\end{proof}

\section{Quadratic extensions and adelic automorphisms}
\label{section:quadratic}

In this section, we show that the map \(\overline{h}\) is not injective if \(\Sym \Delta \cong \Z / 2\) in general and that injectivity can actually fail in two different ways.  In particular, we obtain many quasi-split groups, even over locally determined fields, that are not congruence rigid.  On the other hand we show that \(\overline{h}\) is injective if \(k/\Q\) is Galois.

Recall that the map \(\overline{h}\) is given by
\[\overline{h}\colon \ H^1(k, \Sym \Delta) \,/\, \Aut(k) \longrightarrow \prod_{v \nmid \infty} H^1(k_v, \Sym \Delta)\ \big/ \Aut(\A_k^f). \]
When \(\Sym \Delta = \Z / 2\), corresponding to the types \(A_n\) for \(n \ge 2\), \(D_n\) for \(n \ge 5\) and \(E_6\), decoding the cocycle definition of Galois cohomology gives that the concrete description
\[ \overline{h} \colon \ k^\times/(k^\times)^2 \,/\, \Aut(k) \longrightarrow \prod_{v \nmid \infty} k_v^\times / (k_v^\times)^2 \,/\, \Aut(\A_k^f),  \]
where the action of the automorphism groups are the apparent ones.  In view of \Cref{thm:inner-twists-of-nontrivial-h-bar} and in preparation for \Cref{thm:f-g-equivalence}, our main concern is to identify the singleton fibers of this map.  To do so, let us first give an algebraic interpretation of the situation.  Elements of the domain of \(\overline{h}\) correspond to quadratic extensions \(l/k\) up to the equivalence \(l_1 \sim l_2\) if and only if there exists an isomorphism \(\phi \colon l_1 \rightarrow l_2\) over \(\Q\) that restricts to an automorphism of \(k\).  Two quadratic extensions \(l_1/k\) and \(l_2 / k\) have the same image under \(\overline{h}\) if and only if there exists an isomorphism \(\Phi \colon \A_{l_1}^f \rightarrow \A_{l_2}^f\) over \(\A_\Q^f\) that restricts to an automorphism of \(\A_k^f\).

This description puts us on the right track to find examples of nontrivial fibers.  Number fields \(l_1\) and \(l_2\) with \(\A_{l_1}^f \cong_{\A_\Q^f} \A_{l_2}^f\) are called \emph{locally isomorphic}.  There certainly exist non-isomorphic but locally isomorphic number fields that are quadratic extensions of a common subfield.  An example is provided by K.\,Komatsu~\cite{Komatsu:adele-rings}: we can take \(k = \Q(\sqrt[4]{7})\), \(l_1 = \Q(\sqrt[8]{7})\), and \(l_2 = \Q(\sqrt[8]{112})\).  So the question is whether the local isomorphism can be chosen to restrict to an automorphism on \(\A^f_k\).  Equivalently, for each prime \(p\), we have an isomorphism \(l_1 \otimes_\Q \Q_p \cong l_2 \otimes_\Q \Q_p\) of {\'e}tale algebras and the task is to find such an isomorphism that restricts to an automorphism of the {\'e}tale subalgebra \(k \otimes_\Q \Q_p \subset l_i \otimes_\Q \Q_p\) for \(i = 1, 2\).  Outside ramified primes, this can always be achieved.

\begin{proposition} \label{prop:unramified}
  Let \(k\) be a number field, let \(l_1\) and \(l_2\) be quadratic extensions of \(k\), and let \(p\) be a rational prime which is unramified both in \(l_1\) and \(l_2\).  Suppose that there exists an isomorphism of {\'e}tale algebras
  \[ l_1 \otimes_\Q \Q_p \cong l_2 \otimes_\Q \Q_p. \]
  Then there exists such an isomorphism which restricts to an automorphism of the {\'e}tale subalgebra \(k \otimes_\Q \Q_p \subset l_i \otimes_\Q \Q_p\) for \(i = 1, 2\).
\end{proposition}

\begin{proof}
  Recall that we have \(k \otimes_\Q \Q_p \cong \prod_{v \mid p} k_v\).  Sorting the inertia degrees \(f_v = \dim_{\Q_p} k_v\) in a non-decreasing tuple
  \[ (f_1, \ldots, f_1, f_2, \ldots, f_2, \ldots, f_r, \ldots, f_r) \]
gives the \emph{decomposition type} of \(p\) in \(k\).  By the assumption \(l_1 \otimes_\Q \Q_p \cong l_2 \otimes_\Q \Q_p\), we know that the inertia degree \(f_1\) reoccurs as many times in \(l_1 / \Q\) as it does in \(l_2 / \Q\).  Hence the same number of places \(v \mid p\) of \(k\) with inertia degree \(f_1\) split and are inert in \(l_1\) and \(l_2\).  So we can match the split and inert places up to obtain an isomorphism
  \[ \prod_{v \colon f_v = f_1} l_1 \otimes_k k_v \cong \prod_{v \colon f_v = f_1} l_2 \otimes_k k_v. \]
Now delete all inertia degrees from the decomposition types of \(p\) in \(l_1\) and \(l_2\) that occur over all \(v\) of \(k\) with \(f_v = f_1\).  Counting the remaining occurrences of \(f_2\) in these decomposition types, we now see that the same number of places \(v\) of \(k\) with \(f_v = f_2\) split and are inert in \(l_1\) and \(l_2\), resulting in an isomorphism
  \[ \prod_{v \colon f_v = f_2} l_1 \otimes_k k_v \cong \prod_{v \colon f_v = f_2} l_2 \otimes_k k_v. \]
Continuing inductively, we obtain 
  \[ \prod_{i=1}^r \ \prod_{v \colon f_v = f_i} l_1 \otimes_k k_v \,\cong\, \prod_{i=1}^r \ \prod_{v \colon f_v = f_i} l_2 \otimes_k k_v \]
  and this defines an isomorphism \(l_1 \otimes_\Q \Q_p \cong l_2 \otimes_\Q \Q_p\).  By construction, the isomorphism permutes the factors of the outer product of
  \[ l_i \otimes_\Q \Q_p \cong \prod_{v \mid p} \,\prod_{w \mid v} {l_i}_w, \]
    hence it restricts to an isomorphism on the diagonal embedding of \(k \otimes_\Q \Q_p \cong \prod_{v \mid p} k_v\).
\end{proof}

For given locally isomorphic quadratic extensions \(l_1/k\) and \(l_2/k\), it thus only remains to check the condition at the finitely many ramified primes.  Note that locally isomorphic number fields have equal discriminant~\cite{Klingen:similarities}*{Theorem~III.1.4.e), p.\,79}, so the same rational primes ramify in \(l_1\) and \(l_2\).

\begin{example}
  Consider Komatsu's example \(k = \Q(\sqrt[4]{7})\), \(l_1 = \Q(\sqrt[8]{7})\), and \(l_2 = \Q(\sqrt[8]{112})\) from above.  Then \(2\) and \(7\) are the ramified primes and both are totally ramified in \(k\), \(l_1\), and \(l_2\).  Since \(\sqrt{-7} \in \Q_2\), we know that \(\Q_2(\sqrt[4]{7})\) contains \(\pm \frac{\sqrt{-7}}{\sqrt[4]{7}} = \pm i\sqrt[4]{7}\).  Thus \(\Q_2(\sqrt[4]{7})\) is the splitting field of \(x^4-7\) whence Galois over \(\Q_2\).  So the condition on {\'e}tale algebras is automatic.  The local field \(\Q_7(\sqrt[4]{7})\), in turn, is not Galois over \(\Q_7\).  However, we have \(\sqrt{2} \in \Q_7\), hence the required isomorphism can be constructed explicitly as in
  \[
    \begin{tikzcd}[row sep=8mm, column sep=-5mm]
      \Q_7[Y]/(Y^8-7) \arrow[rr, "Y \mapsto Z/\sqrt{2}"] & & \Q_7[Z]/(Z^8-112) \\
      & \Q_7[X]/(X^4-7) \arrow[lu, "X \mapsto Y^2"] \arrow[ru, "X \mapsto Z^2 / 2"']. &
\end{tikzcd}
\]
\end{example}

As the upshot of the discussion thus far, we have the following result.

\begin{proposition}
  For \(k = \Q(\sqrt[4]{7})\), we have that \(\overline{h}([\sqrt[4]{7}]) = \overline{h}([\sqrt[4]{112}])\).  In particular, the map \(\overline{h}\) is not injective.
\end{proposition}

\begin{proof}
  The first statement follows from the discussion above.  The fields \(\Q(\sqrt[8]{7})\) and \(\Q(\sqrt[8]{112})\) are moreover not even isomorphic over \(\Q\) so that they define distinct elements in the domain of \(\overline{h}\).
\end{proof}

Together with Theorem~\ref{thm:inner-twists-of-nontrivial-h-bar}, this shows that groups with the corresponding outer types are never congruence rigid.

\begin{theorem}
  Let \(k=\Q(\sqrt[4]{7})\) and let \(\mathbf{G}\) be any \(k\)-group of type \({}^2 A_n\) for \(n \ge 2\), \({}^2 D_n\) for \(n \ge 5\), or \({}^2 E_6\) whose quasi-split type is given by \(\sqrt[4]{7} \in H^1(k, \Sym \Delta)\).  Then \(\mathbf{G}\) is not congruence rigid.
\end{theorem}

The number fields \(l_1 = \Q(\sqrt[8]{7})\) and \(l_2 = \Q(\sqrt[8]{112})\) in the above examples are not isomorphic.  But in fact, we can also construct quadratic extensions \(l_1/k\) and \(l_2/k\) with \(l_1 \cong_\Q l_2\) and with \(\overline{h}(l_1) = \overline{h}(l_2)\) but such that \(l_1\) and \(l_2\) define different elements in \(k^\times/(k^\times)^2 \,/\, \Aut(k)\).

\begin{example}
  Consider the finite group \(G = A_4 \times C_2\) which can also be written as the transitive wreath product \(G = C_2^3 \rtimes C_3\).  Let
  \[ U = C_2 \times C_2 \times \{0\} \le C_2^3 \le G, \]
  let \(V_1 = C_2 \times \{0\} \times \{0\} \le U\) and let \(V_2 = \{0\} \times C_2 \times \{0\} \le U\).  Note that \(V_2 = g V_1 g^{-1}\) for some \(g \in G\) but no such \(g\) conjugates \(U\) to itself.

  Since the inverse Galois problem is known for solvable groups, we have a Galois extension \(K/\Q\) with \(\operatorname{Gal}(K/\Q) \cong G\). Let \(k\), \(l_1\), and \(l_2\) be the fixed fields of \(U\), \(V_1\), and \(V_2\), respectively.  By the above, we have \(\sigma \in \operatorname{Gal}(K/\Q)\) with \(\sigma(l_1) = l_2\).  In particular, this implies \(\A^f_{l_1} \cong_{\A_\Q} \A^f_{l_2}\) and depending on the choice of \(K\), we will see in a moment that it is possible to arrange also at ramified primes that such an isomorphism carries \(\A^f_k\) to itself.  But any isomorphism \(l_1 \cong l_2\) extends to an automorphism of \(K\) conjugating \(V_1\) to \(V_2\).  So it cannot normalize \(U\), meaning it cannot map \(k\) to itself.  Hence \(l_1 \neq l_2\) in \(k^\times/(k^\times)^2 \,/\, \Aut(k)\) but \(\overline{h}(l_1) = \overline{h}(l_2)\).

  Using the \emph{L-functions and modular forms database} (LMFDB)~\cite{LMFDB}, we found that the splitting field \(K\) of the polynomial
  \[ x^6 - x^5 + 14x^4 - 148x^3 + 1241x^2 - 3647x + 5023 \]
  has Galois group isomorphic to \(A_4 \times C_2\) and has discriminant \(163^{20}\).  The number field \(k\) that this polynomial defines is the fixed field corresponding to \(U\).  The prime \(p = 163\) is totally ramified in \(k\).  The corresponding local field extension \(k \otimes_\Q \Q_p / \Q_p\) is Galois and cyclic.  Hence regardless of whether \(l_i \otimes_\Q \Q_p\) is a field or not, there always exists an isomorphism \(l_1 \otimes_\Q \Q_p \cong l_2 \otimes_\Q \Q_p\) that restricts to an automorphism of the embedded \(k \otimes_\Q \Q_p\).  Together with Proposition~\ref{prop:unramified}, we thus obtain an isomorphism \(\A^f_{l_1} \cong \A^f_{l_2}\) that restricts to an automorphism of \(\A^f_k\).
\end{example}

In contrast, we have the following.

\begin{theorem} \label{thm:h-bar-injective}
  Let \(k/\Q\) be Galois. Then \(\overline{h}\) is injective.
\end{theorem}

We prepare the proof with a group theoretical consideration.  Recall that two subgroups \(U_1\) and \(U_2\) of a finite group \(G\) are called \emph{almost conjugate} if \(U_1\) intersects each conjugacy class of \(G\) in the same number of elements as \(U_2\) does.  Clearly, conjugate subgroups are almost conjugate.  The stabilizer \(U_1\) of a point and the stabilizer \(U_2\) of a line for the action of \(G = \PSL(3,2)\) on the Fano plane \(\mathbb{P}^2(\mathbb{F}_2)\) provide an example of almost conjugate subgroups that are not conjugate.  We give a criterion to exclude the occurrence of non-conjugate almost conjugate subgroups.

\begin{proposition} \label{prop:almost-conjugate}
  Let \(G\) be a finite group and let \(U_1, U_2 \le G\) be almost conjugate.  Suppose there exists a normal subgroup \(N \trianglelefteq G\) containing both \(U_1\) and \(U_2\) as subgroups of index two.  Then \(U_1\) is conjugate to \(U_2\).
\end{proposition}

\begin{proof}

  For \(i=1,2\), let \(1_{U_i}\) be the trivial character on \(U_i\).  Since \(N\) is normal in \(G\) and \(U_i \le N\), we have \(N \backslash G / U_i \cong N \backslash G\) and \(gU_ig^{-1} \cap N = g U_i g^{-1}\), so the Mackey restriction formula~\cite{Serre:representations}*{Proposition~22, p.\,58} shows that
  \[ \operatorname{Res}^G_N \operatorname{Ind}^G_{U_i} \,1_{U_i} = \sum_{Ng \in N \backslash G} \operatorname{Ind}^N_{gU_ig^{-1}} \,1_{gU_ig^{-1}}. \]
  Since \([N : gU_ig^{-1}] = 2\), we have
  \[ \operatorname{Ind}^N_{gU_ig^{-1}} \,1_{gU_ig^{-1}} = 1_N + \chi_{g U_i g^{-1}} \]
  where \(\chi_{g U_i g^{-1}} \colon N \longrightarrow \{ \pm 1 \}\) is the unique linear character with kernel \(g U_i g^{-1}\).  Thus
  \[ \operatorname{Res}^G_N \operatorname{Ind}^G_{U_i} \,1_{U_i} = [G : N] 1_N + \sum_{Ng \in N \backslash G} \chi_{gU_ig^{-1}}. \]
    By~\cite{Klingen:similarities}*{Theorem~III.1.3, p.\,77}, the almost conjugacy of \(U_1\) and \(U_2\) has the equivalent characterization that we have \(\operatorname{Ind}^G_{U_1} 1_{U_1} = \operatorname{Ind}^G_{U_2} 1_{U_2}\).  Restricting both sides to \(N\), we thus obtain
    \[ \sum_{Ng \in N \backslash G} \chi_{gU_1g^{-1}} = \sum_{Ng \in N \backslash G} \chi_{gU_2g^{-1}}. \]
    But the irreducible characters of \(N\) form a basis of the vector space of all class functions on \(N\) and each \(\chi_{gU_ig^{-1}}\), being linear, is irreducible.  As \(\chi_{U_2}\) occurs in the right sum, it thus must occur in the left sum as well.  So there exists \(g \in G\) such that \(\chi_{gU_1g^{-1}} = \chi_{U_2}\).  But equal characters have equal kernels, so we conclude \(gU_1g^{-1} = U_2\).
\end{proof}

\begin{proof}[Proof of Theorem~\ref{thm:h-bar-injective}.]
  Let \(l_1/k\) and \(l_2/k\) be two quadratic extensions such that \(\overline{h}(l_1) = \overline{h}(l_2)\).  As we discussed, we then have \(\A^f_{l_1} \cong_{\A^f_\Q} \A^f_{l_2}\) so that \(l_1\) and \(l_2\) are in particular arithmetically equivalent.  Hence they have the same Galois closure \cite{Klingen:similarities}*{Theorem~III.1.4.b), p.\,79} which we denote by \(K\).  Set \(G = \operatorname{Gal}(K/\Q)\), let \(U_i \le G\) be the subgroup fixing \(l_i \subset K\), and let \(N \trianglelefteq G\) be the normal subgroup fixing the Galois extension \(k/\Q\).  Since \(l_i/k\) is a quadratic extension, we have \([N : U_i] = 2\) and since \(l_1\) is arithmetically equivalent to \(l_2\), we have that \(U_1\) is almost conjugate to \(U_2\) by~\cite{Klingen:similarities}*{Theorem~III.1.3, p.\,77}.  We have thus established the situation of Proposition~\ref{prop:almost-conjugate}.  We conclude that \(U_2 = \sigma U_1 \sigma^{-1}\) for some \(\sigma \in G\), so \(\sigma(l_1) = l_2\).  Since \(k/\Q\) is Galois, \(\sigma\) restricts to an automorphism of \(k\).
\end{proof}

\section{Reduction to inner forms}
\label{section:inner-forms}

In this section, and as a complement to \Cref{thm:inner-twists-of-nontrivial-h-bar}, we will see that \(\overline{g}\) has trivial fiber at \(\beta_1\) if and only if both \(\overline{h}\) has trivial fiber at the image of \(\beta_1\) and the map
\begin{align*}
  \overline{f_1}\colon & H^0(k,\Sym \Delta)
    \setminus H^1(k, \Ad \mathbf{G_1}) / \Aut(k)_{\gamma_1} \\
  & \longrightarrow \prod_{v \nmid \infty} \left(H^0(k_v, \Sym \Delta)
    \setminus H^1(k_v, \Ad \mathbf{G_1})\right)\
    \big/ \Aut(\A_k^f)_{\gamma_1'}
\end{align*}
has trivial fiber at a preimage of \(\beta_1\).  The precise statemant is given in the following theorem.

\begin{theorem} \label{thm:f-g-equivalence}
  Let \(\beta_1 \in H^1(k, \Aut \mathbf{G_0})\) be an arbitrary \(k\)-form of a split \(k\)-group \(\mathbf{G_0}\) not of type \(D_4\).  Denote the \(k\)-quasi-split form of \(\beta_1\) by \(\mathbf{G_1}\).  Let \(\gamma_1 \in H^1(k, \Sym \Delta)\) be the image of \(\beta_1\) under
  \[ p\colon H^1(k, \Aut \mathbf{G_0}) \rightarrow H^1(k, \Sym \Delta) \]
  and assume that \([\gamma_1] \in H^1(k, \Sym \Delta)/\Aut(k)\) has trivial fiber under
\[\overline{h}\colon H^1(k, \Sym \Delta) / \Aut(k) \longrightarrow \prod_{v \nmid \infty} H^1(k_v, \Sym \Delta)\ \big/ \Aut(\A_k^f).\]
Let further \(\gamma_1' \coloneqq h(\gamma_1)\) and denote by \(\Aut(k)_{\gamma_1}\) and \(\Aut(\A_k^f)_{\gamma_1'}\) the respective stabilizers.  Then \([\beta_1] \in H^1(k, \Aut \mathbf{G_0})/\Aut(k)\) has trivial fiber under
\[\overline{g}\colon H^1(k, \Aut \mathbf{G_0}) / \Aut(k) \longrightarrow \prod_{v \nmid \infty} H^1(k_v, \Aut \mathbf{G_0})\ \big/ \Aut(\A_k^f)\]
if and only if any preimage \(\alpha_1\) of (the translate under twisting of) \(\beta_1\) under \(i_1\colon H^1(k, \Ad \mathbf{G_1})\to H^1(k, \Aut \mathbf{G_1})\) has trivial fiber under
\begin{align*}
  \overline{f_1}\colon & H^0(k,\Sym \Delta)
    \setminus H^1(k, \Ad \mathbf{G_1}) / \Aut(k)_{\gamma_1} \\
  & \longrightarrow \prod_{v \nmid \infty} \left(H^0(k_v, \Sym \Delta)
    \setminus H^1(k_v, \Ad \mathbf{G_1})\right)\
    \big/ \Aut(\A_k^f)_{\gamma_1'}.
\end{align*}
\end{theorem}

Firstly, we would like to shed some light on the nature of the two-sided quotient appearing in the theorem.

\begin{proposition} \label{prop:two-sided-quotient}
  \begin{enumerate}[label=(\roman*)]
    \item \label{item:well-defined} The action of \(\Aut(k)_{\gamma_1}\) on \(H^1(k, \Ad \mathbf{G_1})\) and the action of \(\Aut(\A_k^f)_{\gamma_1'}\) on \(H^1(k_v, \Ad \mathbf{G_1})\) each given by the usual functorial operation are both well-defined.
    \item \label{item:compatible} For any \(\Phi \in \Aut(\A_k^f)_{\gamma_1'}\) and \(\Sigma \in \prod_{v \nmid \infty} H^0(k_v, \Sym \Delta)\), there is a \(\Sigma' \in \prod_{v \nmid \infty} H^0(k_v, \Sym \Delta)\), such that
      \[(\Sigma . \omega) . \Phi = \Sigma' . (\omega . \Phi)\]    
    holds for all \(\omega \in \prod_{v \nmid \infty} H^1(k_v, \Ad \mathbf{G_1})\).  Furthermore, the actions of \(H^0(k,\Sym \Delta)\) and \(\Aut(k)_{\gamma_1}\) commute.
  \end{enumerate}
\end{proposition}

\begin{proof}
  The second part of statement~\ref{item:well-defined} follows from the fact that the functorial action of \(\Phi_v\colon k_v\to k_w\) maps \(H^1(k_v, \Ad \mathbf{G_{(\gamma_1')_v}})\) to
  \[H^1(k_w, \Ad \mathbf{G_{\Phi_v((\gamma_1')_v)}})
    = H^1(k_w, \Ad \mathbf{G_{(\gamma_1')_w}})\]
  as seen in the proof of \Cref{thm:H^1-decomposition}.  The first part follows from an analogous argument for number fields.
  
  The action of \(\Sigma\) is coordinatewise given by some involution that is equivariant with respect to the local isomorphisms \(\Phi_v\colon k_v\to k_w\).  This can be seen by noting that the action of sigma is either the inversion or the interchanging of two (with respect to \(\Phi_v\) canonically identified) factors.  The first part of statement~\ref{item:compatible} therefore follows by choosing \(\Sigma'\) to be the corresponding (trivial or nontrivial) action of \(H^0(k_w, \Sym \Delta)\) at the respective codomain places \(w\) of \(\Phi_v\).  The second part of statement~\ref{item:compatible} follows by observing the action of \(H^0(k, \Sym \Delta)\) on the codomain of \(f_1\) by diagonally embedding its nontrivial element and the utilising the first part of ~\ref{item:compatible}.
\end{proof}

In particular \Cref{prop:two-sided-quotient} means that the two-sided quotient in \Cref{thm:f-g-equivalence} is well-defined.

\begin{proof}[Proof of \Cref{thm:f-g-equivalence}.]
  As all preimages of \(\beta_1\) under \(i_1\) lie in the same orbit under the action of \(H^0(k,\Sym \Delta)\) on \(H^1(k, \Ad \mathbf{G_1})\), the condition of the theorem on \(\overline{f_1}\) holding for one preimage is equivalent to the condition holding for all preimages.  First assume \(\overline{g}\) to have trivial fiber at \([\beta_1]\).  Let \(\alpha_1\) be a preimage of \(\beta_1\) and let \(\alpha_2 \in H^1(k, \Ad \mathbf{G_1})\), \(\Sigma \in \prod_{v \nmid \infty} H^0(k_v, \Sym \Delta)\) and \(\Phi \in \Aut(\A_k^f)_{\gamma_1'}\) such that \(\Sigma(\Phi(f(\alpha_1))) = f(\alpha_2)\).  As previously, let \(\alpha_1' \coloneqq f_1(\alpha_1)\).  Because \(I_1 \circ \Sigma = I_1\) holds, we obtain
  \begin{align*}
  \Phi(g(\beta_1))
    & = \Phi(g(i_1(\alpha_1))) = \Phi(I_1(f_1(\alpha_1)))
      = I_1(\Phi(f_1(\alpha_1))) \\
    & = I_1(\Sigma(\Phi(f_1(\alpha_1))) = I_1(f_1(\alpha_2)) = g(i_1(\alpha_2))
      = g(\beta_2),
  \end{align*}
  so \([\beta_2]\) is in the same fiber of \(\overline{g}\) as \([\beta_1]\).  As we assumed \(\overline{g}\) to have trivial fiber at  \([\beta_1]\), this implies \([\beta_1] = [\beta_2]\), so there has to be some \(\phi \in \Aut(k)\) such that \(\phi(\beta_1) = \beta_2\).  As both \(\beta_1\) and \(\beta_2\) are in the image of \(i_1\colon H^1(k, \Ad \mathbf{G_1})\to H^1(k, \Aut \mathbf{G_1})\), they must be mapped to the same image \(\gamma_1\) under \(p\), so
  \[\phi(\gamma_1) = \phi(p(\beta_1)) = p(\phi(\beta_1))
    = p(\beta_2) = \gamma_1.\]
  Thus \(\phi \in \Aut(k)_{\gamma_1}\) and
  \[i_1(\phi(\alpha_1)) = \phi(i_1(\alpha_1)) = \phi(\beta_1) = \beta_2
    = i_1(\alpha_2),\]
  so there is a \(\sigma \in H^0(k, \Sym \Delta)\) such that \(\sigma(\phi(\alpha_1)) = \alpha_2\), so the fiber of \(\alpha_1\) under \(\overline{f_1}\) is trivial.
  
  Assume now that a preimage \(\alpha_1\) of \(\beta_1\) has trivial fiber under \(\overline{f_1}\).  Let \(\Phi \in \Aut(\A_k^f)\) and \(\beta_2 \in H^1(k, \Aut \mathbf{G_0})\) such that \(\Phi(g(\beta_1)) = g(\beta_2)\).  Defining \(\gamma_2 \coloneqq p(\beta_2)\), we see
  \begin{align*}
    \Phi(h(\gamma_1)) & = \Phi(h(p(\beta_1))) = \Phi(P(g(\beta_1)))
    = P(\Phi(g(\beta_1))) = P(g(\beta_2)) \\
    & = h(p(\beta_2)) = h(\gamma_2),
  \end{align*}
  so there has to exist a \(\psi \in \Aut(k)\) such that \(\psi(\gamma_1) = \gamma_2\), as \(\overline{h}\) was assumed to have trivial fiber at \(\gamma_1\).  Let \(\Psi \in \Aut(\A_k^f)\) be the image of said \(\psi\) and define \(\widetilde{\Phi} \coloneqq \Psi^{-1} \circ \Phi \in \Aut(\A_k^f)_{\gamma_1'}\).  Choose a preimage \(\alpha_1\) of \(\beta_1\) and \(\alpha_2\) of \(\psi^{-1}(\beta_2)\) under \(i_1\), which is possible, as
  \[p(\psi^{-1}(\beta_2)) = \psi^{-1}(p(\beta_2)) = \psi^{-1}(\gamma_2)
    = \gamma_1 = p(\beta_1).\]
  Now
  \begin{align*}
    I_1(\widetilde{\Phi}(f_1(\alpha_1)))
      & = \widetilde{\Phi}(I_1(f_1(\alpha_1)))
        = \widetilde{\Phi}(g(i_1(\alpha_1)))
        = \Psi^{-1} \circ \Phi(g(\beta_1)) \\
      & = \Psi^{-1}(g(\beta_2))
        = g(\psi^{-1}(\beta_2))
        = g(i_1(\alpha_2))
        = I_1(f_1(\alpha_2)),
  \end{align*}
  so there exists a \(\Sigma \in \prod_{v \nmid \infty} H^0(k_v, \Sym \Delta)\) such that \(\Sigma(\widetilde{\Phi}(f_1(\alpha_1))) = f_1(\alpha_2)\).  As \(\overline{f}\) was assumed to have trivial fiber at \(\alpha_1\), there exist \(\sigma \in H^0(k, \Sym \Delta)\) and \(\phi \in \Aut(k)_{\gamma_1}\) such that \(\sigma(\phi(\alpha_1)) = \alpha_2\).  Thus,
  \[\phi(\beta_1) = \phi(i_1(\alpha_1)) = i_1(\phi(\alpha_1))
    = i_1(\sigma(\phi(\alpha_1))) = i_1(\alpha_2) = \psi^{-1}(\beta_2),\]
    so \(\beta_2 = \psi \circ \phi(\beta_1)\) and the fiber of \(\overline{g}\) at \(\beta_1\) is trivial.
  \end{proof}

  \section{Real Galois cohomology}
  \label{section:real-galois}

  In this section, we demonstrate how triviality of the fiber of \(\overline{f_1}\) again decomposes into two separate statements: the triviality of the corresponding fiber of the map \(\overline{d_1}\) and the non-existence of certain inner twists at real places.

To give an exact description of when \(\overline{f_1}\) has trivial fiber at the class of some element \(\alpha \in H^1(k, \Ad 
\mathbf{G_1})\), we have to inspect the diagram
\begin{equation} \label{eq:second-funamental-extended} \begin{tikzcd}
    \bigoplus\limits_v H^1(k_v, \mathbf{G_1}) \ar[r, "Q_1"]
      & \bigoplus\limits_v H^1(k_v, \Ad \mathbf{G_1}) \ar[r, "\bigoplus (\delta_1)_v"]
      & \bigoplus\limits_v H^2(k_v, Z(\mathbf{G_1})) \\
    H^1(k, \mathbf{G_1}) \ar[u, "\underline{e_1}"] \ar[r, "q_1"]
      & H^1(k, \Ad \mathbf{G_1}) \ar[u, "\underline{f_1}"] \ar[r, "\delta_1"]
      & H^2(k, Z(\mathbf{G_1})) \ar[u, "\underline{d_1}"],
  \end{tikzcd} \end{equation}
that extends Diagram~\eqref{eq:ad-center-square} by mapping to all finite and infinite places and ommitting the quotients by \(S\) and \(\prod S_v\).  If we are now interested in the fiber of some \(\alpha \in H^1(k, \Ad \mathbf{G_1})\)  under \(\delta_1\), we can again twist the first two columns by a cocycle in the class \(\alpha\).  If \(i_1(\alpha) \in H^1(k, \Aut \mathbf{G_1})\) originally corresponded to a \(k\)-form \(\mathbf{G}\), we obtain
\[
  \begin{tikzcd}
    \bigoplus\limits_{v} H^1(k_v, \mathbf{G}) \arrow[r, "Q"]
      & \bigoplus\limits_{v} H^1(k_v, \Ad \mathbf{G})\\
    H^1(k, \mathbf{G}) \arrow[r, "q"] \arrow[u, "\underline{e_2}"]
      & H^1(k, \Ad \mathbf{G}). \arrow[u, "\underline{f_2}"]
\end{tikzcd}\]

Here \(\underline{e_2}\) is bijective by~\cite{Platonov-Rapinchuk:algebraic-groups}*{Theorem~6.6} and \(\underline{f_2}\) is injective due to~\cite{Platonov-Rapinchuk:algebraic-groups}*{Theorem~6.22} combined with a twisting argument.  Furthermore, we know that \(H^1(k_v, \mathbf{G_1})\) vanishes for finite \(v\), so only the \(Q_v\) for real \(v\) can have a nontrivial image.

Adams and Ta{\"\i}bi provide a detailed description of the map
\[Q_{\R}: H^1(\R, \mathbf{G}) \rightarrow H^1(\R, \Ad(\mathbf{G}))\]
depending on the real form of \(\mathbf{G}\) in~\cite{Adams-Taibi:galois-cohomology}*{Corollary~9.3}: The order of the kernel of the map can be computed by
\[|H^1(\R, Z(\mathbf{G}))| / |\pi_0(\Ad(\mathbf{G})(\R))|.\]
As they give a complete list of \(|\pi_0(\Ad(\mathbf{G})(\R))|\) as well as for \(|H^1(\R, \mathbf{G})|\) for all simply connected simple \(\mathbf{G}\), we can conclude nontrivial image of \(Q_w\) for real places \(w\) by computation of \(|H^1(\R, Z(\mathbf{G}))|\).  As a first step, we list in Table~\ref{table:real-groups} the simply connected forms, adjoint forms, and centers of \(\R\)-groups.

\begin{table}[htb]
\begin{center}
  \begin{tabular}{|c|c|c|c|} 
    \hline
    Type & \(\mathbf{G}\) & \(\Ad(\mathbf{G})\) & \(Z(\mathbf{G})\) \\
    \hline\hline
    \multirow{2}{*}{\({}^1 A_n\)} & \(\SL(n+1, \R)\) & \(\PSL(n+1, \R)\)
      & \multirow{2}{*}{\(\mu_{n+1}\)} \\
      & \(\SL(\frac{n+1}{2}, \mathbb{H})\) & \(\PSL(\frac{n+1}{2}, \mathbb{H})\) & \\
    \hline
    \multirow{2}{*}{\({}^2 A_n\)} & \(\SU(r,s)\) & \(\PSU(r,s)\) & \multirow{2}{*}{\(\mathbf{R}^{(1)}_{\mathbb{C}/\R}(\mu_{n+1})\)} \\
      & \(r+s = n+1\) & \(r+s = n+1\) & \\
    \hline
    \multirow{2}{*}{\(B_n\)} & \(\Spin(r,s)\) & \(\PSO(r,s)\) & \multirow{2}{*}{\(\mu_2\)} \\
     & \(r+s = 2n + 1\) & \(r+s = 2n + 1\) & \\
    \hline
    \multirow{3}{*}{\(C_n\)} & \(\Sp(2n, \R)\)
      & \(\PSp(2n, \R)\) & \multirow{3}{*}{\(\mu_2\)} \\
      & \(\Sp(r,s)\) & \(\PSp(r,s)\) & \\
      & \(r+s = n\) & \(r+s = n\) & \\
    \hline
    \multirow{5}{*}{\({}^1 D_n\)} & \(\Spin(r,s)\) & \(\PSO(r,s)\)
      & \((\mu_2)^2\) if \(n\) even \\
      & \(r+s=2n\) & \(r+s=2n\) & \(\mu_4\) if \(n\) odd \\
      & \(r \equiv n \mod 2\) & \(r \equiv n \mod 2\) & \\
      \cline{2-4}
      & \(\Spin^*(2n)\) & \(\PSO^*(2n)\) & \multirow{2}{*}{\((\mu_2)^2\)} \\
      & \(n\) even & \(n\) even & \\
    \hline
    \multirow{5}{*}{\({}^2 D_n\)} & \(\Spin(r,s)\) & \(\PSO(r,s)\)
      & \(R_{\C/\R}(\mu_2)\) if \(n\) even \\
      & \(r+s=2n\) & \(r+s=2n\)
        & \(\mathbf{R}^{(1)}_{\C/\R}(\mu_4)\) if \(n\) odd \\
      & \(r \not\equiv n \mod 2\) & \(r \not\equiv n \mod 2\) & \\
      \cline{2-4}
      & \(\Spin^*(2n)\) & \(\PSO^*(2n)\) & \multirow{2}{*}{\(\mathbf{R}^{(1)}_{\C/\R}(\mu_4)\)} \\
      & \(n\) odd & \(n\) odd & \\
    \hline
    \multirow{4}{*}{\(E_7\)} & split & & \(\mu_2\) \\
       & quaternionic & & \(\mu_2\) \\
       & hermitian & & \(\mu_2\) \\
       & compact & & \(\mu_2\) \\
    \hline
  \end{tabular}
\caption{Simply connected forms, adjoint forms, and centers of real groups.}
\label{table:real-groups}
\end{center}
\end{table}

The centers in this table are obtained from~\cite{Tits:classification-algebraic}*{Table II} in combination with~\cite{Platonov-Rapinchuk:algebraic-groups}*{p.\,332}. Note that triality does not occur, as \(\C/\R\) has degree two.

\begin{proposition}
  \label{prop:cohomology-orders}
  The first cohomology sets of the above centers have the following orders.
  \begin{align*}
    |H^1(\R, \mu_n)| & = 2 \textit{ if } n \textit{ even} \\
    |H^1(\R, \mu_n)| & = 1 \textit{ if } n \textit{ odd} \\
    |H^1(\R, (\mu_2)^ 2)| & = 4 \\
    |H^1(\R, R_{\C/\R}(\mu_n))| & = 1 \\
    |H^1(\R, \mathbf{R}^{(1)}_{\C/\R}(\mu_n))| & = 2 \textit{ if } n \textit{ even} \\
    |H^1(\R, \mathbf{R}^{(1)}_{\C/\R}(\mu_n))| & = 1 \textit{ if } n \textit{ odd}
  \end{align*}
\end{proposition}

\begin{proof}
  By Kummer theory we have \(H^1(\R, \mu_n) \simeq \R^\times / (\R^\times)^2\), so the first three equations follow.
  
  Next, \(H^1(\C/\R, R_{\C/\R}(\mu_n)) \simeq H^1(\C/\C, \mu_n) \simeq 1\) by the Nakayama Lemma, so the fourth equation follows.
  
  Lastly, \cite{Platonov-Rapinchuk:algebraic-groups}*{Equation~6.30} provides us with the short exact sequence
  \begin{align*}  
  1 & \rightarrow (\mu_n)_K/N_{L/K}((\mu_n)_L) \rightarrow
    H^1(K, \mathbf{R}^{(1)}_{L/K}(\mu_n)) \\
    & \rightarrow \ker(N_{L/K}: L^\times/(L^\times)^n
    \rightarrow K^\times/(K^\times)^n) \rightarrow 1.
  \end{align*}
  As \(\C^\times/(\C^\times)^n \simeq 1\), we know that the right group in this short exact sequence is trivial. Furthermore, \((\mu_n)_{\R}\) is \(\{1, -1\}\) or \(\{1\}\), depending on the parity of \(n\) and \(N_{\C/\R}(\zeta_n) = 1\) for any \(n\)-th root of unity \(\zeta_n\), so the left term in the sequence has two or one elements depending on wether \(n\) is even or odd.  It follows that the same holds for the group in the middle of the sequence, proving the last two equalities.
\end{proof}

The results of Proposition~\ref{prop:cohomology-orders} lead to the first column of Table~\ref{table:real-cohomology} which gives an overview of the relevant real cohomology sets.  The second and third column can be found in~\cite{Adams-Taibi:galois-cohomology}*{Section~10} and are included for convenient reference.  The function \(\delta\) in the table takes values depending on \(r,s \text{ mod } 4\) gathered in \Cref{table:delta}.
\begin{table}[htb]
\begin{center}
  \begin{tabular}{c|cccc}
    & 0 & 1 & 2 & 3 \\
  \hline
  0 & 3 & 2 & 2 & 2 \\
  1 & 2 & 1 & 1 & 0 \\
  2 & 2 & 1 & 0 & 0 \\
  3 & 2 & 0 & 0 & 0
  \end{tabular}
  \caption{Values of \(\delta\).}
  \label{table:delta}
\end{center}
\end{table}

\begin{table}[htb]
\begin{center}
  \begin{tabular}{|c|rl|c|c|} 
    \hline
    \(\mathbf{G}\) & \multicolumn{2}{c|}{\(|H^1(\R, Z(\mathbf{G}))|\)}
      & \(|\pi_0(\Ad(\mathbf{G})(\R))|\) & \(|H^1(\R, \mathbf{G})|\)\\
    \hline\hline
    \multirow{2}{*}{\(\SL(n, \R)\)}
      & \multicolumn{2}{c|}{2 if \(n\) even} & 2 if \(n\) even & \multirow{2}{*}{1} \\
      & \multicolumn{2}{c|}{1 if \(n\) odd} & 1 if \(n\) odd & \\
    \hline
    \(\SL(n, \mathbb{H})\) & \multicolumn{2}{c|}{2} & 1 & 2 \\
    \hline
    \multirow{2}{*}{\(\SU(r,s)\)} 
      & \multicolumn{2}{c|}{2 if \(r+s\) even} & 2 if \(r=s\)
      & \multirow{2}{*}{\(\left \lfloor{\frac{r}{2}}\right \rfloor + \left \lfloor{\frac{s}{2}}\right \rfloor + 1\)} \\
      & \multicolumn{2}{c|}{1 if \(r+s\) odd} & 1 if \(r \neq s\) & \\
    \hline
    \multirow{5}{*}{\(\Spin(r,s)\)} & 4 & if \(4 \mid r+s\), & 1 if \(rs = 0\)
      & \multirow{5}{*}{\(\left \lfloor{\frac{r+s}{4}}\right \rfloor + \delta(r,s)\)} \\
      & & \(r\) even & 1 if \(r\), \(s\) odd & \\
      & 1 & if \(4 \mid r+s\), & and \(r \neq s\) & \\
      & & \(r\) odd & 4 if \(r = s\) even & \\
      & 2 & else & 2 else & \\
    \hline
    \multirow{2}{*}{\(\Spin^*(2n)\)} & \multicolumn{2}{c|}{4 if \(n\) even} & 2 if \(n\) even & \multirow{2}{*}{2} \\
      & \multicolumn{2}{c|}{2 if \(n\) odd} & 1 if \(n\) odd & \\
    \hline
    \(\Sp(2n, \R)\) & \multicolumn{2}{c|}{2} & 2 & 1 \\
    \hline
    \(\Sp(r,s)\) & \multicolumn{2}{c|}{2} & 2 & \(r+s+1\) \\
    \hline
    \(E_7\) split & \multicolumn{2}{c|}{2} & 2 & 2 \\
    \(E_7\) quaternionic & \multicolumn{2}{c|}{2} & 1 & 4 \\
    \(E_7\) hermitian & \multicolumn{2}{c|}{2} & 2 & 2 \\
    \(E_7\) compact & \multicolumn{2}{c|}{2} & 1 & 4 \\
    \hline
  \end{tabular}
  \caption{Cohomology and components of \(\R\)-groups.}
  \label{table:real-cohomology}
\end{center}
\end{table}

Furthermore, in the cases of \(E_6\), \(E_8\), \(F_4\) and \(G_2\), the equality
\[H^1(\R, \mathbf{G}) = H^1(\R, \Ad(\mathbf{G}))\]
holds, so that \(Q_{\R}\) is bijective. As \(H^1(\R, \Ad(\mathbf{G}))\) is coincidentally nontrivial in all of those cases, we can conclude \(Q_{\R}\) having a nontrivial image in all of them.

We insert a small observation on \(\Spin(r,s)\), as the orders of the involved groups are particularly chaotic:

\begin{lemma}
  Let \(r, s \in \mathbb{N}_0\) with \(r+s \geq 12\) and \(\mathbf{G} = \Spin(r,s)\). Then the map \(Q_{\R}: H^1(\R, \mathbf{G}) \rightarrow H^1(\R, \Ad(\mathbf{G}))\) has nontrivial image.
\end{lemma}

\begin{proof}
  From the table above we see immediately that \(|H^1(\R, \mathbf{G})| \geq 3\). The only way for \(Q_{\R}\) to have trivial image in that case would be for its kernel to have order \(\geq 3\). This can only occur in the case of \(|H^1(\R, Z(\mathbf{G}))|/|\pi_0(\Ad(\mathbf{G})(\R))| = 4\), which only happens for \(r, s\) even and \(r + s \equiv 0 \mod 4\) and additionally \(rs = 0\) (as the case of \(r, s\) odd and \(r \neq s\) can obviously not occur). This however means that (assuming without loss of generality \(r \geq s\)) \(r \equiv 0 \mod 4\), \(s = 0\) and therefore \(\delta(r,s) = 3\). So
  \[|H^1(\R, \mathbf{G})| \geq 6 > 4 \geq |\ker(Q_{\R})|,\]
  implying nontrivial image.
\end{proof}

All of the observations combined lead to:

\begin{theorem} \label{thm:map-real-groups}
  Let \(\mathbf{G}\) be a simply connected absolutely almost simple algebraic group over \(\R\). Then the map
  \[Q_{\R}: H^1(\R, \mathbf{G}) \rightarrow
  H^1(\R, \Ad \mathbf{G})\]
  has trivial image if and only if \(\mathbf{G}\) is of one of the following groups:
  
  \begin{itemize}
    \item \(\SL(n, \R)\) for \(n \geq 2\)
    \item \(\SL(n, \mathbb{H})\) (sometimes also denoted by \(\SU^*(2n)\)) for \(n \geq 1\)
    \item \(\SU(2,0)\), \(\SU(1,1)\), \(\SU(3,1)\)
    \item \(\Spin(3,2)\), \(\Spin(6,2)\), \(\Spin(7,3)\)
    \item \(\Spin^*(2n)\) for \(n \geq 4\)
    \item \(\Sp(2n, \R)\) for \(n \geq 2\)
  \end{itemize}
\end{theorem}

This list technically contains duplicates, as over \(\R\) we have the accidental isomorphisms
\begin{align*}
  \SU(2,0) & \cong \SL(1, \mathbb{H}) \\
  \SU(1,1) & \cong \SL(2, \R) \\
  \Spin(3,2) & \cong \Sp(4, \R) \\
  \Spin(6,2) & \cong \Spin^*(8).
\end{align*}
Furthermore, one can manually check that the calculations for \(\Spin(r,s)\) with \(2 \leq r+s \leq 6\) agree with the results for the groups that they are isomorphic to, meaning that
\begin{align*}
  \Spin(3,0) & \cong \SL(1, \mathbb{H}) \\
  \Spin(2,1) & \cong \SL(2, \R) \\
  \Spin(5,1) & \cong \SL(2, \mathbb{H}) \\
  \Spin(3,3) & \cong \SL(4, \R)
\end{align*}
all have trivial image under \(Q_{\R}\). Additionally,
\begin{align*}
  \Spin(2,0) & \cong \SO(2, 0) \\
  \Spin(1,1) & \cong \GL(1, \R) \\
  \Spin(4,0) & \cong \SU(2, 0) \times \SU(2, 0) \\
  \Spin(3,1) & \cong \SL(2, \C) \\
  \Spin(2,2) & \cong \SL(2, \R) \times \SL(2, \R)
\end{align*}
all have trivial image as well, but are not simple over \(\R\).

We need an additional minor detail in preparation for \Cref{thm:f_1-decomposition}:

\begin{proposition} \label{prop:map-real-groups}
  Let \(\mathbf{G}\) be an \(\R\)-group not listed in \Cref{thm:map-real-groups}.  The image of
  \[H^1(\R, \mathbf{G}) \xrightarrow{Q_{\R}} H^1(\R, \mathbf{G}) \xrightarrow{\tau} H^1(\R, \mathbf{G_1}) \twoheadrightarrow H^1(\R, \mathbf{G_1}) / H^0(\R, \Sym \Delta)\]
  is then still nontrivial.
\end{proposition}

\begin{proof}
  By \cite{Adams-Taibi:galois-cohomology}*{Remark~8.2}, the only case in which the Dynkin diagram symmetries act nontrivially on \(H^1(\R, \Ad \mathbf{G_1})\) is in type \(D_{2n}\), where they map the two classes corresponding to \(\PSO^*(4n)\) onto each other and additionally the class of \(\Spin(6,2)\) in the case \(D_4\).  But \(\Spin^*(4n)\) and \(\Spin(6,2)\) are already contained in the list of \Cref{thm:map-real-groups} and thus the statement follows.
\end{proof}

With the previous observations in mind, let us formally decompose the map \(\overline{f_1}\) from \Cref{thm:f-g-equivalence}.  Recall that we can transform the map \(d_1\) in \eqref{diagram:boundary-map} into
\begin{align*}
  \qquad \overline{d_1}\colon & S
    \setminus H^2(k, Z(\mathbf{G_1})) / \Aut(k)_{\gamma_1} \\
  & \longrightarrow \bigoplus_{v \nmid \infty} \left(S_v
    \setminus H^2(k_v, Z(\mathbf{G_1}))\right)\
    \big/ \Aut(\A_k^f)_{\gamma_1'}.
\end{align*}
by taking the quotients by \(S\), \(S_v\), \(\Aut(k)_{\gamma}\) and \(\Aut(\A_k^f)_{\gamma'}\) where
\begin{align*}
  \gamma & = p([\mathbf{G}]) \in H^1(k, \Sym \Delta)\\
  \gamma' & = h(p([\mathbf{G}]))
    \in \prod_{v \nmid \infty} H^1(k_v, \Sym \Delta).
\end{align*}

\begin{theorem} \label{thm:f_1-decomposition}
  The map \(\overline{f_1}\) has trivial fiber at some class
  \[ [\alpha] \in S \setminus H^1(k, \Ad \mathbf{G_1}) / \Aut(k)_{\gamma_1} \]
  if and only if both of the following hold:
  \begin{itemize}
    \item The map \(\overline{d_1}\) has trivial fiber at \([\delta_1(\alpha)]\).
    \item At all real places \(w\), \(\mathbf{G}\) is isomorphic over \(\R\) to one of the groups listed in \Cref{thm:map-real-groups}.
  \end{itemize}
\end{theorem}

\begin{proof}
  Assume first that \(\overline{d_1}\) has nontrivial fiber at \([\delta_1(\alpha_1)]\).  As \(\delta_1\) is surjective, there has to be an \(\alpha_2 \in H^1(k, \Ad \mathbf{G_1})\) such that \(d_1(\delta_1(\alpha_1))\) and \(d_1(\delta_1(\alpha_2))\) are mapped to each other by some elements \(\Sigma \in \prod_{v \nmid \infty} S_v\) and \(\Phi \in \Aut (\A_k^f)_{\gamma'}\) but \(\delta_1(\alpha_1)\) and \(\delta_1(\alpha_2)\) are not mapped to each other by \(S\) and \(\Aut (k)\).  By \(S\)- and \(\Aut(k)_{\gamma}\)-equivariance of \(\delta_1\), the latter also applies to \(\alpha_1\) and \(\alpha_2\).  By \(\prod_{v \nmid \infty} S_v\)- and \(\Aut(\A_k^f)_{\gamma'}\)-equivariance of \(\bigoplus_{v \nmid \infty} (\delta_1)_v\), we meanwhile have \((\Sigma . f_1(\alpha_1)) . \Phi = f_1(\alpha_2)\), so \(\overline{f_1}\) has nontrivial fiber at \([\alpha_1]\).
  
  Assume next that there exists some real place \(w\) where \(\mathbf{G}\) is isomorphic to some \(\R\)-group not contained in the list of \Cref{thm:map-real-groups}. Then there are \(\epsilon_1', \epsilon_2' \in \bigoplus_{v \mid \infty} H^1(k_v, \mathbf{G})\) for which \((\epsilon_1')_v = (\epsilon_2')_v\) for all real places except \(w\), where instead \(Q(\epsilon_1')_w\) and \(Q(\epsilon_2')_w\) are even in different \(S_w\)-orbits by \Cref{prop:map-real-groups}.  Recalling that \(H^1(k_v, \mathbf{G_1})\) is trivial for finite \(v\) and that \(\underline{e_2}\) is bijective, we thus find preimages \(\epsilon_1, \epsilon_2 \in H^1(k, \mathbf{G})\) of \(\epsilon_1'\) and \(\epsilon_2'\) under \(\underline{e_2}\).  We know that \(S\) cannot map \(q(\epsilon_1)\) and \(q(\epsilon_2)\) to each other due to \(S\)-\(S_w\)-equivariance of \(\underline{f_1}\), which we can untilize after twisting back to \(\mathbf{G_1}\).  Furthermore we can also use this twisting process to retrace the action of \(\Aut(k)_{\gamma}\) on \(q(\epsilon_i)\) by using that \(\underline{f_1}\) is \(\Aut (k)_{\gamma_1}\)-\(\Aut (\A_k)_{\underline{\gamma_1'}}\)-equivariant.  Here \(\underline{\gamma_1'}\) is the image of \(\gamma_1\) mapped to \(\prod_v k^\times/(k^\times)^2\) for all places of \(k\) and \(\Aut (\A_k)_{\underline{\gamma_1'}}\) is its stabilizer in \(\Aut (\A_k)\).  This implies that \(q(\epsilon_1)\) and \(q(\epsilon_2)\) also lie in different \(\Aut (k)_{\gamma_1}\)-orbits:  The action by \(\Aut (\A_k)_{\underline{\gamma_1'}}\) on the real places is just given by permutation because \(\Aut (\R)\) is trivial.  As \(\underline{f_1}(q(\epsilon_1)) = Q(\underline{e_1}(\epsilon_1))\) and \(\underline{f_1}(q(\epsilon_2)) = Q(\underline{e_1}(\epsilon_2))\) differ at exactly one real coordinate from each other, \(\Aut (\A_k)_{\underline{\gamma_1'}}\) cannot map them to each other and hence neither can \(\Aut (k)_{\gamma_1}\).  Thus, \(q(\epsilon_1)\) and \(q(\epsilon_2)\) are different elements in \(S \setminus H^1(k, \Ad \mathbf{G_1}) / \Aut(k)_{\gamma_1}\).  But on the other hand, \(f_2 \circ q (\epsilon_i) = Q \circ e_2 (\epsilon_i)\) hold and as \(e_2\) is the trivial map, \(f_2\) thus maps \(q(\epsilon_i)\) to the base point of \(\bigoplus_{v \nmid \infty} H^1(k_v, \Ad \mathbf{G})\).  As \(\alpha\) is the base point of \(H^1(k, \Ad \mathbf{G})\) and \(f_2\) is a map of pointed sets, \([q(\epsilon_i)]\) are distinct elements of \(S \setminus H^1(k, \Ad \mathbf{G_1}) / \Aut (k)\) that lie in the fiber of \([\delta_1(\alpha)]\) under \(\overline{f_1}\).  Hence \(\overline{f_1}\) once more has nontrivial fiber at \([\alpha_1]\).
  
  Assume finally both conditions hold and let \(\alpha_1, \alpha_2 \in H^1(k, \Ad \mathbf{G_1})\), \(\Sigma \in \prod_{v \nmid \infty} S_v\) and \(\Phi \in \Aut (\A_k^f)_{\gamma_1'}\) such that \(f_1(\alpha_1) = (\Sigma . f_1(\alpha_2)) . \Phi\).  Then
  \begin{align*}
    d_1(\delta_1(\alpha_1))
      & = (\bigoplus_{v \nmid \infty} (\delta_1)_v) \circ f_1 (\alpha_1)
        = (\bigoplus_{v \nmid \infty} (\delta_1)_v)
          ((\Sigma . f_1 (\alpha_2)) . \Phi) \\
      & = (\Sigma . (\bigoplus_{v \nmid \infty} (\delta_1)_v)
          \circ f_1 (\alpha_2)) . \Phi
        = (\Sigma . d_1(\delta_1(\alpha_2))) . \Phi.
  \end{align*}
  By triviality of the fiber of \([\delta_1(\alpha_1)]\) under \(\overline{d_1}\), there have to exist \(\sigma \in S\) and \(\phi \in \Aut (k)_{\gamma_1}\) such that \(\delta_1(\alpha_1) = \sigma . \delta_1(\alpha_2) . \phi\).  Exactness of the second row of Diagram~\eqref{eq:second-funamental-extended} implies that after twisting to \(\mathbf{G}\), both \(\alpha_1\) and \(\sigma . \alpha_2 . \phi\) have to lie in the image of \(q\colon H^1(k, \mathbf{G})\to H^1(k, \Ad \mathbf{G})\).  We choose preimages \(\epsilon_1\) and \(\epsilon_2\) and observe that
  \[\underline{f_2}(q(\epsilon_1)) = Q(\underline{e_2}(\epsilon_1)) = Q(\underline{e_2}(\epsilon_2)) = \underline{f_2}(q(\epsilon_2)),\]
  as \(Q\) is the product over the \(Q_w\) for the real places \(w\) of \(k\), which all have trivial image.  But as \(\underline{f_2}\) was injective,
  \[\alpha_1 = q(\epsilon_1) = q(\epsilon_2) = \sigma . \alpha_2 . \phi\]
  holds, so
  \[[\alpha_1] = [\alpha_2] \in S \setminus H^1(k, \Ad \mathbf{G_1}) / \Aut(k)_{\gamma_1}\]
  and \(\overline{f_1}\) has trivial fiber at \([\alpha_1]\).
\end{proof}

\section{Congruence rigidity without diagram symmetries}
\label{section:abceefg}

In this section, we prove our first main result: the characterization of congruence rigidity in types \(A_1\), \(B_{n \ge 3}\), \(C_{n \ge 2}\), \(E_7\), \(E_8\), \(F_4\), and \(G_2\).  As the only quasi-split form is the actual split form, we can drop the index 1 from all maps for this section.
  
We introduce a helpful definition related to the map \(d\) as seen in Diagram~\eqref{eq:ad-center-square} to allow for more concise phrasing in some of the following results.  Let for this purpose as before \(\mathbf{G}\) be a simply connected absolutely almost simple algebraic \(k\)-group.  Let furthermore \(\alpha \in H^1(k, \Ad \mathbf{G_0})\) be the preimage of \(\beta\) under the isomorphism \(i\), let \(\xi \coloneqq \delta(\alpha)\) and let
  \[\omega \coloneqq d(\xi) \in \bigoplus_{v \nmid \infty} H^2(k_v, Z(\mathbf{G_0}))\]
as well as
  \[\underline{\omega} \coloneqq b(\xi) \in \bigoplus_{v \in V(k)} H^2(k_v, Z(\mathbf{G_0})).\]
It follows from \Cref{cor:poitou-compatability} that 
\[\omega . \Aut(k) \subseteq \omega . \Aut(\A_k^f) 
  \subseteq \im d\]
where \(\Aut(k)\) acts on \(\omega\) by the inclusion \(\Aut(k) \hookrightarrow \Aut(\A_k^f)\).

\begin{definition} \label{def:weakly-uniform-no-sym}
  If \(\mathbf{G}\) has no Dynkin diagram symmetries, we call \(\mathbf{G}\) \emph{weakly uniform} if
  \[\omega . \Aut(k) = \omega . \Aut(\A_k^f).\]
\end{definition}

In type \(A_1\), weakly uniform is a weakening of the property locally uniform as introduced in \cite{Bridson-et-al:absolute}*{Definition~4.10} which asserts \(\omega . \Aut(\A_k^f) = \omega\).

\begin{theorem} \label{thm:cong-rigid-no-symmetries}
  Let \(k\) be a locally determined number field and \(\mathbf{G}\) a \(k\)-group.  Assume furthermore that the Dynkin diagram of \(\mathbf{G}\) has no nontrivial symmetries.  Then \(\mathbf{G}\) is congruence rigid if and only if it is weakly uniform and furthermore satisfies one of the following conditions:
  \begin{enumerate}[label=(\roman*)]
    \item \(\mathbf{G}\) is of type \(B_n\) where \(n \geq 3\) or of type \(E_7\), \(E_8\), \(F_4\) or \(G_2\) and \(k\) is totally imaginary.
    \item \(\mathbf{G}\) is of type \(A_1\) or \(C_n\) where \(n \geq 2\) and the field \(k\) has at most one real place.  If such a real place exists and \(\mathbf{G}\) is of type \(C_n\), it furthermore needs to be isomorphic to \(\Sp(2n, \R)\) at that place.
    \item \(\mathbf{G}\) is of type \(A_1\), the field \(k\) has exactly two real places \(w_1\) and \(w_2\), \(\mathbf{G}\) is isomorphic to \(\SL(2, \R)\) and to \(\SL(1, \mathbb{H}) \cong \SU(2)\) once each at the two real places, \(\Aut (k)\) permutes \(w_1\) and \(w_2\) and \(\omega.\Aut(k)_{w_1} = \omega. \Aut(\A_k^f)\) holds.
  \end{enumerate}
\end{theorem}

In the last case \(\Aut(k)_{w_1}\) denotes the stabilizer of one (and therefore both) real places in \(\Aut (k)\).  The condition \(\omega.\Aut(k)_{w_1} = \omega. \Aut(\A_k^f)\) is therefore a slight strengthening of weak uniformity.

\begin{proof}
  As there are no nontrivial Dynkin diagram symmetries, the equality \(\Ad(\mathbf{G_0}) \cong \Aut(\mathbf{G_0})\) holds and \(\overline{g}\) has trivial fiber at the class of \(\mathbf{G}\) if and only if \(\overline{f}\) does.  Hence we need to check if both conditions of \Cref{thm:f_1-decomposition} are satisfied.
  
  As \Cref{thm:map-real-groups} does not list any real groups of type \(B_{n \geq 3}\), \(E_7\), \(E_8\), \(F_4\) or \(G_2\), we conclude that \(k\) cannot have any real places in those cases if we want \(\mathbf{G}\) to be congruence rigid.  In the case of \(C_{n \geq 2}\), we instead conclude that \(\mathbf{G}\) has to be isomorphic to \(\Sp(2n, \R)\) at all real places, if we want to meet the second condition of \Cref{thm:f_1-decomposition}.
  
  The first condition of \Cref{thm:f_1-decomposition} is equivalent to
  \[\xi . \Aut(k)
    = d^{-1}((\omega . \Aut (\A_k^f)) \cap \im(d)).\]
  If \(k\) has no real places, \(d\) is injective and we can equivalently verify if
  \[\omega . \Aut (k) = (\omega . \Aut (\A_k^f)) \cap \im(d)\]
  holds.  But as we have remarked previously, \Cref{cor:poitou-compatability} implies that \((\omega . \Aut (\A_k^f)) \subseteq \im(d)\), so we in turn have to check if
  \[\omega . \Aut (k) = \omega . \Aut (\A_k^f)\]
  is true, which is exactly the definition of weak uniformity of \(\mathbf{G}\).  This completes the proof for all types except \(C_n\) and \(A_1\).
  
  Assume first that \(\mathbf{G}\) is of type \(C_{n \geq 2}\), \(k\) has at least two real places and \(\mathbf{G}\) is isomorphic to \(\Sp(2n, \R)\) at all of them.  Then \(\overline{d}\) no longer has trivial fiber at \([\xi]\):  As \(\Sp(2n, \R)\) is the \(\R\)-split form of type \(C_n\), the real coordinates of \(\underline{\omega}\) are zero for all real places \(w\).  Choose two real places \(w_1\) and \(w_2\) and construct the element \(\theta\) by
  \begin{align*}
    \theta_v & = \underline{\omega}_v \text{ if } v \neq w_i \\
    \theta_v & = \underline{\omega}_v+1 \text{ if } v = w_i.
  \end{align*}
  By \Cref{thm:tate-map}, \(\theta\) has a preimage \(\eta\) under \(b\) and on the other hand \(d\) maps \(\xi\) and \(\theta\) to the same element.  By construction, \(\xi\) and \(\eta\) are not mapped onto each other by \(\Aut (k)\) and \(S\) is trivial, hence \([\xi]\) has nontrivial fiber under \(\overline{d}\).  If \(\mathbf{G}\) is of type \(C_{n \geq 2}\), \(k\) has exactly one real place and \(\mathbf{G}\) is isomorphic to \(\Sp(2n, \R)\) at that place.  Then \Cref{thm:tate-map} tells us that \(d\) is still injective.  As in the case of no real places \([\xi]\) has therefore trivial fiber under \(\overline{d}\) if and only if \(\mathbf{G}\) is weakly uniform.
  
  Assume now that \(\mathbf{G}\) is of type \(A_1\) and \(k\) has exactly one real place.  Then \(\mathbf{G}\) is isomorphic to either \(\SL(2, \R)\) or \(\SL(1, \mathbb{H})\) at the real place, so the second condition of \Cref{thm:f_1-decomposition} is satisfied either way.  As in the case of \(C_n\), the map \(d\) is still injective and the equivalence to weak uniformity follows in the same way as when \(k\) was totally imaginary.  If \(k\) has two real places \(w_1\), \(w_2\) and \(\mathbf{G}_{w_i} \cong \SL(2, \R)\) or \(\mathbf{G}_{w_i} \cong \SL(1, \mathbb{H})\) holds, we can again define
  \begin{align*}
    \theta_v & = \underline{\omega}_v \text{ if } v \neq w_i \\
    \theta_v & = \underline{\omega}_v+1 \text{ if } v = w_i.
  \end{align*}
  \(\Aut (k)\) cannot map \(\xi\) to the preimage of \(\theta\), so the fiber of \([\xi]\) under \(\overline{d}\) is nontrivial.  Similarly, if \(k\) has at least three real places, we can find two of them at which \(\underline{\omega}_{w_i}\) agree.  By replacing those \(\underline{\omega}_{w_i}\) with \(\underline{\omega}_{w_i}+1\), we again obtain elements whose real coordinates can not be mapped to eachother by permutation and the fiber of \(\overline{d}\) is nontrivial.
  
  Assume finally that \(k\) has two real places \(w_1\), \(w_2\) and without loss of generality let \(\mathbf{G}_{w_1} \cong \SL(2, \R)\) and \(\mathbf{G}_{w_2} \cong \SL(1, \mathbb{H})\).  If \(\Aut (k)\) does not interchange the two real places, we can construct \(\theta\) again by selecting
  \begin{align*}
    \theta_v & = \underline{\omega}_v \text{ if } v \neq w_i \\
    \theta_v & = 1 \text{ if } v = w_1 \\
    \theta_v & = 0 \text{ if } v = w_2
  \end{align*}
  and see that its preimage is not in the \(\Aut (k)\)-orbit of \(\xi\).  If \(\Aut (k)\) does interchange \(w_1\) and \(w_2\), we need to ensure that
  \[\underline{\omega} . \Aut (k) = \omega . \Aut (\A_k^f)
    \times \{(0_{w_1}, 1_{w_2}), (1_{w_1}, 0_{w_2})\},\]
  as the latter is (via \(b\)) in bijection with the preimage of \(\Aut (\A_k^f) . \omega\) under \(d\).  This is precisely equivalent to
  \[\omega . \Aut(k)_{w_1} = \omega . \Aut(\A_k^f),\]
  where \(\Aut(k)_{w_1}\) is the stabilizer of one real place.
\end{proof}

If \(\mathbf{G}\) is a \(k\)-split algebraic group, then \(\omega\) is trivial at all coordinates and therefore \(\mathbf{G}\) is locally uniform. If \(\mathbf{G}\) is of type \(C_n\), it is furthermore isomorphic to \(\Sp(2n, \R)\) at all real places and if it is of type \(A_1\), it is isomorphic to \(\SL(2, \R)\) at all real places.  We therefore obtain the statement on split groups not of type \(A_{n \geq 2}\), \(D_n\) or \(E_6\) as given in \cite{Kammeyer-Spitler:chevalley}*{Theorem 2} as a consequence of \Cref{thm:cong-rigid-no-symmetries}.

\section{Weak uniformity}
\label{section:weak-uniformity}

In this section, we generalize the property of being \emph{weakly uniform} to the case that \(\mathbf{G}\) has nontrivial Dynkin diagram symmetries.  In particular, we will make the condition explicit in the various types and give a convenient criterion to achieve it in concrete cases.  The reason to introduce the term will be made clear in the next section.  There, we will see that in most cases, the condition ensures that the map \(\overline{d_1}\) has trivial fiber, similarly to \Cref{thm:cong-rigid-no-symmetries}.

From now on, we assume that \(\mathbf{G}\) has Dynkin diagram symmetry group \(\Z/2\).  Recall that we denote by \(\beta \in H^1(k, \Aut \mathbf{G_0})\) the class of \(\mathbf{G}\), \(p(\beta) \eqqcolon \gamma \in H^1(k, \Sym \Delta)\) and \(\gamma' \coloneqq h(\gamma)\). Furthermore we pick a preimage \(\alpha \in H^1(k, \Ad \mathbf{G_1})\) under \(i_1\), let \(\xi \coloneqq \delta_1(\alpha)\) and let as before
  \[\omega \coloneqq d_1(\xi) \in \bigoplus_{v \nmid \infty} H^2(k_v, Z(\mathbf{G_1}))\]
as well as
  \[\underline{\omega} \coloneqq b(\xi) \in \bigoplus_{v \in V(k)} H^2(k_v, Z(\mathbf{G_1})).\]
Finally we define
  \[S^{\omega} \coloneqq \left\{\Sigma \in \prod_{v \nmid \infty} S_v \, \middle| \, c(\Sigma . \omega) = c(\omega)\right\},\]
the set of all tuples of local involutions that preserve the value of \(\omega\) under the Poitou--Tate map \(c\) in~\eqref{eq:poitou-tate} restricted to all finite places.  Note that though \(S^\omega\) and \(\Aut(\A_k^f)_{\gamma'}\) are uncountable in general, the orbits \(S^\omega.\omega\) and \(S^\omega.\omega. \Aut(\A_k^f)_{\gamma'}\) are finite.

As in the case of no Dynkin diagram symmetries, we have the inclusion
\[S^\omega . \omega . \Aut(\A_k^f)_{\gamma'} \subseteq \im d_1.\]
The inclusion 
\[S . \omega . \Aut(k)_{\gamma}
  \subseteq S^\omega . \omega . \Aut(\A_k^f)_{\gamma'}\]
also holds, except if \(\mathbf{G}\) is of type \(D_{2n}\), \(k\) has real places and some specific further conditions are met.  The precise treatment of that case can be found in the proof of \Cref{thm:cong-rigid-D-n}.

\begin{definition} \label{def:weakly-uniform-sym}
  If \(\mathbf{G}\) has Dynkin diagram symmetry group \(\Z/2\), we call \(\mathbf{G}\) \emph{weakly uniform} if
  \[S . \omega . \Aut(k)_{\gamma} = (S^{\omega} . \omega) . \Aut(\A_k^f)_{\gamma'}.\]
\end{definition}

Recall that \(S\) and \(\Aut(k)_{\gamma}\) commute by \Cref{prop:two-sided-quotient}~\ref{item:compatible}, so there are no parentheses required on the left side of the equation.  Note furthermore that this reduces to the previous definition if \(\Sym \Delta\) is trivial.  Again, the definition says intuitively that \(\mathbf{G}\) has only a few local variations and all of these can be realized globally.  We remark that the set \(S . \omega\) is the image of \(S . \xi\) under \(d_1\), which G.\,Harder calls the Brauer--Witt invariant of \(\mathbf{G}\) in \cite{Harder:Bericht}*{p.\,208}.

\begin{definition} \label{def:inner-twin}
  We say that \(\mathbf{G}\) has an \emph{inner twin} at a finite prime \(v\) if one of the following two equivalent conditions holds true.
  \begin{enumerate}[label=(\roman*)]
    \item The coordinate \(\omega_v\) is not fixed by the action of \(S_v\).
    \item The fiber of \([\mathbf{G}] \in H^1(k_v, \Aut \mathbf{G_1})\) under the map
      \[i_1\colon H^1(k_v, \Ad \mathbf{G_1})\to
        H^1(k_v, \Aut \mathbf{G_1})\]
      is not a singleton.
  \end{enumerate}
\end{definition}

The equivalence of the two conditions follows from \cite{Serre:galois-cohomology}*{I.5.5~Cor.\,2} as described in \Cref{section:exact-sequences}.  With this definition, the orbit \(S^\omega.\omega\) can be descirbed more efficiently as follows.

\begin{proposition} \label{prop:S-omega-characterization}
  Denote by \(\sigma \in S \cong \Z/2\) and \(\sigma_v \in S_v \cong \Z/2\) the nontrivial element in the respective groups.  Let \(V\) be the set of all finite places of \(k\) at which \(\mathbf{G}\) has an inner twin.  Then \(S^{\omega} . \omega\) is in bijection with the set
  \[\left\{U \subseteq V \,\middle|\, \sigma . \left(\sum_{v \in U} c_v(\omega_v)\right)
    = \sum_{v \in U} c_v(\omega_v)\right\}.\]
\end{proposition}

\begin{proof}
  The action of \(\prod_{v \nmid \infty} S_v\) on \(\omega\) is nontrivial exactly on the factors \(v \in V\).  As the \(c_v\) are \(S_v\)-\(S\)-equivariant, for \(\Sigma \in S^{\omega}\) we furthermore have
  \begin{align*}
  c(\Sigma . \omega) & = \sum_{v \nmid \infty} c_v(\Sigma_v . \omega_v) \\
    & = \sum_{v \nmid \infty} \Sigma_v . c_v(\omega_v) \\
    & = \sigma . \left(\sum_{v \in W} c_v(\omega_v)\right)
      + \left(\sum_{v \not\in W} c_v(\omega_v)\right),
  \end{align*}
  where \(W \subseteq V_f(k)\) is the set of places at which \(\Sigma_v = \sigma_v\) holds.  We can deduce
  \[\sigma . \left(\sum_{v \in W} c_v(\omega_v)\right) = \sum_{v \in W} c_v(\omega_v)\]
  and conclude that every element in \(S^{\omega} . \omega\) can be uniquely described by \(U \coloneqq W \cap V\).
\end{proof}

\begin{corollary} \label{cor:S-omega-characterization}
  The set \(V\) introduced in \Cref{prop:S-omega-characterization} contains exactly the finite places \(v\) where \(\mathbf{G}\) becomes inner and \(\omega_v\) is
  \begin{itemize}
    \item nontrivial in type \(A_{2n}\) or \(E_6\),
    \item not equal to 0 or \(n+1\) in type \(A_{2n+1}\),
    \item equal to \((1,0)\) or \((0,1)\) in type \(D_{2n}\),
    \item equal to 1 or 3 in type \(D_{2n+1}\).
  \end{itemize}
  The characterization of \(S^{\omega} . \omega\) in \Cref{prop:S-omega-characterization} can furthermore be simplified as follows, depending on the type of \(\mathbf{G}\):
  \begin{itemize}
    \item Type \({}^1 A_{2n}\) or \(E_6\): \(\{U \subseteq V \,|\, \sum_{v \in U} c_v(\omega_v) = 0\}\)
    \item Type \({}^1 A_{2n+1}\): \(\{U \subseteq V \,|\, \sum_{v \in U} c_v(\omega_v) = 0, n+1\}\)
    \item Type \({}^1 D_{n \geq 5}\): \(\{U \subseteq V \,|\, |U| \text{ even}\}\)
    \item Outer type: \(\{U \subseteq V\}\)
  \end{itemize}
\end{corollary}

\begin{proof}
  The first part follows immediately from \Cref{rmk:dynkin-actions}.  The second part is then just careful inspection of the individual cases: In the types \({}^1 A_n\) and \({}^1 E_6\) we cannot simplify the statement significantly.  Meanwhile any sum of two \(\omega_v\) not fixed under \(S_v\) is mapped to an element invariant under \(S\) in both \({}^1 D_{2n}\) and \({}^1 D_{2n+1}\).  Finally, any singular \(\omega_v\) not fixed under \(S_v\) already satisfies this condition in all of the outer types.
\end{proof}

Weak uniformity can be checked by investigating the actions of \(S^{\omega}\) and \(\Aut(\A_k^f)_{\gamma'}\) separately, turning them into more manageable combinatorial problems.

\begin{proposition}
  If \(\omega . \Aut(k)_{\gamma} = \omega . \Aut(\A_k^f)_{\gamma'}\) and \(S . \omega = S^{\omega} . \omega\) hold, then \(\mathbf{G}\) is weakly uniform.
\end{proposition}
  
\begin{proof}
  Let \((\Sigma . \omega) . \Phi \in S^{\omega} . \omega . \Aut(\A_k^f)_{\gamma'}\) with \(\Sigma \in S^{\omega}\), \(\Phi \in \Aut(\A_k^f)_{\gamma'}\).  Then there is a \(\sigma \in S\) such that \(\sigma . \omega = \Sigma . \omega\) and a \(\phi \in \Aut(k)_{\gamma}\) such that \(\omega . \phi = \omega . \Phi\).  But then
  \[\sigma . \omega . \phi = \sigma . (\omega . \Phi)
    = (\sigma . \omega) . \Phi = (\Sigma . \omega) . \Phi,\]
  as the diagonal action of \(S\) commutes with the action of \(\Aut(\A_k^f)_{\gamma'}\) by \Cref{prop:two-sided-quotient}~\ref{item:compatible}.
\end{proof}

However, neither condition in the proposition is necessary for weak uniformity.  To show this, we construct suitable conterexamples.

First, let \(k = \Q(i)\) and choose \(\mathbf{G_1}\) to be the \(k\)-split group of type \(A_4\).  We construct an inner twist \(\mathbf{G}\) of \(\mathbf{G_1}\) by choosing \(\omega \in \bigoplus H^2(k_v, Z(\mathbf{G_1}))\), where all \(H^2(k_v, Z(\mathbf{G_1})) \cong \Z/5\), as no real places exist.  Pick coordinates \((2,3)\) at the two finite places over the split prime 5, coordinate 1 at the non-split prime 3 and coordinate 4 at the non-split prime 7.  Choosing all other coordinates to be trivial, this \(\omega\) clearly satisfies the Tate condition from Sequence~\eqref{eq:poitou-tate} and therefore lies in the image of \(b = d_1\).  Note that \(\Aut(\A_k^f)_{\gamma'} = \Aut(\A_k^f)\), since \([\mathbf{G_1}] = [1] \in H^1(k, \Sym \Delta)\).

\begin{table}[htb]
  \begin{tabular}{r|lcl}
    & 3 & 5 & 7 \\ \hline
  \(\omega\) & 1 & \((2,3)\) & 4 \\
  \(\sigma . \omega\) & 1 & \((3,2)\) & 4 \\
  \(\omega . \phi\) & 4 & \((2,3)\) & 1 \\
  \(\sigma . \omega . \phi\) & 4 & \((3,2)\) & 1 \\
  \end{tabular}
  \caption{\(S . \omega . \Aut(k)_{\gamma} = (S^{\omega} . \omega) . \Aut(\A_k^f)_{\gamma'} \)}
  \label{table:counterexample-1}
\end{table}

A short computation shows that \(S . \omega . \Aut(k)_{\gamma} = (S^{\omega} . \omega) . \Aut(\A_k^f)_{\gamma'} \) consists of the list given in Table~\ref{table:counterexample-1}, where \(\sigma\) is the nontrivial element of \(S\) and \(\phi\) the compex conjugation in \(\Aut(k)\).

But on the other hand
\[S . \omega = \{(1_3, (2,3)_5, 4_7), (4_3, (3,2)_5, 1_7)\}, \]
while \(S^{\omega} . \omega\) is already the full set displayed in Table~\ref{table:counterexample-1}, so the two are not equal.

Now, let again \(k = \Q(i)\) and choose \(\mathbf{G_1}\) to be the \(k\)-split group of type \(A_5\) this time.  We construct \(\mathbf{G}\) by picking coordinates \((2,4)\) at the two finite places over the split prime 5, coordinates \((0,3)\) at the two finite places over 11 and finally 3 at the non-split prime 3.  Again choosing all other coordinates to be trivial, this \(\omega\) also satisfies the Tate condition from Sequence~\eqref{eq:poitou-tate} and therefore lies in the image of \(b = d_1\).

\begin{table}[htb]
  \begin{tabular}{r|lcc}
    & 3 & 5 & 11 \\ \hline
    \(\omega\) & 3 & \((2,4)\) & \((0,3)\) \\
    \(\sigma . \omega\) & 3 & \((4,2)\) & \((0,3)\) \\
    \(\omega . \phi\) & 3 & \((2,4)\) & \((3,0)\) \\
    \(\sigma . \omega . \phi\) & 3 & \((4,2)\) & \((3,0)\) \\
  \end{tabular}
  \caption{\(S . \omega . \Aut(k)_{\gamma} = (S^{\omega} . \omega) . \Aut(\A_k^f)_{\gamma'} \)}
  \label{table:counterexample-2}
\end{table}

A short computation shows that \(\overline{d_1}\) once more has trivial fiber at \(\omega\), as can be checked in Table~\ref{table:counterexample-2}.  But on the other hand
\[\omega . \Aut(k)_{\gamma} = \{(3_3, (2,4)_5, (0,3)_{11}), (3_3, (4,2)_5, (3,0)_{11})\} \]
while \(\omega . \Aut(\A_k^f)_{\gamma'}\) is already the full set displayed in Table~\ref{table:counterexample-2}, so the two are not equal.

In groups of outer type, however, many of the difficulties described above do not arise.  Instead, they allow for a clearer description of weak uniformity.

\begin{theorem} \label{thm:outer-types-weakly-uniform}
  Let \(\mathbf{G}\) be a \(k\)-group of type \({}^2 A_{n \geq 2}\), \({}^2 D_{n \geq 5}\) or \({}^2 E_6\).  Then \(\mathbf{G}\) is weakly uniform if and only if there is at most one finite place of \(k\) where \(\mathbf{G}\) has an inner twin and furthermore
  \[\omega . \Aut(k)_{\gamma} = \omega . \Aut(\A_k^f)_{\gamma'}\]
  holds.
\end{theorem}

\begin{proof}
  As in the proof of \Cref{cor:S-omega-characterization}, we observe that for any \(\Sigma \in \prod_{v \nmid \infty} S_v\) the equality
  \[\sum_{v \nmid \infty} c_v(\Sigma_v . \omega_v)
    = \sum_{v \nmid \infty} c_v(\omega_v)\]
  holds.  If we again let \(V\) be the set of finite places at which \(\mathbf{G}\) has an inner twin, we see that any subset of \(V\) provides us with another element in \(S^{\omega} . \omega . \Aut(\A_k^f)_{\gamma'}\) by \Cref{cor:S-omega-characterization}.  The action of \(\Aut(\A_k^f)_{\gamma'}\) and therefore also the induced action of \(\Aut(k)_{\gamma}\) leaves the number of places \(v\) at which \(\omega_v\) takes a fixed value \(x\) invariant.
  
  Assume that we have at least two finite places \(v_1\), \(v_2\) at which \(\mathbf{G}\) has an inner twin and let \(x_i \coloneqq \omega_{v_i}\).  Let again \(\sigma \in S\) and \(\sigma_v \in S_v\) be the nontrivial elements respectively.  The elements \(\theta_i \in S^{\omega} . \omega\) given by the four different combinations of applying or not applying \(\sigma_{v_i}\) to \(\omega_{v_i}\) then have at least three different numbers of occurrences of \(x_i\) in \(\omega\) (four if \(x_1 \neq x_2, \sigma_{v_2}.x_2\)).  Hence
  \[S. \omega . \Aut(k)_{\gamma} = \omega . \Aut(k)_{\gamma} \,
    \sqcup \, \sigma . \omega . \Aut(k)_{\gamma}\]
  cannot be equal to \(S^{\omega} . \omega . \Aut(\A_k^f)_{\gamma'}\), as the number of occurrences of the \(x_i\) agree for all elements of \(\omega . \Aut(k)_{\gamma}\) and \(\sigma . \omega . \Aut(k)_{\gamma}\) respectively.  Hence not all of the \(\theta_i\) can lie in \(S. \omega . \Aut(k)_{\gamma}\) and \(\mathbf{G}\) is not weakly uniform.
  
  Assume conversely that there is at most one finite place \(v\) at which \(\mathbf{G}\) has an inner twin.  Then \(S . \omega = S^{\omega} . \omega\), so
  \[S. \omega . \Aut(k)_{\gamma}
    = (S^{\omega} . \omega) . \Aut(\A_k^f)_{\gamma'}\]
  is equivalent to
  \[S. \omega . \Aut(k)_{\gamma}
    = S . \omega . \Aut(\A_k^f)_{\gamma'}\]
  or equivalently
  \[\omega . \Aut(k)_{\gamma}
    = \omega . \Aut(\A_k^f)_{\gamma'}. \qedhere\]
\end{proof}

\section{Congruence rigidity with diagram symmetries}
\label{section:ade}

In this section, we present the characterization of congruence rigidity in types in type \(A_{n \ge 2}\), \(D_{n \ge 5}\), and \(E_6\).  We first make a general statement on congruence rigidity for \(A_{n \ge 2}\), \(D_{n \ge 5}\) and \(E_6\) in the case that \(k\) has no real places.   Afterwards we will address the presence of real places on a type by type basis. 

\begin{theorem} \label{thm:sym-rigidity-no-reals}
  Let \(k\) be a locally determined number field and let \(\mathbf{G}\) be a \(k\)-group.  Assume that the Dynkin diagram of \(\mathbf{G}\) has symmetry group isomorphic to \(\Z/2\) and that \(k\) is totally imaginary.  Then \(\mathbf{G}\) is congruence rigid if and only if both
  \begin{itemize}
    \item the fiber of \([\gamma]\) along \(\overline{h}\) is trivial and
    \item \(\mathbf{G}\) is weakly uniform.
  \end{itemize}
\end{theorem}

\begin{proof}
  Theorems~\ref{thm:inner-twists-of-nontrivial-h-bar}, \ref{thm:f-g-equivalence} and \ref{thm:f_1-decomposition} show that congruence rigidity is equivalent to \([\gamma]\) having trivial fiber along \(\overline{h}\), \(\mathbf{G}\) being isomorphic to one of the groups listed in \Cref{thm:map-real-groups} and \([\xi]\) having trivial fiber along \(\overline{d_1}\).  As \(k\) is totally imaginary, the condition imposed by \Cref{thm:map-real-groups} is empty.  Furthermore \(d_1\) is equal to \(b\) and the latter is injective as seen in \eqref{eq:poitou-tate}.  Thus \([\xi]\) having trivial fiber along \(\overline{d_1}\) is equivalent to
  \[d_1(S . \xi . \Aut(k)_{\gamma}) = \im d_1 \cap
    (\prod_{v \nmid \infty} S_v . \omega . \Aut(\A_k^f)_{\gamma'}).\]
  But \(\Aut(\A_k^f)_{\gamma'}\) acts in a well-defined manner on the image of \(d_1\) by \Cref{cor:poitou-tate-sum-equivariance} and
  \[\im d_1 \cap (\prod_{v \nmid \infty} S_v . \omega) = S^{\omega} . \omega\]
  follows from Sequence~\eqref{eq:poitou-tate}.  So in turn we get that \([\xi]\) having trivial fiber along \(\overline{d_1}\) is equivalent to the statement
  \[S. \omega . \Aut(k)_{\gamma}
    = S . d_1(\xi) . \Aut(k)_{\gamma}
    = (S^{\omega} . \omega) . \Aut(\A_k^f)_{\gamma'}. \qedhere\]
\end{proof}

Recall that \Cref{thm:outer-types-weakly-uniform} provides a very clear description of weak uniformity for groups of outer type.  For groups of inner type \Cref{thm:many-inner-twins} gives a necessary condition for this property.  In most types, being weakly uniform is necessary for congruence rigidity even if real places exist, as the following proposition shows.

\begin{proposition} \label{prop:weakly-uniform-necessary}
  If \(\mathbf{G}\) is of type \(A_{n \ge 2}\), \(D_{2n+1}\) or \(E_6\) and not weakly uniform, then \([\xi]\) has nontrivial fiber along \(\overline{d_1}\) regardless of the number of real places of \(k\).
\end{proposition}

\begin{proof}
  Even if \(d_1\) is not injective,
  \[d_1(S . \xi . \Aut(k)_{\gamma}) = \im d_1 \cap
    (\prod_{v \nmid \infty} S_v . \omega . \Aut(\A_k^f)_{\gamma'})\]
  is still a necessary condition for triviality of the fiber of \([\xi]\) along \(\overline{d_1}\).  We still have
  \[S^{\omega} . \omega \subseteq \im d_1
    \cap (\prod_{v \nmid \infty} S_v . \omega).\]
  The action of \(S_w\) on \(H^2(k_w, Z(\mathbf{G_1}))\) is trivial for real places \(w\) in the specified types and hence \(S . \omega \subseteq S^{\omega} . \omega\).  As \(\mathbf{G}\) was required to not be weakly uniform, we obtain
  \[S . \omega . \Aut(k)_{\gamma} \subsetneq
    (S^{\omega} . \omega) . \Aut(\A_k^f)_{\gamma'} \subseteq
    \im d_1 \cap ((\prod_{v \nmid \infty} S_v . \omega) .
      \Aut(\A_k^f)_{\gamma'}),\]
  as the action of \(\Aut(\A_k^f)_{\gamma'}\) on the image of \(d_1\) is well-defined.  Thus
  \[d_1(S . \xi . \Aut(k)_{\gamma}) = S . \omega . \Aut(k)_{\gamma}
    \subsetneq \im d_1 \cap ((\prod_{v \nmid \infty} S_v . \omega) .
      \Aut(\A_k^f)_{\gamma'})\]
  holds and the fiber of \([\xi]\) along \(\overline{d_1}\) is nontrivial.
\end{proof}

The reason we have to exclude type \(D_{2n}\) here is that \(S_w\) acts nontrivially on \(H^2(k_w, Z(\mathbf{G_1}))\) for real places \(w\).  This leads to the fact that there in general \(S . \omega \not\subseteq S^{\omega} . \omega\), making the proof slightly more complicated there, as we will see in \Cref{thm:cong-rigid-D-n}.  In preparation for the case that \(k\) has real places we make precise statements on condition~\ref{item:real-places} from the introduction for the types \(A_n\) and \(D_n\) in the following two propositions.

\begin{proposition} \label{prop:type-2n+1-3-real-places}
  Let \(\mathbf{G}\) be a \(k\)-group of type \(A_{2n+1}\) or \(D_{2n+1}\).  Assume furthermore that \(k\) has at least three real places.  Then the fiber of \([\xi]\) along \(\overline{d_1}\) is never a singleton.
\end{proposition}

\begin{proof}
  In the listed types that \(\mathbf{G}\) can take on, \(H^2(k_w, Z(\mathbf{G_1})) \cong \Z/2\) holds for all real places \(w\).  As there have to be at least two places \(w_1\), \(w_2\) with \(\underline{\omega}_{w_1} = \underline{\omega}_{w_2}\), we can construct \(\theta \in \bigoplus_{v \in V(k)} H^2(k_v, Z(\mathbf{G_1}))\) by
  \begin{align*}
    \theta_v & = \underline{\omega}_v \text{ if } v \neq w_1, w_2 \\
    \theta_{w_i} & = \underline{\omega}_{w_i} + 1.
  \end{align*}
  The restriction of \(\theta\) to the finite places then is equal to \(\omega\) and \(\theta\) has a preimage under \(b\), as \(c(\theta) = c(\underline{\omega}) = 0\), which we denote by \(\eta\).  Hence \([\xi]\) and \([\eta]\) lie in the preimage of \([\omega]\) under \(\overline{d_1}\).  The action of \(\Aut(k)_{\gamma}\) on \(\xi\) can be traced by looking at the induced action on \(\underline{\omega}\) after embedding \(\Aut(k)_{\gamma}\) into \(\Aut(\A_k^f)_{\gamma'}\).  There we see that the action leaves the number of real coordinates \(w\) where \(\underline{\omega}_w = 0\) holds invariant, as it is merely a permutation.  This number is different from the one for \(\theta = b(\eta)\), hence \(\Aut(k)_{\gamma}\) cannot map \(\xi\) to \(\eta\).  As the same invariance holds for the action of \(S\), we obtain that \([\xi] \neq [\eta] \in S \setminus H^2(k, Z(\mathbf{G_1})) / \Aut(k)_{\gamma}\).  Hence the preimage of \([\omega] = \overline{d_1}([\xi])\) under \(\overline{d_1}\) contains at least two elements.
\end{proof}

\begin{proposition} \label{prop:2-real-places-D-n}
  Let \(\mathbf{G}\) be a \(k\)-group of type \(D_{2n}\) where \(n \geq 3\).  Assume furthermore that there are at least two real places at which \(\mathbf{G}\) is of inner type.  Then the fiber of \([\xi]\) along \(\overline{d_1}\) is never a singleton.
\end{proposition}

\begin{proof}
  We again construct an element \(\theta\) with a preimage which cannot be mapped to \(\xi\) by \(S\) or \(\Aut(k)_{\gamma}\).  For this purpose we pick two real places \(w_1\) and \(w_2\) where \(\mathbf{G}\) is of inner type and construct \(\theta\) by replacing the coordinates of \(\underline{\omega}\) in \(H^2(k_{w_i}, Z(\mathbf{G_1})) \cong \Z/2 \times \Z/2\) in the follwing manner:
  
  If neither \(\underline{\omega}_{w_1}\) nor \(\underline{\omega}_{w_1}\) are equal to \((0,0) \in H^2(k_{w_i}, Z(\mathbf{G_1}))\), we let
  \begin{align*}
    \theta_v & = \underline{\omega}_v \text{ for } v \neq w_1, w_2 \\
    \theta_{w_1} & = (0,0) \\
    \theta_{w_1} & = \underline{\omega}_{w_1} + \underline{\omega}_{w_2}.
  \end{align*}
  The number of real places at which \(\underline{\omega}\) has coordinate \((0,0)\) is invariant under the actions of \(\Aut(k)_{\gamma}\) and \(S\) on \(\underline{\omega}\).  Hence this \(\theta\) has a preimage under \(b\) that has a different class than \(\xi\) in \(S \setminus H^2(k, Z(\mathbf{G_1})) / \Aut(k)_{\gamma}\) but is mapped to \([\omega]\) as well under \(\overline{d_1}\).
  
  If neither \(\underline{\omega}_{w_1}\) nor \(\underline{\omega}_{w_1}\) are equal to \((1,1) \in H^2(k_{w_i}, Z(\mathbf{G_1}))\), we let similarly
  \begin{align*}
    \theta_v & = \underline{\omega}_v \text{ for } v \neq w_1, w_2 \\
    \theta_{w_1} & = (1,1) \\
    \theta_{w_1} & = \underline{\omega}_{w_1} + \underline{\omega}_{w_2} - (1,1)
  \end{align*}
  and remark that the number of \((1,1)\)-coordinates is also invariant under both \(\Aut(k)_{\gamma}\) and \(S\).
  
  If finally \(\underline{\omega}_{w_1}\) and \(\underline{\omega}_{w_1}\) are once each equal to \((0,0)\) and \((1,1)\), we let
  \begin{align*}
    \theta_v & = \underline{\omega}_v \text{ for } v \neq w_1, w_2 \\
    \theta_{w_1} & = (0,1) \\
    \theta_{w_1} & = (1,0). \qedhere
  \end{align*}
\end{proof}

We can finally give an exact characterization of congruence rigid groups for locally determined number fields with at least one real place in the types where \(\Sym \Delta \cong \Z/2\) holds.  We split those cases up into the three subsequent theorems.

\begin{theorem} \label{thm:cong-rigid-A-n}
  Let \(k\) be a locally determined number field and \(\mathbf{G}\) a \(k\)-group.  Assume furthermore that \(\mathbf{G}\) is of type \(A_{n \geq 2}\) and that \(k\) has at least one real place.  Then \(\mathbf{G}\) is congruence rigid if and only if it is weakly uniform, \([\gamma]\) has trivial fiber under \(\overline{h}\) and one of the following conditions is satisfied:
  \begin{enumerate}[label=(\roman*)]
    \item \(\mathbf{G}\) is of type \({}^1 A_{2n}\).
    \item \(\mathbf{G}\) is of type \({}^2 A_{2n}\) and splits at all real places.
    \item \label{item:1-A-2n+1-i} \(\mathbf{G}\) is of type \({}^1 A_{2n+1}\), \(k\) has one real place and furthermore there is no subset of coordinates of \(\omega\) adding to \(\pm \frac{n+1}{2}\) in \(\Z/(2n+2) \cong H^0(k,Z')^*\).
    \item \label{item:1-A-2n+1-ii} \(\mathbf{G}\) is of type \({}^1 A_{2n+1}\), \(k\) has two real places \(w_1\) and \(w_2\), \(\mathbf{G}\) is isomorphic once each to \(\SL(2n+2, \R)\) and to \(\SL(n+1, \mathbb{H})\) at those real places, \(\Aut(k)_{\gamma}\) permutes the two real places of \(k\), \(S . \omega . \Aut(k)_{\gamma, w_1} = (S^{\omega} . \omega) . \Aut(\A_k^f)_{\gamma'}\) and furthermore there is no subset of coordinates of \(\omega\) adding to \(\pm \frac{n+1}{2}\) in \(\Z/(2n+2) \cong H^0(k,Z')^*\).
    \item \(\mathbf{G}\) is of type \({}^2 A_{2n+1}\), \(k\) has one real place and there \(\mathbf{G}\) either becomes inner type or is (in the case of \(n=1\)) isomorphic to \(\SU(3,1)\).
    \item \(\mathbf{G}\) is of type \({}^2 A_{2n+1}\), \(k\) has two real places \(w_1\) and \(w_2\), \(\mathbf{G}\) is isomorphic once each to \(\SL(2n+2, \R)\) and to \(\SL(n+1, \mathbb{H})\) at those real places, \(\Aut(k)_{\gamma}\) permutes the two real places of \(k\) and \(S . \omega . \Aut(k)_{\gamma, w_1} = (S^{\omega} . \omega) . \Aut(\A_k^f)_{\gamma'}\).
  \end{enumerate}
\end{theorem}

Note that \([\gamma]\) always has trivial fiber under \(\overline{h}\) if \(\mathbf{G}\) is of type \({}^1 A_n\), as \([\gamma]\) then corresponds to the trivial element in \(H^1(k, \Sym \Delta)\).  In this case \(\Aut(k)_{\gamma} = \Aut(k)\) and \(\Aut(\A_k^f)_{\gamma'} = \Aut(\A_k^f)\) hold.

\begin{proof}
  By Theorems~\ref{thm:inner-twists-of-nontrivial-h-bar}, \ref{thm:f-g-equivalence} and \ref{thm:f_1-decomposition} the \(k\)-group \(\mathbf{G}\) is congruence rigid if and only if \([\gamma]\) has trivial fiber under \(\overline{h}\), the map \(\overline{d_1}\) has trivial fiber at \([\xi]\) and at all real places \(\mathbf{G}\) is isomorphic over \(\R\) to one of the groups listed in \Cref{thm:map-real-groups}.  Note first that if \(\mathbf{G}\) is not weakly uniform, \(\overline{d_1}\) has nontrivial fiber at \([\xi]\) regardless of any behaviour at real places by \Cref{prop:weakly-uniform-necessary}, so this constitutes a necessary condition.
  
  \medskip
  If \(\mathbf{G}\) is of type \(A_{2n}\), \(H^2(k_w, Z(\mathbf{G_1}))\) is trivial for all real places by Table~\ref{tate-table}.  Hence \(d_1\) is injective and \(\overline{d_1}\) has trivial fiber at \([\xi]\) if and only if \(\mathbf{G}\) is weakly uniform.  The only real form of type \(A_{2n}\) listed in \Cref{thm:map-real-groups} is \(\SL(2n+1, \R)\), which also is the only real form of type \({}^1 A_{2n}\) in general.  Thus the behavioural condition on \(\mathbf{G}\) at real places is empty for \({}^1 A_{2n}\) and for type \({}^2 A_{2n}\) it demands that \(\mathbf{G}\) splits.
  
  \medskip
  If \(\mathbf{G}\) is of type \(A_{2n+1}\), we can assume by \Cref{prop:type-2n+1-3-real-places} that \(k\) has at most two real places.  The condition of \Cref{thm:map-real-groups} tells us that in order for \(\mathbf{G}\) to be congruence rigid, it has to be isomorphic at all real places to either \(\SL(2n+2, \R)\) or \(\SL(n+1, \mathbb{H})\), which are the two real forms of type \({}^1 A_{2n+1}\) or isomorphic to the outer real form \(\SU(3,1)\) in the case \({}^2 A_3\).
  
  \medskip
  First, let \(\mathbf{G}\) be of type \({}^1 A_{2n+1}\) and assume that it satisfies neither of the conditions \ref{item:1-A-2n+1-i} or \ref{item:1-A-2n+1-ii}.  Assume first that \(n+1\) is even and there is a subset \(U \in V_f(k)\) with \(\sum_{v \in U} c_v(\omega_v)\) adding to
  \[\pm \frac{n+1}{2} \in \Z/(2n+2) \cong H^0(k,Z')^*.\]
  Then we can construct \(\theta \in \bigoplus_{v \in V(k)} H^2(k_v, Z(\mathbf{G_1}))\) with \(\theta|_{V_f(k)} \in \prod_{v \nmid \infty} S_v . \omega\) and \(\theta \in \ker(c) = \im(b)\) in the following way, fixing one real place \(w\):
  \begin{align*}
    \theta_v & = -\omega_v \text{ if } v \in U \\
    \theta_w & = \underline{\omega}_w + 1 \\
    \theta_v & = \underline{\omega}_v \text{ for all other places}
  \end{align*}
  But as \(\theta\) and \(\underline{\omega}\) have a different number of zeroes and ones at the real places, they cannot be mapped onto each other by \(S\) or \(\Aut(k)_{\gamma}\), so the fiber of \([\xi]\) along \(\overline{d_1}\) is nontrivial.
  
  If \(k\) has two real places and \(\mathbf{G}\) is isomorphic to either \(\SL(2n+2, \R)\) at both or to \(\SL(n+1, \mathbb{H})\), we can apply the argument from \Cref{prop:type-2n+1-3-real-places} and see that \(\overline{d_1}\) has nontrivial fiber at \([\xi]\).  If \(\mathbf{G}\) is once each isomorphic to \(\SL(2n+2, \R)\) and \(\SL(n+1, \mathbb{H})\) and \(\Aut(k)_{\gamma}\) does not permute the two real places, then
  \begin{align*}
    \theta_{w_i} & = \underline{\omega}_{w_i} + 1 \\
    \theta_v & = \omega_v \text{ for } v \nmid \infty.
  \end{align*}
  has a preimage that neither \(S\) nor \(\Aut(k)_{\gamma}\) can map to \(\xi\), so so the fiber of \([\xi]\) along \(\overline{d_1}\) is again nontrivial.
  
  If \(\mathbf{G}\) is once each isomorphic to \(\SL(2n+2, \R)\) and \(\SL(n+1, \mathbb{H})\) and \(\Aut(k)_{\gamma}\) permutes the two real places but there is an element
  \[\omega' \in (S^{\omega} . \omega) . \Aut(\A_k^f)_{\gamma'} \setminus S . \omega . \Aut(k)_{\gamma, w_1},\]
  we construct \(\theta\) by
  \begin{align*}
    \theta_{w_i} & = \underline{\omega}_{w_i} \\
    \theta_v & = \omega'_v \text{ for } v \nmid \infty.
  \end{align*}
  Let \(\eta\) be a preimage of \(\theta\) under \(b\).  Then \(\theta|_{V_f(k)} \in (S^{\omega} . \omega) . \Aut(\A_k^f)_{\gamma'}\) and therefore \([d_1(\eta)] = [\omega]\).  If \([\eta] = [\xi]\) were true, then there would be \(\sigma \in S\) and \(\phi \in \Aut(k)_{\gamma}\) such that \(\sigma . \xi . \phi = \eta\).  But then \(\sigma . \underline{\omega} . \phi = \theta\) holds and in particular \((\underline{\omega} . \phi)_{w_i} = \theta_{w_i} = \underline{\omega}_{w_i}\), so \(\phi \in \Aut(k)_{\gamma, w_1}\).  But then
  \[\omega' = \theta|_{V_f(k)}
    = (\sigma . \underline{\omega} . \phi)|_{V_f(k)}
    = \sigma . \omega . \phi \in S . \omega . \Aut(k)_{\gamma, w_1}\]
  in contradiction to the assumption.  So \([\eta] \neq [\xi]\) and the fiber of \([\xi]\) along \(\overline{d_1}\) is not trivial.
  
  \medskip
  Next, let \(\mathbf{G}\) be of type \({}^1 A_{2n+1}\) and assume that the fiber of \([\xi]\) along \(\overline{d_1}\) is nontrivial.  Let \([\eta] \neq [\xi]\) be another element in the fiber and let \(\theta \coloneqq b(\eta)\).
  
  If the number of zeroes at the real places of \(\underline{\omega}\) and \(\theta\) differs by two (and \(k\) thus has two real places), we can conclude that one has to be isomorphic to \(\SL(2n+2, \R)\) at both and the other is isomorphic to \(\SL(n+1, \mathbb{H})\) at both.  This is in opposition to the requirement that \(\mathbf{G}\) be isomorphic to \(\SL(2n+2, \R)\) and \(\SL(n+1, \mathbb{H})\) once each at the real places.
  
  If the number of zeroes at the real places of \(\underline{\omega}\) and \(\theta\) differs by one, then we let \(\Sigma \in \prod_{v \nmid \infty} S_v\) such that \(\theta|_{V_f(k)} = \Sigma . \omega\).  If we let \(U\) be the set of finite places where \(\Sigma_v\) is equal to the nontrivial element of \(S_v\), then
  \[n+1 = c(\omega) - c(\theta)
    = 2 \cdot \sum_{v \in U} \omega_v
    \in \Z/(2n+2) \cong H^0(k,Z')^*.\]
   This implies that \(n+1\) is even and that there is a subset \(U \subseteq V_f(k)\) with \(\sum_{v \in U} c_v(\omega_v)\) adding to \(\pm \frac{n+1}{2}\).
  
  Assume now that the number of zeroes at the real places of \(\underline{\omega}\) and \(\theta\) is the same.  In this case \([\eta] \neq [\xi]\) implies \(\theta \not\in S . \underline{\omega} . \Aut(k)_{\gamma}\).  If \(k\) has on real place, this implies \(\theta|_{V_f(k)} \not\in S . \omega . \Aut(k)_{\gamma}\), hence \(\mathbf{G}\) is not weakly uniform.  If \(k\) has two real places and \(\theta_{w_i} = \underline{\omega}_{w_i}\) holds, then \(\theta|_{V_f(k)} \not\in S . \omega . \Aut(k)_{\gamma, w_1}\) follows.  Assume finally that \(k\) has two real places and that \(\theta_{w_i} = \underline{\omega}_{w_i}+1\) holds.  If \(\Aut(k)_{\gamma}\) does not permute the real places of \(k\), then \(\theta|_{V_f(k)} \not\in S . \omega . \Aut(k)_{\gamma, w_1} = S . \omega . \Aut(k)_{\gamma}\) follows immediately.  So assume that \(\Aut(k)_{\gamma}\) does permute the real places and let \(\psi \in \Aut(k)_{\gamma}\) such that \(\psi(w_1) = w_2\).  Then
  \begin{align*}
    [\eta . \psi] = [\eta] & \neq [\xi], \\
    b(\eta . \psi) & = \theta . \psi, \\
    (\theta . \psi)|_{V_f(k)} & \in (S^{\omega} . \omega) . \Aut(\A_k^f)_{\gamma'} \text{ and} \\
    (\theta . \psi)_{w_i} & = \underline{\omega}_{w_i}
  \end{align*}
  all hold.  So \((\theta . \psi)|_{V_f(k)} \not\in S . \omega . \Aut(k)_{\gamma, w_1}\) follows and we again obtain \(S . \omega . \Aut(k)_{\gamma, w_1} \neq (S^{\omega} . \omega) . \Aut(\A_k^f)_{\gamma'}\).
  
  \medskip
  Finally, let \(\mathbf{G}\) be of type \({}^2 A_{2n+1}\).  At all real places \(\mathbf{G}\) needs to be either isomorphic to \(\SL(2n+2, \R)\) or \(\SL(n+1, \mathbb{H})\) and thus become inner type or be isomorphic to \(\SU(3,1)\) in the particular case of \(A_3\).  All further arguments go roughly in the same way as with \({}^1 A_{2n+1}\), with two key differences:
  
  First, the scenario that the number of zeroes at the real places of \(\underline{\omega}\) and \(\theta\) differing by one cannot occur in general:  This is because here real places contribute a summand of 0 or 1  in \(\Z/2 \cong H^0(k, Z(\mathbf{G_1})')\) for the Tate-condition.  As a difference between those two cannot be divided by two, we do not have to account for a \(\theta\) with a number of zeroes at the real places different from \(\underline{\omega}\) to appear in \(\im d_1 \cap (\prod S_v . \omega)\).  As opposed to type \({}^1 A_{2n+1}\) we hence receive no additional condition on subsets of coordinates of \(\omega\) adding to some specific value.
  
  Second, if \(k\) has two real places, \(\underline{\omega}\) will again have to have different coordinates \(\omega_{w_i} \in H^2(k_{w_i}, Z(\mathbf{G})) \cong \Z/2\) if we want \(\mathbf{G}\) to be congruence rigid.  But if \(\mathbf{G}\) is isomorphic to \(\SU(3,1)\) at one place and becomes an inner form at the other, then \(\Aut(k)_{\gamma}\) can never switch \(w_1\) and \(w_2\).  So the only possibility is again \(\mathbf{G}\) becoming isomorphic once each to \(\SL(2n+2, \R)\) and \(\SL(n+1, \mathbb{H})\) and the additional attached conditions.
\end{proof}

\begin{theorem} \label{thm:cong-rigid-D-n}
  Let \(k\) be a locally determined number field and \(\mathbf{G}\) a \(k\)-group.  Assume furthermore that \(\mathbf{G}\) is of type \(D_{n \geq 5}\) and that \(k\) has at least one real place.  Then \(\mathbf{G}\) is congruence rigid if \(\omega . \Aut(k)_{\gamma} = \omega . \Aut(\A_k^f)_{\gamma'}\) holds, \([\gamma]\) has trivial fiber under \(\overline{h}\) and one of the following conditions is satisfied:
  \begin{enumerate}[label=(\roman*)]
    \item \(\mathbf{G}\) is of type \({}^1 D_{2n}\), \(k\) has one real place and there \(\mathbf{G}\) is isomorphic to \(\Spin^*(4n)\) and \(\mathbf{G}\) has an inner twin at exactly one finite place.
    \item \(\mathbf{G}\) is of type \({}^2 D_{2n}\), \(k\) has one real place and there \(\mathbf{G}\) is isomorphic to \(\Spin^*(4n)\) and \(\mathbf{G}\) has no inner twins at finite places and.
    \item \(\mathbf{G}\) is of type \({}^1 D_5\), \(k\) has one real place and there \(\mathbf{G}\) is isomorphic to \(\Spin(7,3)\) and \(\mathbf{G}\) has no inner twins at finite places.
    \item \(\mathbf{G}\) is of type \({}^2 D_{2n+1}\), \(k\) has one real place and there \(\mathbf{G}\) is isomorphic to \(\Spin^*(4n+2)\) or (in the case of \(n=2\)) to \(\Spin(7,3)\) and \(\mathbf{G}\) has an inner twin at at most one finite place.
  \end{enumerate}
\end{theorem}

Note again that \([\gamma]\) always has trivial fiber under \(\overline{h}\) if \(\mathbf{G}\) is of type \({}^1 D_n\), as \([\gamma]\) then corresponds to the trivial element in \(H^1(k, \Sym \Delta)\).  The condition \(\omega . \Aut(k)_{\gamma} = \omega . \Aut(\A_k^f)_{\gamma'}\) turns out to be almost the same as the weak uniformity imposed in the types without Dynkin diagram symmetries. The only difference is that \(\Aut(k)\) and \(\Aut(\A_k^f)\) have to be restricted to stabilizers of \(\gamma\) and \(\gamma'\), as they otherwise would not act in a well-defined manner in the outer types.

\begin{proof}
  As in the proof of \Cref{thm:cong-rigid-A-n}, we have to consider when exactly \(\overline{d_1}\) has trivial fiber at \([\xi]\) under the assumption that \(\mathbf{G}\) is isomorphic to one of the groups listed in \Cref{thm:map-real-groups} at all real places.
  
  \medskip
  If \(\mathbf{G}\) is of type \(D_{2n}\), \Cref{thm:map-real-groups} then implies that \(\mathbf{G}\) has to be isomorphic at all real places to \(\Spin^*(4n)\), which is of inner type. Then \Cref{prop:2-real-places-D-n} tells us that \(\mathbf{G}\) can only be congruence rigid if it has at most one real place.
  
  Assume first that \(\mathbf{G}\) is of type \({}^1 D_{2n}\).  Then
  \[H^0(k, Z(\mathbf{G})')^* \cong \Z/2 \times \Z/2
    \cong H^2(k_v, Z(\mathbf{G})\]
  for both finite and real \(v\) by \Cref{thm:tate-map}.  Because \(\mathbf{G}\) is isomorphic to \(\Spin^*(4n)\) at the real place \(w\), we see as in \Cref{prop:map-real-groups} that \(\underline{\omega}_w\) corresponds to one of the two elements of \(\Z/2 \times \Z/2\) that is not fixed under the action of \(S_w\) as described in \Cref{rmk:dynkin-actions}.  Without loss of generality, we assume \(\underline{\omega}_w = (1,0)\) and observe that
  \[\sum_{v \nmid \infty} c_v(\omega_v) = (1,0) \in H^0(k, Z')^*.\]
  Thus for any \(\Sigma \in \prod_{v \nmid \infty} S_v\)
  \[\sum_{v \nmid \infty} c_v(\Sigma_v . \omega_v) \in \{(1,0), (0,1)\},\]
  as any application of \(\Sigma_v\) will change the sum  by \((1,-1)\) or \((-1,1)\) in total, so by \((1,1) \in \Z/2 \times \Z/2\).  Hence
  \[(\prod_{v \nmid \infty} S_v . \omega) \cap \im d_1 = S^{\omega} . \omega \sqcup T^{\omega} . \omega\]
  where
  \[T^{\omega} \coloneqq \{\Sigma \in \prod_{v \nmid \infty} S_v
    | \sum_{v \nmid \infty} c_v(\Sigma_v . \omega_v) = (0,1)\}\]
  under our previously fixed identification \(H^0(k, Z')^* \cong \Z/2 \times \Z/2\).  As \(d_1\) is injective according to \Cref{thm:tate-map}, we get the statement
  \[S . \omega . \Aut(k)_{\gamma} = (S^{\omega} . \omega \sqcup T^{\omega} . \omega) . \Aut(\A_k^f)_{\gamma'}\]
  to be equivalent to the triviality of the fiber of \([\xi]\) along \(\overline{d_1}\), similar to the proof of \Cref{thm:sym-rigidity-no-reals}.  So
  \begin{align*}
    \omega . \Aut(k)_{\gamma} & \subseteq (S^{\omega} . \omega)
      . \Aut(\A_k^f)_{\gamma'} \text{ and} \\
    \sigma . \omega \Aut(k)_{\gamma} & \subseteq (T^{\omega} . \omega)
      . \Aut(\A_k^f)_{\gamma'}
  \end{align*}
  hold, where \(\sigma\) is the nontrivial element of \(S\).  Furthermore, applying \(\sigma\) provides a bijection between both inclusions, so we get that the fiber of \([\xi]\) along \(\overline{d_1}\) is trivial if and only if \(\omega . \Aut(k)_{\gamma} = S^{\omega} . \omega . \Aut(\A_k^f)_{\gamma'}\) holds.
  
  As the number of places \(v\) where \(\omega_v\) is equal to \((1,0)\) is fixed under \(\Aut(k)_{\gamma}\), we then conclude that there can be at most one such place in a similar way as in the proof of \Cref{thm:outer-types-weakly-uniform}.  From
  \[\sum_{v \nmid \infty} c_v(\omega_v) = (1,0)\]
  we then deduce that there has to be exactly one such finite place, i.e. a place where \(\mathbf{G}\) has an inner twin.  But then \(S^{\omega} . \omega = \omega\) and the condition reduces to \(\omega . \Aut(k)_{\gamma} = \omega . \Aut(\A_k^f)_{\gamma'}\).
  
  \medskip
  Assume now that \(\mathbf{G}\) is of type \({}^2 D_{2n}\).  Then
  \[H^0(k, Z(\mathbf{G})')^* \cong \Z/2\]
  and for the real place \(w\) we still have
  \[\underline{\omega}_w = (1,0) \in \Z/2 \times \Z/2
    \cong H^2(k_w, Z')^*\]
  without loss of generality.  From \Cref{thm:tate-map} we deduce
  \[\sum_{v \nmid \infty} c_v(\omega_v) = 1
    \in H^0(k, Z(\mathbf{G})')^*\]
  and the same holds for any \(\Sigma . \omega\) with \(\Sigma \in \prod_{v \nmid \infty}\).  We can now use the injective map \(b\) to describe the preimage of \(\prod_{v \nmid \infty} S_v . \omega\) under \(d_1\) by
  \[b({d_1}^{-1}(\prod_{v \nmid \infty} S_v . \omega))
    = S^{\omega} . \omega \times S_w . \underline{\omega}_w.\]
  Thus we see that the fiber of \([\xi]\) along \(\overline{d_1}\) is trivial if and only if
  \[b(d^{-1}(\prod_{v \nmid \infty} S_v . \omega .
    \Aut(\A_k^f)_{\gamma'}))
    = b(S . \xi . \Aut(k)_{\gamma})\]
  or equivalently
  \[S^{\omega} . \omega \times \{(1,0), (0,1)\}_w
    = S . \underline{\omega} . \Aut(k)_{\gamma},\]
  where the second factor on the left side denotes the coordinate at the real place \(w\).  As in the previous case, \(\underline{\omega} . \Aut(k)_{\gamma} \subseteq S^{\omega} . \omega \times \{(1,0)\}_w\) and by the same argument we conclude that the fiber of \([\xi]\) along \(\overline{d_1}\) is trivial if and only if \(\omega . \Aut(k)_{\gamma} = S^{\omega} . \omega . \Aut(\A_k^f)_{\gamma'}\) holds.
  
  We can now once more employ the argumentation used in the proof of \Cref{thm:outer-types-weakly-uniform} to deduce that this time \(\omega\) can not have any coordinate \(\omega_v\) that is modified by the action of \(S_v\).  Hence \(\mathbf{G}\) can not have any inner twins at finite places and the condition again reduces to \(\omega . \Aut(k)_{\gamma} = \omega . \Aut(\A_k^f)_{\gamma'}\).
  
  \medskip
  If \(\mathbf{G}\) is of type \(D_{2n+1}\), then \Cref{thm:map-real-groups} implies that \(\mathbf{G}\) has to be isomorphic to either \(\Spin^*(4n+2)\) or (in the case of \(D_5\)) to \(\Spin(7,3)\) at all real places.  But only \(\Spin(7,3)\) is an inner form, so this behaviour is impossible for groups of type \({}^1 D_{2n+1}\) where \(n \geq 3\).  If \(k\) has more than two real places, \Cref{prop:type-2n+1-3-real-places} states that \(\mathbf{G}\) cannot be congruence rigid.
  
  Assume first that \(\mathbf{G}\) is of type \({}^1 D_5\).  If \(k\) has exactly two real places, it would have to be isomorphic at both to \(\Spin(7,3)\).  But then \(\underline{\omega}_{w_1} = \underline{\omega}_{w_2}\) for those real places and we can employ the same argument as in the proof of \Cref{prop:type-2n+1-3-real-places} to see that \(\mathbf{G}\) is not congruence rigid.
  
  If \(k\) has exactly one real place and we want \(\overline{d_1}\) to have trivial fiber at \([\xi]\), we see that as in type \({}^1 A_{2n+1}\) we have to avoid any subset of \(\omega_v\) adding to
  \[\pm 1 = \pm \frac{1}{2} 2 \in \Z/4 \cong H^0(k, Z')^*.\]
  Here 2 is the difference between the two possible images of the real places under \(c_v\).  This value of \(c_v(\omega_v)\) occurs exactly whenver \(\omega_v\) is 1 or 3 in \(H^2(k_v, Z) \cong \Z/4\), so \(\omega\) cannot take those values at any finite place.  But if this is true, then \(\mathbf{G}\) has no inner twins at finite places and \(\omega = (\prod_{v \nmid \infty} S_v) . \omega\) follows.  Because \(d_1\) is injective, we conclude that the first condition of \Cref{thm:f_1-decomposition}, namely
  \[S . \xi . \Aut(k)_{\gamma}
    = {d_1}^{-1}((\prod_{v \nmid \infty} S_v . \omega) . \Aut(\A_k^f)_{\gamma'} \cap \im(d)),\]
  is equivalent to
  \[\omega . \Aut(k)_{\gamma} = \omega . \Aut(\A_k^f)_{\gamma'}\]
  by \Cref{cor:poitou-compatability}.
  
  Assume now that \(\mathbf{G}\) is of type \({}^2 D_{2n+1}\).  If \(k\) has more than two real places, \Cref{prop:type-2n+1-3-real-places} again states that \(\mathbf{G}\) cannot be congruence rigid.  If \(k\) has exactly two real places and is isomorphic at both to either \(\Spin^*(4n+2)\) or \(\Spin(7,3)\), we again argue as in the proof of \Cref{prop:type-2n+1-3-real-places} and conclude that \(\mathbf{G}\) is not congruence rigid.  If it is isomorphic to \(\Spin^*(4n+2)\) and \(\Spin(7,3)\) once each, we can once more construct \(\theta \in \bigoplus_{v \in V(k)} H^2(k_v, Z(\mathbf{G_1}))\) by
  \begin{align*}
    \theta_v & = \omega_v \text{ if } v \neq w_1, w_2 \\
    \theta_{w_i} & = \underline{\omega}_{w_i} + 1.
  \end{align*}
  In this case however there is no danger of \(\Aut(k)_{\gamma}\) mapping \(\xi\) to a preimage \(\eta\) of \(\theta\) under \(b\), even if the number of zeroes in the \(H^2(k_{w_i}, Z(\mathbf{G_1})) \cong \Z/2\) is the same for \(\underline{\omega}\) and \(\theta\).  This is because \(\mathbf{G}\) is of inner type at one real place and of outer type at the other, so \(\Aut(k)_{\gamma}\) cannot permute them.
  
  If \(k\) has exactly one real place, \(d_1\) is once more injective and \(S_w\) acts trivially on \(H^2(k_w, Z(\mathbf{G_1}))\), so \(S . \omega \subseteq S^{\omega} . \omega\).  Furthermore, there is no \(\Sigma \in \prod_{v \nmid \infty} S_v\) for which
  \[c(\Sigma . \omega) - c(\omega)\]
  is equal to the differnce between two values that \(c_w\) could take at the real places, as opposed to the case \({}^1 D_{2n+1}\).  So we can argue as in the case of \({}^2 A_{2n+1}\) and obtain that \(\mathbf{G}\) is congruence rigid if and only if it is weakly uniform and isomorphic to either \(\Spin^*(4n+2)\) or \(\Spin(7,3)\) at the real place.  We then use \Cref{thm:outer-types-weakly-uniform} to translate weak uniformity, as this allows us to impose the condition \(\omega . \Aut(k)_{\gamma} = \omega . \Aut(\A_k^f)_{\gamma'}\) in all cases \(D_n\).
\end{proof}

\begin{theorem} \label{thm:cong-rigid-E-6}
  Let \(k\) be a locally determined number field and \(\mathbf{G}\) a \(k\)-group.  Assume furthermore that \(\mathbf{G}\) is of type \(E_6\) and that \(k\) has at least one real place.  Then \(\mathbf{G}\) is never congruence rigid.
\end{theorem}

Note again that \([\gamma]\) always has trivial fiber under \(\overline{h}\) if \(\mathbf{G}\) is of type \({}^1 E_6\), as \([\gamma]\) then corresponds to the trivial element in \(H^1(k, \Sym \Delta)\).

\begin{proof}
  By Theorems~\ref{thm:inner-twists-of-nontrivial-h-bar}, \ref{thm:f-g-equivalence} and \ref{thm:f_1-decomposition} the \(k\)-group \(\mathbf{G}\) can only be congruence rigid if \(\mathbf{G}\) is isomorphic at all real places to one of the groups listed in \Cref{thm:map-real-groups}.  But this list does not contain any groups of type \(E_6\), so the condition can never be satisfied.
\end{proof}

Note that the three previous theorems together with \Cref{thm:cong-rigid-no-symmetries} prove \Cref{thm:3-real-places-not-rigid}.

\section{Special cases}
\label{section:special-cases}

In this final section, we deal with noteworthy special cases of our main results.  Most of them were already highlighted in the introduction but we start with a more general statement on quasi-split groups.

\begin{corollary} \label{cor:cong-rigid-split-sym-2}
  Let \(k\) be a locally determined number field and \(\mathbf{G}\) a \(k\)-group.  Assume furthermore that \(\mathbf{G}\) is of type \(A_{n \geq 2}\), \(D_{n \geq 5}\) or \(E_6\) and that \(\mathbf{G}\) is quasi-split over \(k\).  Then \(\mathbf{G}\) is congruence rigid if and only if  one of the following conditions is satisfied:
  \begin{enumerate}[label=(\roman*)]
    \item \(\mathbf{G}\) is of type \({}^1 A_{2n}\).
    \item \(\mathbf{G}\) is of type \({}^2 A_{2n}\), \(\mathbf{G}\) splits at all real places and \([\gamma]\) has trivial fiber along \(\overline{h}\).
    \item \(\mathbf{G}\) is of type \({}^1 A_{2n+1}\) and \(k\) has at most one real place.
    \item \(\mathbf{G}\) is of type \({}^2 A_{2n+1}\), \(k\) has at most one real place, there \(\mathbf{G}\) splits and \([\gamma]\) has trivial fiber along \(\overline{h}\).
    \item \(\mathbf{G}\) is of type \({}^1 D_{2n}\) and \(k\) has no real places.
    \item \(\mathbf{G}\) is of type \({}^2 D_{2n}\), \(k\) has no real places and \([\gamma]\) has trivial fiber along \(\overline{h}\).
    \item \(\mathbf{G}\) is of type \({}^1 D_{2n+1}\) and \(k\) has no real places.
    \item \(\mathbf{G}\) is of type \({}^2 D_{2n+1}\), \(k\) has no real places and \([\gamma]\) has trivial fiber along \(\overline{h}\).
    \item \(\mathbf{G}\) is of type \({}^1 E_6\) and \(k\) has no real places.
    \item \(\mathbf{G}\) is of type \({}^2 E_6\), \(k\) has no real places and \([\gamma]\) has trivial fiber along \(\overline{h}\).
  \end{enumerate}
\end{corollary}

\begin{proof}
  Quasi-split groups are always weakly uniform and have no inner twins, as \(\omega_v = 0\) for all (finite and real) places holds.  Furthermore the only quasi-split groups of the considered types listed in \Cref{thm:map-real-groups} are the groups \(\SL(n+1, \R)\).  Then the statements follows immediately from applying Theorems~\ref{thm:cong-rigid-A-n}, \ref{thm:cong-rigid-D-n} and \ref{thm:cong-rigid-E-6}.
\end{proof}

One also retrieves the remaining results as given in \cite{Kammeyer-Spitler:chevalley}*{Theorem 2} from this with the exception of the unwieldy case \(D_4\).  The additional assumption that \(k/\Q\) is Galois allows us to simplify \Cref{thm:cong-rigid-no-symmetries} and \Cref{cor:cong-rigid-split-sym-2} with the help of \Cref{thm:h-bar-injective} in the following way:

\begin{corollary} \label{cor:cong-rigid-split}
  Let \(k\) be Galois over \(\Q\) and let \(\mathbf{G}\) be a \(k\)-group.  Assume that \(\mathbf{G}\) is quasi-split over \(k\) and not of type \(D_4\).  Then \(\mathbf{G}\) is congruence rigid if and only if  one of the following conditions is satisfied:
  \begin{enumerate}[label=(\roman*)]
  \item \(\mathbf{G}\) is of type \(A_{2n}\) for \(n \ge 1\) and splits at all infinite places of \(k\).
  \item \(\mathbf{G}\) is of type \(A_{2n+1}\) for \(n \ge 0\) or \(C_{n \ge 2}\), and either \(k = \Q\) and \(\mathbf{G}\) splits at the real place or \(k\) is totally imaginary.
    \item \(\mathbf{G}\) is of type \(B_{n \geq 3}\), \(D_{n \geq 5}\), \(E_6\), \(E_7\), \(E_8\), \(F_4\), or \(G_2\) and \(k\) is totally imaginary.
  \end{enumerate}
\end{corollary}

Similarly, we can specialize the previous results to make clearer statements on the case \(k= \Q\).  Keep in mind that for \(\Q\) the notion of weak uniformity reduces to the statement that \(|S^{\omega} . \omega| \leq 2\), which in turn can be computed by \Cref{cor:S-omega-characterization}.

\begin{corollary} \label{cor:cong-rigid-Q}
  Let \(\mathbf{G}\) be a \(\Q\)-group not of type \(D_4\).  Then \(\mathbf{G}\) is congruence rigid if and only if  one of the following conditions is satisfied:
  \begin{enumerate}[label=(\roman*)]
  \item \(\mathbf{G}\) is of type \(A_1\).
  \item \(\mathbf{G}\) is of type \({}^1 A_{2n}\) or \({}^1 A_{4n+1}\) for \(n \ge 1\) and weakly uniform.
    \item \(\mathbf{G}\) is of type \({}^2 A_{2n}\) for \(n \ge 1\), is not quasi-split at at most one finite place and splits over \(\R\).
   \item \(\mathbf{G}\) is of type \({}^1 A_{4n+3}\) for \(n \ge 0\), weakly uniform and no subset of coordinates of \(\omega\) adds up to \(\pm (n+1)\) in \(\Z/(4n+4) \cong H^0(k,Z')^*\).
    \item \(\mathbf{G}\) is of type \({}^2 A_{2n+1}\) for \(n \ge 1\), has an inner twin at at most one finite place, and over \(\R\) the group \(\mathbf{G}\) either becomes inner type or is isomorphic to \(\SU(3,1)\) if \(n=1\).
    \item \(\mathbf{G}\) is of type \(C_n\) for \(n \geq 2\) and over \(\R\) the group \(\mathbf{G}\) becomes isomorphic to \(\Sp(2n, \R)\).
    \item \(\mathbf{G}\) is of type \({}^1 D_{2n}\) for \(n \geq 3\), has an inner twin at exactly one finite place, and over \(\R\) the group \(\mathbf{G}\) becomes isomorphic to \(\Spin^*(4n)\).
    \item \(\mathbf{G}\) is of type \({}^2 D_{2n}\) for \(n \geq 3\), has no inner twins at finite places, and over \(\R\) the group \(\mathbf{G}\) becomes isomorphic to \(\Spin^*(4n)\).
    \item \(\mathbf{G}\) is of type \({}^1 D_5\), has no inner twins at finite places, and over \(\R\) the group \(\mathbf{G}\) becomes isomorphic to \(\Spin(7,3)\).
    \item \(\mathbf{G}\) is of type \({}^2 D_{2n+1}\) for \(n \geq 2\), has an inner twin at at most one finite place, and over \(\R\) the group \(\mathbf{G}\) becomes isomorphic to \(\Spin^*(4n+2)\) or to \(\Spin(7,3)\) if \(n=2\).
  \end{enumerate}
\end{corollary}

In \Cref{thm:outer-types-weakly-uniform} we gave a very easily checked criterion on when a group of outer type is not weakly uniform.  This implies \Cref{thm:outer-types-cong-rigid}.

\begin{proof}[Proof of \Cref{thm:outer-types-cong-rigid}.]
  \Cref{thm:outer-types-weakly-uniform} implies the claim by virtue of \Cref{prop:weakly-uniform-necessary} except if \(\mathbf{G}\) is of type \({}^2 D_{2n}\) and \(k\) has real places.  But in that case we can apply \Cref{thm:cong-rigid-D-n} and observe that the statement holds true as well.
\end{proof}

We now complement this observation with some combinatorial considerations on groups of inner type.

\begin{proposition}
  Let \(\mathbf{G}\) be a \(k\)-group such that \(\Sym \Delta \cong \Z/2\).  If \(|S^{\omega} . \omega| > 2\) and \(\Aut (k)\) is trivial, then \(\mathbf{G}\) is not congruence rigid.
\end{proposition}

\begin{proof}
  The first condition implies that there is an \(\omega' \in (\prod_{v \nmid \infty} S_v . \omega) \cap \im(d)\) not contained in \(S . \omega\).  As \(\Aut(k)_{\gamma} \subseteq \Aut (k)\) is trivial, \(\omega'\) is not contained in \(S . \omega . \Aut(k)_{\gamma}\) either and thus the second condition of \Cref{thm:sym-rigidity-no-reals} is not met.
\end{proof}

\begin{proposition} \label{prop:partial-sums}
  Let \(r\) be the number of finite places at which \(\mathbf{G}\) has an inner twin.  Then \(|S^{\omega} . \omega| > 2\) follows if one of the following holds:
  \begin{itemize}
    \item \(\mathbf{G}\) is of type \({}^1 A_{2n}\) and \(r > 2n+1\).
    \item \(\mathbf{G}\) is of type \({}^1 A_{2n+1}\) and \(r > n+1\).
    \item \(\mathbf{G}\) is of type \({}^1 D_n\) and \(r > 2\).
    \item \(\mathbf{G}\) is of type \({}^1 E_6\) and \(r > 3\).
  \end{itemize}
\end{proposition}

\begin{proof}
  Let \(\mathbf{G}\) first be of type \({}^1 A_{2n}\) and \(v_1, \dots v_{2n+2}\) be pairwise distinct finite places at which \(\omega\) does not vanish.  Let
  \[z_r := \sum_{i=1}^r c_{v_i}(\omega_{v_i})
    \in H^ 0(k, Z')^* \cong \Z/(2n+1). \]
  As we have \(2n+2\) elements in \(\Z/(2n+1)\Z\), two of the \(z_i\), say \(z_l = z_m\) with \(l < m\) have to conicide.  So
  \[ 0 = z_m - z_l = \sum_{i=l+1}^m c_{v_i}(\omega_{v_i}) \]
  provides a partial sum adding to 0 for the \(c\)-images of \(\omega\).  The only element of \(H^2(k_v, Z(\mathbf{G_1})) \cong \Z/(2n+1)\Z\) that is fixed under \(S_v\) is 0.  Hence subset corresponds to an element of \(S^{\omega} . \omega\) by \Cref{cor:S-omega-characterization} and \(\mathbf{G}\) has a local symmetry.
  
  Let \(\mathbf{G}\) next be of type \({}^1 A_{2n+1}\).  The condition of \(\omega\) having a partial sum adding to 0 or \(n+1\) modulo \(2n+2\) where all \(\omega_v\) are not fixed under the respective \(S_v\) is equivalent to having a partial sum adding to 0 modulo \(n+1\) where all summands are nontrivial.  With the same argument as before we see that having more than \(n+1\) places \(v\) where \(\omega_v\) differs from 0 and \(n+1\) guarantees us such a partial sum corresponding to \(U\) as defined in \Cref{cor:S-omega-characterization}.
  
  For the cases \({}^1 D_n\) and \({}^1 E_6\) we can directly employ \Cref{cor:S-omega-characterization} to establish the claimed lower bounds.
\end{proof}

We can generalize this approximation to fields with arbitrary automorphism group in the following way.

\begin{theorem} \label{thm:many-inner-twins}
  Let \(\mathbf{G}\) be a \(k\)-group such that \(\Sym \Delta \cong \Z/2\).  Let \(r\) be the number of finite places \(v\) where \(\mathbf{G}\) has an inner twin.  Let furthermore
  \begin{align*}m =
    \begin{cases}
      2n+1 & \text{if } \mathbf{G} \text{ has type } {}^1 A_{2n} \\
      n+1 & \text{if } \mathbf{G} \text{ has type } {}^1 A_{2n+1} \\
      2 & \text{if } \mathbf{G} \text{ has type } {}^1 D_{n \geq 5} \\
      3 & \text{if } \mathbf{G} \text{ has type } {}^1 E_6
    \end{cases}
  \end{align*}
  If
  \[|\Aut(k)_{\gamma}| < 2^{\left\lfloor\frac{r-1}{m}\right\rfloor}\]
  then \(\mathbf{G}\) is not congruence rigid.
\end{theorem}

As \(|\Aut(k)_{\gamma}|\) is in turn bounded from above by \([k:\Q]\), this Theorem provides an easily checked criterion for construction of non-rigid algebraic groups of type \({}^1 A_{n \geq 2}\), \({}^1 D_{n \geq 5}\) or \({}^1 E_6\) over an arbitrary field.

\begin{proof}
  If \(\mathbf{G}\) has an inner twin at \(v\), then \(\omega_v\) is not fixed under the action of \(S_v\).  As in the proof of \Cref{prop:partial-sums}, having \(m+1\) finite places with a coordinate not fixed under \(S_v\) results in a partial sum adding to a value \(x \in H^0(k, Z')^*\) that is fixed under \(S\).  By induction we thus obtain
  \[2^{\left\lfloor\frac{r-1}{m}\right\rfloor+1}\]
  elements in \((\prod_{v \nmid \infty} S_v . \omega) \cap \im(d)\), which bounds \(S^{\omega} . \omega . \Aut(\A_k^f)_{\gamma'}\) from below.  On the other hand, \(|S . \omega . \Aut(k)_{\gamma}|\) is bounded from above by \(2 \cdot |\Aut(k)_{\gamma}|\).  If the inequality given in the theorem holds, we therefore see that \(\overline{d_1}\) has nontrivial fiber at \([\xi]\) and the statement follows from \Cref{thm:sym-rigidity-no-reals}.
\end{proof}


\begin{bibdiv}[References]

  \begin{biblist}

    \bib{Adams-Taibi:galois-cohomology}{article}{
   author={Adams, Jeffrey},
   author={Ta\"{\i}bi, Olivier},
   title={Galois and Cartan cohomology of real groups},
   journal={Duke Math. J.},
   volume={167},
   date={2018},
   number={6},
   pages={1057--1097},
   issn={0012-7094},
   review={\MR{3786301}},
 } 

 \bib{Bridson-et-al:absolute}{article}{
   author={Bridson, M. R.},
   author={McReynolds, D. B.},
   author={Reid, A. W.},
   author={Spitler, R.},
   title={Absolute profinite rigidity and hyperbolic geometry},
   journal={Ann. of Math. (2)},
   volume={192},
   date={2020},
   number={3},
   pages={679--719},
   issn={0003-486X},
   review={\MR{4172619}},
}

 \bib{Harari:galois}{book}{
   author={Harari, David},
   title={Galois cohomology and class field theory},
   series={Universitext},
   note={Translated from the 2017 French original by Andrei Yafaev},
   publisher={Springer, Cham},
   date={[2020] \copyright2020},
   pages={xiv+338},
   isbn={978-3-030-43901-9},
   isbn={978-3-030-43900-2},
   review={\MR{4174395}},
 }

\bib{Harder:Bericht}{article}{
   author={Harder, G\"unter},
   title={Bericht \"uber neuere Resultate der Galoiskohomologie
   halbeinfacher Gruppen},
   language={German},
   journal={Jber. Deutsch. Math.-Verein.},
   volume={70},
   date={1967/68},
   pages={182--216},
   issn={0012-0456},
   review={\MR{0242838}},
}

\bib{Jaikin-Lubotzky:remarks}{article}{
   author={Jaikin-Zapirain, Andrei},
   author={Lubotzky, Alexander},
   title={Some remarks on Grothendieck pairs},
   journal={Groups Geom. Dyn.},
   volume={19},
   date={2025},
   number={2},
   pages={597--616},
   issn={1661-7207},
   review={\MR{4940654}},
}
 
  \bib{Janusz:automorphisms}{article}{
   author={Janusz, Gerald J.},
   title={Automorphism groups of simple algebras and group algebras},
   conference={
      title={Representation theory of algebras},
      address={Proc. Conf., Temple Univ., Philadelphia, Pa.},
      date={1976},
   },
   book={
      series={Lect. Notes Pure Appl. Math.},
      volume={Vol. 37},
      publisher={Dekker, New York-Basel},
   },
   isbn={0-8247-6714-4},
   date={1978},
   pages={381--388},
   review={\MR{0472928}},
 }

 \bib{Kammeyer:absolute}{article}{
   author={Kammeyer, Holger},
   title={On absolutely profinitely solitary lattices in higher rank Lie
   groups},
   journal={Proc. Amer. Math. Soc.},
   volume={151},
   date={2023},
   number={4},
   pages={1801--1809},
   issn={0002-9939},
   review={\MR{4550371}},
}

 \bib{Kammeyer:profinite-commensurability}{article}{
   author={Kammeyer, Holger},
   title={Profinite commensurability of $S$-arithmetic groups},
   journal={Acta Arith.},
   volume={197},
   date={2021},
   number={3},
   pages={311--330},
   issn={0065-1036},
   review={\MR{4194949}},
}

\bib{Kammeyer-Kionke:adelic-superrigidity}{article}{
   author={Kammeyer, Holger},
   author={Kionke, Steffen},
   title={Adelic superrigidity and profinitely solitary lattices},
   journal={Pacific J. Math.},
   volume={313},
   date={2021},
   number={1},
   pages={137--158},
   issn={0030-8730},
   review={\MR{4313430}},
 }
 
 \bib{Kammeyer-Kionke:rigidity}{article}{
   author={Kammeyer, Holger},
   author={Kionke, Steffen},
   title={On the profinite rigidity of lattices in higher rank Lie groups},
   journal={Math. Proc. Cambridge Philos. Soc.},
   volume={174},
   date={2023},
   number={2},
   pages={369--384},
   issn={0305-0041},
   review={\MR{4545210}},
 }
 
 \bib{Kammeyer-Serafini:sign}{article}{
   author={Kammeyer, Holger},
   author={Serafini, Giada},
   title={On the Euler characteristic of S-arithmetic groups},
   journal={J. Lond. Math. Soc. (2)},
   volume={113},
   date={2026},
   number={3},
   pages={Paper No. e70505},
   issn={0024-6107},
   review={\MR{5042537}},
}

\bib{Kammeyer-Spitler:chevalley}{article}{
   author={Kammeyer, Holger},
   author={Spitler, Ryan},
   title={Galois cohomology and profinitely solitary Chevalley groups},
   journal={Math. Ann.},
   volume={390},
   date={2024},
   number={2},
   pages={2497--2511},
   issn={0025-5831},
   review={\MR{4801834}},
 }
 
 \bib{Klingen:similarities}{book}{
   author={Klingen, Norbert},
   title={Arithmetical similarities},
   series={Oxford Mathematical Monographs},
   note={Prime decomposition and finite group theory;
   Oxford Science Publications},
   publisher={The Clarendon Press, Oxford University Press, New York},
   date={1998},
   pages={x+275},
   isbn={0-19-853598-8},
   review={\MR{1638821}},
}

 \bib{Kneser:galois}{article}{
   author={Kneser, Martin},
   title={Galois-Kohomologie halbeinfacher algebraischer Gruppen \"uber
   ${\germ p}$-adischen K\"orpern. II},
   language={German},
   journal={Math. Z.},
   volume={89},
   date={1965},
   pages={250--272},
   issn={0025-5874},
   review={\MR{0188219}},
}

\bib{Komatsu:adele-rings}{article}{
   author={Komatsu, Keiichi},
   title={On the adele rings and zeta-functions of algebraic number fields},
   journal={Kodai Math. J.},
   volume={1},
   date={1978},
   number={3},
   pages={394--400},
   issn={0386-5991},
   review={\MR{0517831}},
 }
  
\bib{Linowitz-et-al:locally-equivalent}{article}{
   author={Linowitz, Benjamin},
   author={McReynolds, D. B.},
   author={Miller, Nicholas},
   title={Locally equivalent correspondences},
   language={English, with English and French summaries},
   journal={Ann. Inst. Fourier (Grenoble)},
   volume={67},
   date={2017},
   number={2},
   pages={451--482},
   issn={0373-0956},
   review={\MR{3669503}},
 }

 \bib{LMFDB}{misc}{
  label={LMFDB E234446.a},
  author={The {LMFDB Collaboration}},
  title={The {L}-functions and modular forms database, Home page of the number field \texttt{6.0.115063617043.2}},
  date={2026},
  note={[Online; accessed 8 April 2026]},
  eprint={https://www.lmfdb.org/NumberField/6.0.115063617043.2}
}

\bib{Perlis:equation}{article}{
   author={Perlis, Robert},
   title={On the equation $\zeta \sb{K}(s)=\zeta \sb{K'}(s)$},
   journal={J. Number Theory},
   volume={9},
   date={1977},
   number={3},
   pages={342--360},
   issn={0022-314X},
   review={\MR{0447188}},
 }
 
\bib{Platonov-Rapinchuk:algebraic-groups}{book}{
   author={Platonov, Vladimir},
   author={Rapinchuk, Andrei},
   title={Algebraic groups and number theory},
   series={Pure and Applied Mathematics},
   volume={139},
   note={Translated from the 1991 Russian original by Rachel Rowen},
   publisher={Academic Press, Inc., Boston, MA},
   date={1994},
   pages={xii+614},
   isbn={0-12-558180-7},
   review={\MR{1278263}},
 }

 \bib{Platonov-Tavgen:grothendieck}{article}{
   author={Platonov, V. P.},
   author={Tavgen\cprime, O. I.},
   title={Grothendieck's problem on profinite completions and
   representations of groups},
   journal={$K$-Theory},
   volume={4},
   date={1990},
   number={1},
   pages={89--101},
   issn={0920-3036},
   review={\MR{1076526}},
 }
 
 \bib{Prasad-Rapinchuk:prescribed}{article}{
   author={Prasad, Gopal},
   author={Rapinchuk, Andrei S.},
   title={On the existence of isotropic forms of semi-simple algebraic
   groups over number fields with prescribed local behavior},
   journal={Adv. Math.},
   volume={207},
   date={2006},
   number={2},
   pages={646--660},
   issn={0001-8708},
   review={\MR{2271021}},
}
 
 \bib{Serre:galois-cohomology}{book}{
   author={Serre, Jean-Pierre},
   title={Galois cohomology},
   series={Springer Monographs in Mathematics},
   edition={Corrected reprint of the 1997 English edition},
   note={Translated from the French by Patrick Ion and revised by the
   author},
   publisher={Springer-Verlag, Berlin},
   date={2002},
   pages={x+210},
   isbn={3-540-42192-0},
   review={\MR{1867431}},
}

\bib{Serre:representations}{book}{
   author={Serre, Jean-Pierre},
   title={Linear representations of finite groups},
   series={Graduate Texts in Mathematics},
   volume={Vol. 42},
   edition={French edition},
   publisher={Springer-Verlag, New York-Heidelberg},
   date={1977},
   pages={x+170},
   isbn={0-387-90190-6},
   review={\MR{0450380}},
 }
 
\bib{Spitler:profinite}{thesis}{
author = {Ryan F. Spitler},
title = {Profinite Completions and Representations of Finitely Generated Groups},
year = {2019},
note = {PhD thesis},
organization = {Purdue University},
}

\bib{Sun:3-manifold-groups}{article}{
   author={Sun, Hongbin},
   title={All finitely generated 3-manifold groups are Grothendieck rigid},
   journal={Groups Geom. Dyn.},
   volume={17},
   date={2023},
   number={2},
   pages={385--402},
   issn={1661-7207},
   review={\MR{4584669}},
 }
 
\bib{Tits:classification-algebraic}{article}{
   author={Tits, J.},
   title={Classification of Algebraic Semisimple Groups},
   journal={Proceedings of Symposia in Pure Mathematics},
   volume={9},
   date={1966},
   pages={33--62},
   issn={0082-0717},
}
 
\end{biblist}
\end{bibdiv}
\end{document}